\documentclass[11pt]{article}
\usepackage[utf8]{inputenc}
\usepackage{amsmath}
\usepackage{amsfonts}
\usepackage{amssymb}
\usepackage{amsthm}
\usepackage{amsrefs}
\usepackage[all]{xy}
\usepackage{enumerate}

\theoremstyle{plain}
  \newtheorem{thm}{Theorem}[section]
  \newtheorem{lem}[thm]{Lemma}
  
  \newtheorem{cor}[thm]{Corollary} 
  \newtheorem*{thm*}{Theorem}
  \newtheorem*{cor*}{Corollary} 
\theoremstyle{definition}
  \newtheorem{defn}[thm]{Definition}
  \newtheorem{rmk}[thm]{Remark}
  \newtheorem{ex}[thm]{Example}

\theoremstyle{plain}

\DeclareMathOperator{\im}{im}
  
\numberwithin{equation}{section}
\allowdisplaybreaks

\begin{document}
\title{$A_\infty$-Minimal Model on Differential Graded Algebras}
\author{Jiawei Zhou}
\date{\today}
\maketitle

\begin{abstract}
The rational homotopy type of a differential graded algebra (DGA) can be represented by a family of tensors on its cohomology, which constitute an $A_\infty$-minimal model of this DGA. When only the cohomology is needed to determine the rational homotopy type, then the DGA is called formal. By a theorem of Miller, a compact $k$-connected manifold is formal if its dimension is not greater than $4k+2$. We expand this theorem and a result of Crowley-Nordstr\"{o}m to prove that if the dimension of a compact $k$-connected manifold $N\leq (l+1)k+2$, then its de Rham complex has an $A_\infty$-minimal model with $m_p=0$ for all $p\geq l$. Separately, for an odd-dimensional sphere bundle over a formal manifold, we prove that its de Rham complex has an $A_\infty$-minimal model with only $m_2$ and $m_3$ non-trivial. In the special case of a circle bundle over a formal symplectic manifold satisfying the hard Lefschetz property, we give a necessary condition for formality which becomes sufficient when the base symplectic manifold is of dimension six or less. 

\end{abstract}

\tableofcontents

\section{Introduction}

In rational homotopy theory, a differential graded algebra (DGA) is formal if it is quasi-isomorphic to its cohomology. Thus, we can study its homotopy type by its cohomology. If the de Rham complex of a manifold is a formal DGA, the manifold itself is called formal. A natural question is, what conditions or characteristics ensure that a manifold is formal?

For a $k$-connected compact manifold $M$, a theorem of Miller \cite{miller} states that $M$ is formal if its dimension $N\leq 4k+2$. When the dimension $N > 4k+2$, examples constructed by Dranishnikov-Rudyak \cite{Dranishnikov et al} and Fern\'andez-Mu\~noz \cite{Fernandez et al} showed that $k$-connected compact manifolds may no longer be formal without imposing further conditions. For instance, Cavalcanti \cite{caval4} considered imposing conditions on the $(k+1)$-Betti number.  He showed that the manifold is formal when $b_{k+1}=1$ and the dimension $N\leq 4k+4$. For $b_{k+1}=2,3$, Cavalcanti \cite{caval4} had to introduce even more restrictive conditions to prove formality. However, not much can be said about formality for compact $k$-connected manifold when $b_{k+1}$ is larger.

Since higher dimensional $k$-connected compact manifold is non-formal in general, we can consider relaxing the formal condition and instead ask what properties in addition to cohomology will determines a manifold's rational homotopy type. Crowley and Nordstr\"om \cite{cn} showed that when $N\leq 5k+2$, the rational homotopy type of a $k$-connected manifold $M$ can be determined by its cohomology $H^*(M)$ together with a $4$-tensor acting on a subspace of $\left(H^*(M)\right)^{\otimes 4}$, which they called the Bianchi-Massey tensor. This $4$-tensor represents the additional data beyond cohomology needed to determine rational homotopy up to dimension $N\leq 5k+2$. It is reasonable to expect that as the dimension increases, even more data is needed to describe rational homotopy. As we will show, the additional required data for a $k$-connected manifold of arbitrary dimension can be succinctly represented in terms of an $A_\infty$-algebraic structure on $H^*(M)$. Such a structure is called an $A_\infty$-minimal model of $\Omega^*(M)$. 
This leads to our first result:

\begin{thm}\label{connected-intro}
Suppose $M$ is an $N$-dimensional $k$-connected compact manifold. If $l\geq 3$ such that $N\leq (l+1)k+2$, then $\Omega^*(M)$ has an $A_\infty$-minimal model with $m_p=0$ for $p\geq l$. 
\end{thm}

When $l=3$, the $A_\infty$ structure is equivalent to the DGA structure and the statement is just the theorem of Miller \cite{miller}. When $l=4$, it becomes the theorem of Crowley and Nordstr\"om \cite{cn}. For higher $l$, it implies that the homotopy type of $\Omega^*(M)$ is determined by its cohomology together with operations $m_3,\ldots,m_{l-1}$.

Another question is whether we can construct formal manifolds from a given formal manifold. One common way of obtaining new manifolds is constructing a fiber bundle. In this paper, we consider a relatively simple fiber bundle where the fiber is an odd dimensional sphere. A motivation for studying this type of bundle is its relation to the $A_3$-algebra on a symplectic manifold $(M,\omega)$, constructed by Tsai, Tseng and Yau \cite{tty}. By the work of Tanaka and Tseng \cite{tt}, Tsai-Tseng-Yau's $A_3$-algebra is quasi-isomorphic to the de Rham cochain complex of an $S^{2p+1}$ bundle over the symplectic manifold with Euler class $\omega^p$.

As we shall see, the total sphere bundle space can be non-formal even when the base is formal. For example, it is non-formal for the circle bundle over the torus $T^{2n}$ (See Example \ref{ex-torus}). On the other hand, circle bundles over projective spaces or over Euclidean spaces are both formal. To quantify the non-formality of a sphere bundle over a base that is formal, we can consider giving the cohomology of the circle bundle an $A_3$-algebraic structure. It turns out that this $A_3$ structure is sufficient to completely determine the rational homotopy of the odd-sphere bundle over a formal base.
 
\begin{thm}\label{extension-intro}
Let $M$ be a formal manifold, $\omega\in \Omega^*(M)$ be an even-dimensional integral differential form, and $X$ be the sphere bundle over $M$ with Euler class $\omega$. Then $X$ is formal if $\omega$ is exact. When $\omega$ is non-exact, $\Omega^*(X)$ has an $A_\infty$-minimal model with $m_p=0$ for all $p$ except for $p=2$ or $3$.
\end{thm}

Biswas, Fern\'andes, Mu\~noz and Tralle \cite{biswas et al} obtained a similar statement when the base symplectic manifold $M$ satisfies the hard Lefschetz property instead of being formal. The hard Lefchetz property on a symplectic manifold is analogous to the $dd^c$-lemma on a complex manifold. They both can be viewed as special cases of the $dd^{\mathcal{J}}$-lemma of generalized complex geometry \cite{caval4}\cite{merkulov}. In the complex manifolds case, there is the well known result of Deligne, Griffiths, Morgan and Sullivan \cite{deli} that all complex manifolds satisfying the $dd^c$-lemma are formal. By analogy, this may suggest that the hard Lefschetz property would relate to the formality on symplectic manifolds. In fact, it was conjectured by Babenko and Taimanov \cite{bt} in 1998 that a simply-connected compact symplectic manifold is formal if and only if it satisfies the hard Lefschetz property. Both directions of the statement are now known to be false. Gompf \cite{gompf} constructed a simply-conected 6-manifold which does not satisfy the hard Lefschetz property. This example is formal by Miller's criterion \cite{miller} that all simply-connected compact 6-manifolds are formal. The other direction was studied in \cite{irtu}\cite{lo}, and was further clarified by Cavalcanti \cite{caval} (see also \cite{caval3}) who gave a simply connected non-compact symplectic manifold with the hard Lefschetz property, but is not formal.

Although the hard Lefschetz property has no relation with the formality of a symplectic manifold, it is nevertheless useful for our consideration of the formality of circle bundles. In Section 4.2, we focus on a special class of circle bundles where the base symplectic manifold is formal and satisfies the hard Lefschetz property. A circle bundle over the symplectic manifold with Euler class $\omega$ is called a Boothby-Wang fibration. For such fibrations, we have found a necessary condition for formality. When the dimension of the base space $\dim M \leq 6$, this condition is also sufficient.

\begin{thm}\label{condition equivalent to formal geometircally-intro}
Let $(M^{2n},\omega)$ be a formal symplectic manifold satisfying the hard Lefschetz property. Suppose $\omega$ is integral, then there exists a circle bundle $X$ over $M$ whose Euler class is $\omega$. In the following statements, (1) implies (2), and (2) implies (3). Moreover, when $\dim M\leq 6$, (3) also implies (1).
\begin{enumerate}[(1)]
\item X is formal.
\item All generalized Massey products of $\Omega^*(X)$ vanish.
\item For arbitrary $x_1,\ldots,x_k\in PH^r(M)$, $y_1,\ldots y_k\in PH^s(M)$ satisfying $r+s\leq n+1$ and
$$
\sum_{i=1}^{k}x_iy_i=\omega\alpha
$$
for some $\alpha\in H^{r+s-2}(M)$, and for arbitrary $z\in PH^{n+1-s}(M)$, we have
\begin{align}\label{necessary equation for formal-intro}
\sum_{i=1}^{k}x_iL^{-1}(y_iz)=\sum_{i=1}^{k}L^{-1}(x_iy_i)z.
\end{align}

Here, $PH^r(M)=\ker(\omega^{n-r+1}:H^r(M)\to H^{2n-r+2})$ for $r\leq n$. Given an element $a\in PH^k(M)$, we define operator $L^{-1}$ as 
$$
L^{-1}a=0,\quad \text{and}\quad L^{-1}(\omega^ja)=\omega^{j-1}a \text{ for } 1\leq j\leq n-k.
$$
\end{enumerate} 
\end{thm}

Simpler than the vanishing of the Massey products, equation (\ref{necessary equation for formal-intro}) is also a necessary  condition for formality and can be checked straightforwardly. It gives a quick way to determine non-formality for certain classes of manifolds.  For example, we have the following statement:

\begin{cor}
When $M$ is a compact K\"ahler manifold and $[\omega]$ is reducible, $X$ is non-formal.
\end{cor}

This paper is organized as follows. In Section 2, we review the definitions and some basic properties of DGA and $A_\infty$-algebra. Section 3 consists of the proof of Theorem \ref{connected-intro}. In Section 4, we prove Theorem \ref{extension-intro}, then consider the special case of circle bundles. We motivate the statement of Theorem \ref{condition equivalent to formal geometircally-intro} and provide the proof.  A main ingredient of our proof including a detailed analysis of the required $A_\infty$-morphism is provided in the Appendix.

\textbf{Acknowledgements.} The author would like to thank his Ph.D. advisor, Li-Sheng Tseng, for his patient guidance and valuable advice. The author also appreciates his postdoctoral mentor Si Li, for suggesting many ideas and related questions. The author also thanks Vladimir Baranovsky, Xiaojun Chen and Matthew Gibson for their helpful suggestions. This work is supported by National Key Research and Development Program of China  (NO. 2020YFA0713000).

\section{Preliminaries}

In this section, we recall some basic definitions and properties of differential graded algebras and $A_\infty$-algebras (c.f. \cite{compgeo}\cite{keller}).

\subsection{Differential graded algebras}
\begin{defn}
A (commutative) \textbf{differential graded algebra} (\textbf{DGA}) over a field $k$ is a graded $k$-algebra $A=\bigoplus_{i\geqslant 0}A^i$ together with a $k$-linear map $d:A\to A$ such that
\begin{enumerate}[i)]
\setlength{\itemsep}{-5pt}
\item $k\subset A^0$;
\item The multiplication is graded commutative: For $x\in A^i, y\in A^j$, we have $x\cdot y=(-1)^{ij}y\cdot x;$
\item The Leibniz product rule holds: $d(x\cdot y)=dx\cdot y+(-1)^i x\cdot dy$;
\item $d^2=0$.
\end{enumerate}
\end{defn}

\begin{ex}
Let $M$ be a manifold. Its differential forms form a DGA $(\Omega^*(M),d,\wedge)$, where $d$ is the differential operator and $\wedge$ is the wedge product of differential forms. 
\end{ex}

\begin{ex}
Given a DGA $(A,d)$, its cohomology $(H^*(A),d)$ is also a DGA, with $d=0$. The multiplication on $H^*(A)$ is naturally induced by the multiplication on $A$.
\end{ex}

\begin{defn}
Let ($A,d_A$) and ($B,d_B$) be two DGAs. A \textbf{DGA-homomorphism} is a $k$-linear map $f:A\rightarrow B$ such that
\begin{enumerate}[i)]
\setlength{\itemsep}{-5pt}
\item $f(A^i)\subset B^i$;
\item $f(x\cdot y)=f(x)\cdot f(y)$;
\item $d_B\circ f=f\circ d_A$:
$$
\xymatrix
{ \ldots \ar[r]^{d_A} & A^k \ar[r]^{d_A} \ar[d]^f & A^{k+1} \ar[d]^f \ar[r]^{d_A} & \ldots \\
\ldots \ar[r]^{d_B} & B^k \ar[r]^{d_B} & B^{k+1} \ar[r]^{d_B} & \ldots }$$
\end{enumerate}
Naturally, $f$ induces a homomorphism:
$$f^*:H^*(A,d_A)\rightarrow H^*(B,d_B).$$
$f$ is called a \textbf{DGA-quasi-isomorphism} if $f^*$ is an isomorphism.
\end{defn}

\begin{defn}
Two DGAs ($A,d_A$) and ($B,d_B$) are \textbf{equivalent} if there exists a sequence of DGA-quasi-isomorphisms:
$$
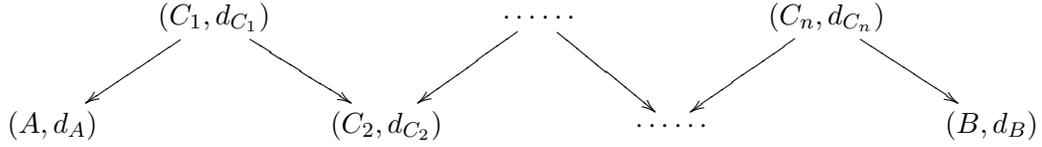
\xymatrixcolsep{1.4pc}\xymatrix{
& (C_1,d_{C_1}) \ar[dl] \ar[dr] & &  \ldots\ldots \ar[dl] \ar[dr] & &(C_n,d_{C_n}) \ar[dl] \ar[dr] \\
(A,d_A) &  & (C_2,d_{C_2}) &  & \ldots\ldots &  & (B,d_B)
}
$$
\end{defn}

\begin{defn}
A DGA ($A,d_A$) is called \textbf{formal} if ($A,d_A$) is equivalent to a DGA ($B,d_B$) with $d_B=0$. Identically, ($A,d_A$) is equivalent to ($H^*(A),d=0$) if and only if ($A,d_A$) is formal.

We say a manifold $M$ is \textbf{formal} if its de Rham complex ($\Omega ^*(M),d,\wedge$) is a formal DGA.
\end{defn}

\begin{rmk}
We usually omit the multiplication sign of $DGA$. For example, in the DGA $\Omega^*$ of differential forms, we write $\alpha\wedge\beta$ as $\alpha\beta$, and write $\alpha\wedge\alpha$ as $\alpha^2$ simply. So $\alpha\beta=(-1)^{|\alpha||\beta|}\beta\alpha$, where $|\alpha|,|\beta|$ are the degrees of $\alpha$ and $\beta$ respectively.
\end{rmk}

\subsection{$A_\infty$-algebra}
\begin{defn}
Let $k$ be a field. An \textbf{$A_\infty$-algebra over $k$} is a $\mathbb{Z}$-graded vector space 
$A=\bigoplus_{i\in \mathbb{Z}} A^i$ endowed with graded $k$-linear maps
$$m_p:A^{\otimes p}\to A,\ p\geq1$$
of degree $2-p$ satisfying
\begin{equation}\label{A-infty operation}
\sum_{r+s+t=p}(-1)^{r+st}m_{r+t+1}(\textbf{1}^{\otimes r}\otimes m_s\otimes\textbf{1}^{\otimes t})=0.
\end{equation}
\end{defn}

In particular, when $p=1$, we have $$m_1m_1=0.$$
When $p=2$, we have $$m_1m_2=m_2(m_1\otimes\textbf{1}+\textbf{1}\otimes m_1).$$
If $m_3=0$, $m_2$ is associative. Every DGA is an $A_\infty$-algebra, where $m_1$ is the differential $d$, $m_2$ is the  multiplication, and $m_p=0$ for all $p\geq3$.

\begin{defn}
A \textbf{morphism of $A_\infty$-algebra} $f:(A,m^A)\to (B,m^B)$ is a family of graded maps $f_p:A^{\otimes p}\to B$ of degree $1-p$ such that
\begin{equation}\label{A-infty morphism}
\sum(-1)^{r+st}f_{r+t+1}(\textbf{1}^{\otimes r}\otimes m^A_s\otimes\textbf{1}^{\otimes t})=\sum(-1)^sm^B_r(f_{i_1}\otimes f_{i_2}\otimes\ldots\otimes f_{i_r}).
\end{equation}
where the sum on left-hand side runs over all decompositions $p=r+s+t$, and the sum on the right-hand side runs over all $1\leq r\leq p$ and all decompositions $p=i_1+i_2+\ldots+i_r$. The sign on the right side is given by
$$s=(r-1)(i_1-1)+(r-2)(i_2-1)+\ldots+2(i_{r-2}-1)+(i_{r-1}-1).$$
\end{defn}

Specifically, when $p=1$, we have $$m_1f_1=f_1m_1.$$
$f_1$ also induces a morphism $f_1^*:H^*(A)\to H^*(B)$. The morphism $f$ is called a \textbf{quasi-isomorphism} if $f_1^*$ is an isomorphism.

Alternatively, we can describe an $A_\infty$-algebra by its suspension. In this convention, no negative signs appear. Let $SA$ be a graded vector space such that $(SA)^i=A^{i+1}$. Set $s:A\to SA$, the canonical map of degree -1. Then we can define a family of maps $b_n:(SA)^{\otimes n}\to SA$ corresponding to $m_n$ by
$$
\xymatrixcolsep{3.5pc}\xymatrix{
(SA)^{\otimes n} \ar[r]^{b_n} & SA \\
A \ar[r]^{m_n} \ar[u]^{s^{\otimes n}} & A \ar[u]^s
}
$$
For example, $b_1(s\alpha)=sm_1(\alpha)$ and $b_2(s\alpha,s\beta)=(-1)^{|\alpha|}sm_2(\alpha,\beta)$, where $|\alpha|$ is the degree of $\alpha$.

All $b_n$ are of degree 1, and the  equations of (\ref{A-infty operation}) become equivalent to
$$
\sum_{r+s+t=p}b_{r+t+1}(\textbf{1}^{\otimes r}\otimes b_s\otimes\textbf{1}^{\otimes t})=0.
$$

Similarly, given an $A_\infty$-morphism $f:A\to B$, we can define $(sf)_n:(SA)^{\otimes n}\to SB$ corresponding to $f_n$ by
$$
\xymatrixcolsep{3.5pc}\xymatrix{
(SA)^{\otimes n} \ar[r]^{(sf)_n} & SB \\
A \ar[r]^{f_n} \ar[u]^{s^{\otimes n}} & B \ar[u]^s
}
$$
All $f_n$ are of degree 0, and the equations of (\ref{A-infty morphism}) become equivalent to
$$
\sum (sf)_{r+t+1}(\textbf{1}^{\otimes r}\otimes b^A_s\otimes\textbf{1}^{\otimes r})=\sum b^B_r\big( (sf)_{i_1}\otimes (sf)_{i_2}\otimes\ldots\otimes (sf)_{i_r} \big).
$$
Again, the left-hand side sum runs over all decompositions $p=r+s+t$, and the right-hand side sum runs over all $1\leq r\leq p$ and all decompositions $p=i_1+i_2+\ldots+i_r$.

An $A_\infty$-algebra is quasi-isomorphic to its cohomology equipped with an appropriate $A_\infty$-algebraic structure.

\begin{thm}[Kadeishvili \cite{kadei}, see also \cite{keller}]\label{kadei}
If $(A,m^A)$ is an $A_\infty$-algebra, then $H^*(A)$ has an $A_\infty$-algebraic structure $m$ such that
\begin{enumerate}[i)]
\item $m_1=0$ and $m_2$ is induced by $m_2^A$;
\item There is an quasi-isomorphism of $A_\infty$-algebras $H^*(A)\to A$.
\end{enumerate}
This structure is unique up to isomorphisms of $A_\infty$-algebras.
\end{thm}

\begin{defn}
$(H^*(A),m)$ given above is called an \textbf{$A_\infty$-minimal model} for $(A,m^A)$. We say $(A,m^A)$ is \textbf{formal} if we can choose all $m_p$ to be 0 for $p\geq3$ on its $A_\infty$-minimal model.
\end{defn}

We will give an explicit construction of an $A_\infty$-minimal model, which will be used in later sections. For convenience, we use following notation:

Let $f:(H^*(A),m)\to (A,m^A)$ be an $A_\infty$-morphism. For $p\geq 3$, set
\begin{equation}\label{F_p}
\begin{split}
F_p  &= \sum_{\substack{i_1+\ldots+i_r=p \\ 2\leq r\leq p}}(-1)^{(r-1)(i_1-1)+(r-2)(i_2-1)+\ldots+1\cdot(i_{r-1}-1)}m^A_r(f_{i_1}\otimes\ldots\otimes f_{i_r})\\
& \quad -\sum_{\substack{r+s+t=p \\ 2\leq s\leq p-1}}(-1)^{r+st}f_{r+t+1}(\textbf{1}^{\otimes r}\otimes m_s\otimes \textbf{1}^{\otimes t}).
\end{split}
\end{equation}

$F_p$ is defined by $f_1,\ldots,f_{p-1},m_1,\ldots,m_{p-1},$ and $m^A_1,\ldots,m^A_p$. It then determines $m_p$ and $f_p$. Since $m_1=0$, $m_p$ and $f_p$ need to satisfy $f_1m_p-m^A_1f_p=F_p$. So $m_p$ must be $[F_p]$. Therefore, given an $A_\infty$-algebra $(A,m^A)$, if we want to give $H^*(A)$ an $A_\infty$-algebraic structure $m$ together with a quasi-isomorphism $f:(H^*(A),m)\to (A,m^A)$, we can define $m_p$ and $f_p$ inductively by calculating $F_p$.

Similarly, on the suspension, we set
\begin{equation}\label{sF_p}
(sF)_p=\sum_{\substack{i_1+\ldots+i_r=p \\ 2\leq r\leq p}}b^A_r\big( (sf)_{i_1}\otimes\ldots\otimes (sf)_{i_r} \big)-\sum_{\substack{r+s+t=p \\ 2\leq s\leq p-1}}(sf)_{r+t+1}(\textbf{1}^{\otimes r}\otimes b_s\otimes \textbf{1}^{\otimes t}).
\end{equation}
As $b_1=0$, $b_p$ and $(sf)_p$ need to satisfy $(sf)_1b_p-b^A_1(sf)_p=(sF)_p$.

For an $A_\infty$-morphism $g:A\to B$, we can define $G_p$ and $sG_p$ in a similar way.

\begin{proof}[Proof of Theorem \ref{kadei}]
Let $A^E$ denote the subspace of all $m^A_1$-exact forms in $A$. By the splitting Lemma, we can decompose the space of all $m^A_1$-closed forms as $A^E\oplus A^C$ for $A^C\simeq H^*(A)$, and can decompose $A$ as $A^E\oplus A^C\oplus A^\perp$ for $A^\perp\simeq A/(A^E\oplus A^C)$. Then for each $\alpha\in A^E$, there is a unique $\beta\in A^\perp$ such that $\alpha=m^A_1\beta$. So we can define a map $Q:A^E\to A^\perp$ such that $Q\alpha=\beta$. Then $Qm^A_1=\textbf{1}_{A^E}$ and $m^A_1Q=\textbf{1}_{A^\perp}$, where \textbf{1} is the identity map.

In particular, if $M$ is a Riemannian manifold and $A=\Omega^*(M)$, we can set $A^C$ to be the space of harmonic forms, $A^\perp=\im d^*$, and $Q=d^*$.

Now we define an $A_\infty$-algebra structure $m$ on $H^*(A)$ and a quasi-isomorphism $f:(H^*(A),m)\to (A,m^A)$. For each $[\alpha]\in H^*(A)$, there exists a unique $\alpha_0\in A^C$ representing the cohomology class $[\alpha]$. Set $f_1([\alpha])=\alpha_0$. For $f_2$, set $f_2=Q(f_1m_2-m^A_2(f_1\otimes f_1))$. To define higher $m_p$ and $f_p$, suppose $m_1,\ldots,m_{p-1}$ on $H^*(A)$ and $f_1,\ldots,f_{p-1}$ have been defined, then $m_p=[F_p]$ and $f_p$ needs to satisfy $f_1m_p-m^A_1f_p=F_p$. Set $f_p=Q(f_1m_p-F_p)$. By induction, we can construct an $A_\infty$-minimal model of $A$.
\end{proof}

By the theorem below, a DGA satisfying the definition of formal in the DGA sense is equivalent to satisfying the definition of formal as an $A_\infty$-algebra. So in this context, we will simply say that this DGA is formal.

\begin{thm}[see \cite{keller}]
If A is a DGA, it is formal as a DGA if and only if it is formal as an $A_\infty$-algebra.
\end{thm}

\begin{defn}
An $A_\infty$-algebra $(A,m^A)$ is called \textbf{strictly unital} if there exists some $1_A\in A$ such that $m_1 1_A=0$, $m_2(1_A,\alpha)=m_2(\alpha,1_A)=\alpha$, and $m_p(\alpha_1,\ldots,\alpha_p)=0$ for $p\geq 3$ when some $\alpha_j=1_A$.

An $A_\infty$-morphism $f:(A,m^A)\to(B,m^B)$ is called \textbf{strictly unital} if $A$ and $B$ are strictly unital, $f_1(1_A)=1_B$, and $f_p(\alpha_1,\ldots,\alpha_p)=0$ for $p\geq 2$ when some $\alpha_j=1_A$.
\end{defn}

\begin{thm}\label{strictly-unital}
When $(A,m^A)$ is a strictly unital $A_\infty$-algebra and $1_A$ is non-exact, we can construct a strictly unital minimal model $(H^*(A),m)$, and a strictly unital quasi-isomorphism $f:(H^*(A),m)\to (A,m^A)$.
\end{thm}

\begin{proof}
We construct the minimal model by the proof of Theorem \ref{kadei}. Let $1_H\in H^0(A)$ be the cohomology class of $1_A$. Choose $A^C$ such that it contains $1_A$. Then $f_1(1_H)=1_A$. In particular, when $A^{-1}=0$, $1_A$ is the only representative of $1_H$, so it must be in $A^C$. Also $1_A$ cannot be exact in this case.

Since $m_2$ is induced by $m_2^A$, for each $[\alpha]\in H^*(A)$ we have $m_2(1_H,[\alpha])=[m_2(1_A,\alpha)]=[\alpha]$. Similarly $m_2([\alpha],1_H)=[\alpha]$.

To avoid keeping track of plus and minus signs, we will finish the proof on suspension $SA$. Then the equations of $m_2$ become
$$
b_2(s1_H,[s\alpha])=[s\alpha],\quad b_2([s\alpha],s1_H)=(-1)^{|\alpha|}[s\alpha].
$$

$(sf)_2$ is defined by
$$
(sf)_2=Q\circ\Big( b_2\big((sf)_1\otimes(sf)_1\big)-(sf)_1b_2 \Big).
$$
So for arbitrary $sx\in SH^*(A)$, we have
$$
b_2\big((sf)_1\otimes(sf)_1\big)(s1_H,sx)=b_2 \big(s1_A,(sf)_1(sx) \big)=(sf)_1(sx)=(sf)_1b_2(s1_H,sx)
$$
and
$$
b_2\big((sf)_1\otimes(sf)_1\big)(sx,s1_H)=b_2 \big( (sf)_1(sx),s1_A \big)=(-1)^{|x|}(sf)_1(sx)=(sf)_1b_2(sx,s1_H).
$$
This implies $(sf)_2(s1_H,sx)=(sf)_2(sx,s1_H)=sx$.

For $p\geq3$, we prove the statement inductively. Suppose $b_1,\ldots,b_{p-1}$ and $(sf)_1,\ldots,(sf)_{p-1}$ have been constructed and satisfy the condition for strictly unital. $(sF)_p$ is defined as
\begin{align*}
(sF)_p(sx_1,\ldots,sx_p) &= \sum_{\substack{i_1+\ldots+i_r=p \\ 2\leq r\leq p}}b^A_r\big( (sf)_{i_1}\otimes\ldots\otimes (sf)_{i_r} \big)(sx_1,\ldots,sx_p) \\
& \quad -\sum_{\substack{r+s+t=p \\ 2\leq s\leq p-1}}(sf)_{r+t+1}(\textbf{1}^{\otimes r}\otimes b_s\otimes \textbf{1}^{\otimes t})(sx_1,\ldots,sx_p).
\end{align*}
Assume $x_j=1_H$ for some $1\leq j\leq p$. Consider the first term on the right-hand side. For each summand $b^A_r\big( (sf)_{i_1}\otimes\ldots\otimes (sf)_{i_r} \big)(sx_1,\ldots,sx_p)$, there exists some $k$ such that
$$
i_1+\ldots+i_{k-1}<j\leq i_1+\ldots+i_k.
$$
When $i_k>1$, $(sf)_{i_k}$ is acting on the tensor of some elements including $1_H$. So the image is $0$. The non-trivial summand must satisfy $i_k=1$. In this case if $1<j<p$, we have $i_1+\ldots+i_{k-1}>0$ and $i_1+\ldots+i_k=j<p$. Thus, $r\geq3$. Since $b^A_r$ acting on the tensor of some elements including $1_A$ is 0. All summands are trivial. For the same reason, if $j=1$, the only non-trivial summand is
$$
b^A_2 \big( (sf)_1\otimes (sf)_{p-1} \big)(1_H,sx_2,\ldots,sx_p)=(sf)_{p-1}(sx_2,\ldots,sx_p).
$$
If $j=p$, the only non-trivial summand is
$$
b^A_2 \big( (sf)_{p-1}\otimes (sf)_1 \big)(sx_1,\ldots,sx_{p-1},1_H)=(-1)^{|sx_1|+\ldots+|sx_{p-1}|+1}(sf)_{p-1}(sx_1,\ldots,sx_{p-1}).
$$

Now we turn to the second term: $\sum (sf)_{r+t+1}(\textbf{1}^{\otimes r}\otimes b_s\otimes \textbf{1}^{\otimes t})(sx_1,\ldots,sx_p)$. In each summand, $s\leq p-1$, so $r+t+1\geq 2$. If $j\leq r$ or $j>r+s$, $(sf)_{r+t+1}$ is acting on the tensor of some elements including $1_H$. Hence, the image is $0$. When $r<j\leq r+s$ and $s\geq 3$, the summand is also trivial because $b_s$ is acting on the tensor of some elements including $1_H$. So the only non-trivial summand satisfies $j=r+1$ or $j=r+2$, and $s=2$. When $1<j<p$, the second term becomes
\begin{align*}
&\quad (sf)_{p-1}(\textbf{1}^{\otimes(j-1)}\otimes b_2 \otimes \textbf{1}^{\otimes(p-j-1)}+\textbf{1}^{\otimes(j-2)}\otimes b_2 \otimes \textbf{1}^{\otimes(p-j)})(sx_1,\ldots,sx_p) \\
&= \Big( (-1)^{|sx_1|+\ldots+|sx_{j-1}|}+(-1)^{|sx_1|+\ldots+|sx_{j-2}|}(-1)^{|x_{j-1}|} \Big)(sf)_{p-1}(sx_1,\ldots,sx_{j-1},sx_{j+1},\ldots,sx_p) \\
&= 0.
\end{align*}
When $j=1$, the second term becomes
$$
(sf)_{p-1}(b_2 \otimes \textbf{1}^{\otimes(p-2)})(1_H,sx_2,\ldots,sx_p)=(sf)_{p-1}(sx_2,\ldots,sx_p).
$$
When $j=p$, the second term becomes
\begin{align*}
& \quad (sf)_{p-1}(\textbf{1}^{\otimes(p-2)}\otimes b_2)(sx_1,\ldots,sx_{p-1},1_H) \\
&= (-1)^{|sx_1|+\ldots+|sx_{p-2}|}(-1)^{|sx_{p-1}|}(sf)_{p-1}(sx_1,\ldots,sx_{p-1}) \\
&= (-1)^{|sx_1|+\ldots+|sx_{p-1}|+1}(sf)_{p-1}(sx_1,\ldots,sx_{p-1}).
\end{align*}

In each case, the first term is equal to the second term. Therefore, we always have $sF_p(x_1,\ldots,x_p)=0$ if some $x_j=0$. Then $b_p=[sF_p]$ and $sf_p=sQ \big( (sf)_1b_p-sF_p \big)$ are both 0. So $(H^*(A),m)$ and $f$ are strictly unital.
\end{proof}

\section{Minimal model of $k$-connected compact manifold}
We here recall a result of Miller for $k$-connected compact manifolds. A manifold $M$ is called $k$-connected if it is path-connected and its homotopy group $\pi_r(M)=0$ for $1\leq r\leq k$. Our goal in this section is to generalize Miller's result.

\begin{thm}[Miller \cite{miller}]
Let $M$ be an $N$-dimensional $k$-connected compact manifold. If $N\leq 4k+2$, then M is formal.
\end{thm}

Consider a strictly unital $A_\infty$-algebra $(A,m^A)$ with the following properties.
\begin{enumerate}
\item $A^i=0$ for $i<0$ or $i>N$.
\item $H^i(A)$ is finite dimensional. $\dim H^0(A)= \dim H^N(A)=1$.
\item Take a basis $\left\{ x^{(i)}_1,\ldots,x^{(i)}_{\beta_i}\right\}$ for each $H^i(A)$, where $\beta_i$ is the dimension of $H^i(A)$. Let $\mu$ denote the only generator $x^{(N)}_1$ of $H^N(A)$. Then there exists a "dual" basis $\left\{ y^{(N-i)}_1,\ldots,y^{(N-i)}_{\beta_i}\right\}\in H^{N-i}(A)$ such that on the suspension
$$
b_2(sy^{(N-i)}_v,sx^{(i)}_u)=\delta_{uv}s\mu=
\begin{cases}
s\mu, & \text{if } u=v, \\
0, & \text{if } u\neq v.
\end{cases}
$$
\item For arbitrary $s\alpha_1,\ldots,s\alpha_p\in SA$, let $n=|s\alpha_1|+\ldots+|s\alpha_p|$ be the total degree. Let $\sigma$ denote a permutation such that
$$
\sigma(s\alpha_1,\ldots,s\alpha_p)=(-1)^{|s\alpha_1|(n-|s\alpha_1|)} (s\alpha_2,\ldots,s\alpha_p,s\alpha_1).
$$
Then the cyclic sum
$$
b_p(\mathbf{1}+\sigma+\sigma^2+\ldots+\sigma^{p-1})=0
$$
for all $p\geq 2$.
\end{enumerate}

In particular, if $M$ is a compact orientable manifold, $\Omega^*(M)$ satisfies the conditions above. Condition 3 follows from Poincar\'e duality. Condition 4 follows from a simple calculation below. For $\alpha,\beta\in \Omega^*(M)$,
\begin{align*}
& \quad b_2(s\alpha,s\beta)+b_2\sigma(s\alpha,s\beta) \\
&= (-1)^{|\alpha|}sm_2(\alpha,\beta)+(-1)^{|s\alpha|(|s\alpha|+|s\beta|-|s\alpha|)}b_2(s\beta,s\alpha) \\
&= (-1)^{|\alpha|}(-1)^{|\alpha||\beta|}sm_2(\beta,\alpha)+(-1)^{(|\alpha|+1)(|\beta|+1)}(-1)^{|\beta|}b_2(s\otimes s)(\beta,\alpha) \\
&= (-1)^{(|\alpha|+1)|\beta|}sm_2(\beta,\alpha)+(-1)^{(|\alpha|+1)|\beta|+1}sm_2(\beta,\alpha) \\
&= 0.
\end{align*}
And $b_p=0$ for $p\geq 3$.

We claim that $\sigma^p=\mathbf{1}$ when acting on $(SA)^{\otimes p}$. By definition,
$$
\sigma^p (s\alpha_1,\ldots,s\alpha_p)=\prod_{j=1}^p(-1)^{|s\alpha_j|(n-|s\alpha_j|)}\cdot (s\alpha_1,\ldots,s\alpha_p).
$$
When $n$ is odd, $(-1)^{|s\alpha_j|(n-|s\alpha_j|)}=1$ for every $j$. When $n$ is even, $(-1)^{|s\alpha_j|(n-|s\alpha_j|)}=(-1)^{|s\alpha_j|}$. So $\prod_{j=1}^p(-1)^{|s\alpha_j|(n-|s\alpha_j|)}=\prod_{j=1}^p(-1)^{|s\alpha_j|}=(-1)^n=1$.

\begin{thm}\label{longest}
Suppose $(A,m^A)$ is an $A_\infty$-algebra satisfying the conditions above. If $H^i(A)=0$ for $1\leq r\leq k$, and $l\geq 3$ is an integer such that $N\leq (l+1)k+2$, then $A$ has an $A_\infty$-minimal model with $m_p=0$ for $p\geq l$.
\end{thm}

\begin{proof}
The idea is first constructing a minimal model $(H^*(A),m)$ of $A$ and a quasi-isomorphism $f:H^*(A)\to A$ following the proof of Theorem \ref{kadei}. Then do some modifications to make $m_p=0$ for $p\geq l$.

By Theorem \ref{strictly-unital}, $(H^*(A),m)$ and $f$ are strictly unital. As $\dim H^0(A)=1$ and $H^i(A)=0$ for $1\leq r\leq k$, if $b_l(sx_1,\ldots,sx_l)$ is non-trivial for $l\geq 3$, the degree of each $sx_j$ is at least $k$. Thus, the degree of $b_l(sx_1,\ldots,sx_l)$ is at least $lk+1$. On the other hand, we have assumed that a basis of $H^i(A)$ have a "dual" in $H^{N-i}(A)$. When $i\geq lk+2$, $N-i\leq (l+1)k+2-(lk+2)=k$. In this case $\dim H^i(A)$ is 0 for $i\neq N$ and is 1 for $i=N$. Hence, if $b_l(sx_1,\ldots,sx_l)$ is non-trivial, it must be in $SH^{N-1}(A)$, i.e. the total degree $|sx_1|+\ldots+|sx_l|=N-2$.

Now we define another $A_\infty$-structure $m'$ on $H^*(A)$ and a quasi-isomorphism $g:(H^*(A),m')\to(A,m^A)$ such that $m'_l=0$. Set $m'_p=m_p$ and $g_p=f_p$ for $p\leq l-2$. Then $m'_{l-1}$ must equal to $m_{l-1}$ and $f_{l-1}=g_{l-1}+\delta_{l-1}$ with $m_1^A \delta_{l-1}=0$. Define $sF_l$ and $sG_l$ as in (\ref{sF_p}). We have
\begin{equation}\label{sF-sG}
(sF)_l-(sG)_l= b^A_2 \big( (s\delta)_{l-1}\otimes (sg)_1+(sg)_1 \otimes (s\delta)_{l-1} \big)-\sum_{r=0}^{l-2}(s\delta)_{l-1}\left(\textbf{1}^{\otimes r}\otimes b_2\otimes \textbf{1}^{\otimes(l-r-2)}\right).
\end{equation}

As discussed above, $m'_l=0$ if and only if the cohomology class of $[(sF_l-sG_l)(sx_1,\ldots,sx_l)]=b_l(sx_1,\ldots,sx_l)$ with total degree $|sx_1|+\ldots+|sx_l|=N-2$. We can make $(s\delta)_{l-1}(sy_1,\ldots,sy_{l-1})=0$ when the total degree $|sy_1|+\ldots+|sy_{l-1}|=N-1$, so that the summation part of (\ref{sF-sG}) will vanish. Then the required equation becomes
\begin{align*}
& \quad b_l(sx_1,\ldots,sx_l) \\
&= \Big[ b_2^A\big( (s\delta)_{l-1}(sx_1,\ldots,sx_{l-1}),(sg)_1(sx_l) \big)+b_2^A\big( (sg)_1(sx_1),(s\delta)_{l-1}(sx_2,\ldots,sx_l) \big)\Big]\\
&= \Big[ b_2^A\big( (s\delta)_{l-1}(sx_1,\ldots,sx_{l-1}),(sg)_1(sx_l) \big) \\
& \quad - (-1)^{|sx_1|(N-2-|sx_1|)} b_2^A\big( (s\delta)_{l-1}(sx_2,\ldots,sx_l),(sg)_1(sx_1) \big)\Big] \\
&=\Big[ b_2^A \big( (s\delta)_{l-1}\otimes (sg)_1 \big)(\mathbf{1}-\sigma)(sx_1,\ldots,sx_l) \Big]
\end{align*}
The equation above is equivalent to
$$
b_l\sigma^a(sx_1,\ldots,sx_l)=\Big[ b_2^A \big( (s\delta)_{l-1}\otimes (sg)_1 \big)(\sigma^a-\sigma^{a-1})(sx_1,\ldots,sx_l) \Big]
$$
for any $a\in\mathbb{Z}$.

For each $SH^{i-1}(A)$, take a basis $\left\{ sx^{(i)}_1,\ldots,sx^{(i)}_{\beta_i}\right\}$, and its "dual" basis $\left\{ sy^{(d-i)}_1,\ldots,sy^{(d-i)}_{\beta_i}\right\}\in SH^{d-i-1}(A)$ such that $b_2(sy^{(N-i)}_v,sx^{(i)}_u)=\delta_{uv}s\mu$. Here $s\mu=sx^{(N)}_1$ is the generator of $SH^{N-1}(A)$. Then for arbitrary $sx^{(i_1)}_{u_1},\ldots,sx^{(i_{l-1})}_{u_{l-1}}$ with $sx^{(i_r)}_{u_r}\in SH^{i_r-1}(A)$, $|sx^{(i_r)}_{u_r}|\geq 0$ (so $\geq k$) for $1\leq r\leq l-1$, and $|sx^{(i_1)}_{u_1}|+\ldots+|sx^{(i_{l-1})}_{u_{l-1}}|<N-1$. Let $t=N-\left(|sx^{(i_1)}_{u_1}|+\ldots+|sx^{(i_{l-1})}_{u_{l-1}}|\right)$, then for each generator $sx^{(t)}_v\in SH^{(t-1)}(A)$, the following equation needs to be satisfied.
$$
b_l\sigma^a(sx^{(i_1)}_{u_1},\ldots,sx^{(i_{l-1})}_{u_{l-1}},sx^{(t)}_v) =\Big[ b_2^A \big( (s\delta)_{l-1}\otimes (sg)_1 \big)(\sigma^a-\sigma^{a+1})(sx^{(i_1)}_{u_1},\ldots,sx^{i_{(l-1)}}_{u_{l-1}},sx^{(t)}_v) \Big].
$$
The equations above form a linear equation system for $\Big[ b_2^A \big( (s\delta)_{l-1}\otimes (s\delta)_1 \big)\sigma^a(sx^{(i_1)}_{u_1},\ldots,sx^{(i_{l-1})}_{u_{l-1}},sx^{(t)}_v) \Big]$. Since $\sigma^l=\textbf{1}$, this system has $l$ variables and $l$ equations. Add all these $l$ equations together. The right-hand side is
$$
\Big[ b_2^A \big( (s\delta)_{l-1}\otimes (sg)_1 \big)(\textbf{1}-\sigma^l)(sx^{(i_1)}_{u_1},\ldots,sx^{(i_{l-1})}_{u_{l-1}},sx^{(t)}_v) \Big]=0.
$$
Also by the assumption of Condition 4, the left-hand side is
$$
\sum_{a=0}^{l-1}b_l\sigma^a(sx^{(i_1)}_{u_1},\ldots,sx^{(i_{l-1})}_{u_{l-1}},sx^{(t)}_v)=0.
$$
Therefore, this system has solutions. We can take one of the solutions as
$$
\Big[ b_2^A \big( (s\delta)_{l-1}\otimes (sg)_1 \big)\sigma^a(sx^{(i_1)}_{u_1},\ldots,sx^{(i_{l-1})}_{u_{l-1}},sx^{(t)}_v) \Big]=\sum_{j=0}^{l-1}\frac{l-j}{l} b_l\sigma^{a+j}(sx^{(i_1)}_{u_1},\ldots,sx^{(i_{l-1})}_{u_{l-1}},sx^{(t)}_v)
$$
Since $sH^{N-1}(A)$ is generated by $s\mu$, we can define an operator $*$ from $sH^{N-1}(A)$ to the ground field such that $*(c\cdot s\mu)=c$. So we can set
$$
(s\delta)_{l-1}(sx^{(i_1)}_{u_1},\ldots,sx^{(i_{l-1})}_{u_{l-1}})=\sum_{v=1}^{\beta_{t}}\sum_{j=0}^{l-1}\frac{l-j}{l} \big( *b_l\sigma^j (sx^{(i_1)}_{u_1},\ldots,sx^{(i_{l-1})}_{u_{l-1}},sx^{(t)}_v) \big)(sg)_1(sy^{(N-t)}_v),
$$
when $|sx^{(i_r)}_{u_r}|\geq 0$ for $1\leq r\leq l-1$ and $|sx^{(i_1)}_{u_1}|+\ldots+|sx^{(i_{l-1})}_{u_{l-1}}|<N-1$.
And Set
$$
(s\delta)_{l-1}(sx^{(i_1)}_{u_1},\ldots,sx^{(i_{l-1})}_{u_{l-1}})=0
$$
when some $|sx^{(i_r)}_{u_r}|=-1$ or $|sx^{(i_1)}_{u_1}|+\ldots+|sx^{(i_{l-1})}_{u_{l-1}}|\geq N-1$.

By the discussion above, this construction makes $m'_l=0$.

Then we can define $(sg)_l$ and $b'_{l+1}$ as the proof of Theorem \ref{kadei}. Since $s\delta$ satisfies the condition of strictly unital, so do $(sG)_l$, $(sg)_l$ and $b'_{l+1}$. Thus, when $b'_{l+1}(sx_1,\ldots,sx_{l+1})$ is non-trival, the degree of each $sx_j$ is at least $k$, and the degree of $b'_{l+1}(sx_1,\ldots,sx_{l+1})$ is at least $(l+1)k+1$. It is possible only when $N=(l+1)k+2$ and the total degree $|sx_1|+\ldots+|sx_{l+1}|=N-2$.

Similarly, we define an $A_\infty$-structure $m''$ on $H^*(A)$ and a quasi-isomorphism $h:(H^*(A),m'')\to(A,m^A)$ such that $m'l_{l+1}=0$. Set $m''_p=m'_p$ and $h_p=g_p$ for $p\leq l-1$. Then $m''_l=m'_l=0$ and $g_l=h_l+\delta_l$ with $m_1^A\delta_l=0$. Hence,
\begin{equation}\label{sG-sH}
(sG)_{l+1}-(sH)_{l+1}= b^A_2 \big( (s\delta)_l\otimes (sh)_1+(sh)_1 \otimes (s\delta)_l \big)-\sum_{r=0}^{l-1}(s\delta)_l(\textbf{1}^{\otimes r}\otimes b'_2\otimes \textbf{1}^{\otimes(l-r-1)}).
\end{equation}

In a similar way, we can define $(s\delta)_l$ by setting
$$
(s\delta)_l(sx^{(i_1)}_{u_1},\ldots,sx^{(i_l)}_{u_l})=0
$$
when some $|sx^{(i_r)}_{u_r}|=-1$ or $|sx^{(i_1)}_{u_1}|+\ldots+|sx^{(i_l)}_{u_l}|\geq N-1$. This construction makes the summation part of (\ref{sG-sH}) become 0. And we set
$$
(s\delta)_l(sx^{(i_1)}_{u_1},\ldots,sx^{(i_l)}_{u_l})=\sum_{v=1}^{\beta_t}\sum_{j=0}^l\frac{l+1-j}{l+1} \big( *b'_{l+1}\sigma^j (sx^{(i_1)}_{u_1},\ldots,sx^{(i_l)}_{u_l},sx^{(t)}_v) \big)(sh)_1(sy^{(N-t)}_v),
$$
when $|sx^{(i_r)}_{u_r}|\geq 0$ for $1\leq r\leq l$ and $|sx^{(i_1)}_{u_1}|+\ldots+|sx^{(i_l)}_{u_l}|<N-1$. Here $t=N-\left(|sx^{(i_1)}_{u_1}|+\ldots+|sx^{(i_l)}_{u_l}|\right)$. This construction makes
$$
b'_{l+1}\sigma^a(sx^{(i_1)}_{u_1},\ldots,sx^{(i_l)}_{u_l},sx^{(t)}_v) =\Big[ b_2^A \big( (s\delta)_l\otimes (sh)_1 \big)(\sigma^a-\sigma^{a+1})(sx^{(i_1)}_{u_1},\ldots,sx^{(i_l)}_{u_l},sx^{(t)}_v) \Big]
$$
for arbitrary $sx^{(i_1)}_{u_1},\ldots,sx^{(i_l)}_{u_l},sx^{(t)}_v\in SH^*(A)$ and arbitrary $a\in\mathbb{Z}$.

Therefore, for arbitrary $sx_1,\ldots,sx_{l+1}\in sH^*(A)$ with $|sx_1|+\ldots+|sx_{l+1}|$, we have
\begin{align*}
& \quad b'_{l+1}(sx_1,\ldots,sx_{l+1}) \\
&= \Big[ b_2^A\big( (s\delta)_l(sx_1,\ldots,sx_l),(sh)_1(sx_{l+1}) \big)+b_2^A\big( (sh)_1(sx_1),(s\delta)_l(sx_2,\ldots,sx_{l+1}) \big)\Big]
\end{align*}
This implies $m''_{l+1}=0$.

Same as the proof of Theorem \ref{kadei}, we can continue define $(sh)_{l+1}$ then $b''_p$ and $(sh)_p$ for $p\geq l+2$ inductively. By the discussion of Theorem \ref{strictly-unital}, $b''$ and $sh$ are strictly unital. Thus, for $p\geq l+2$, $b''_p(sx_1,\ldots,sx_p)=0$ if some $|sx_j|=-1$. If every $|sx_j|\geq 0$, the degree is at least $k$. So $b''_p(sx_1,\ldots,sx_p)\geq pk+1>(l+1)k+1\geq N-1$ and it must be 0. Hence, $b''_p=0$.

By the discussion above, $(H^*(A),m'')$ is an $A_\infty$-minimal model of $(A,m^A)$ with $m''_p=0$ for $p\geq l$.
\end{proof}

As a corollary, when $M$ is a $k$-connected compact orientable manifold, $\Omega^*(M)$ satisfied the theorem above. More generally, we can show that the statement is also true when $M$ is not orientable.

\begin{thm}\label{connected}
Suppose $M$ is an $N$-dimensional $k$-connected compact manifold. If $l\geq 3$ such that $N\leq (l+1)k+2$, then $\Omega^*(M)$ has an $A_\infty$-minimal model with $m_p=0$ for $p\geq l$. 
\end{thm}

\begin{proof}
As $M$ is $k$-connected, $\pi_i(M)=0$ for all $1\leq i\leq k$. It follows that $H^i(M)=0$ by the Hurewicz Theorem.

The case that $M$ is orientable is a special case of Theorem \ref{longest}. When $M$ is not orientable, $H^N(M)=0$. Let $\tilde{M}$ be the orientation bundle over $M$. $\tilde{M}$ is also connected and $\pi_i(\tilde{M})=0$ for $1\leq i\leq k$. So when $1\leq i\leq k$, $H_i(\tilde{M})=0$. By twisted Poincar\'e duality, $H^{N-i}(M)\simeq H_i(\tilde{M})=0$. Therefore, $H^i(M)=0$ for all $i\geq N-k$.

Construct a minimal model $(H^*(M),m)$ of $\Omega^*(M)$. Then $m$ is strictly unital. Thus, on the suspension if $b_p(sx_1,\ldots,sx_p)$ is non-trivial, then $|sx_j|\geq 0$ for all $1\leq j\leq p$, i.e. $|sx_j|\geq k$. When $p\geq l$, the degree of $b_p(sx_1,\ldots,sx_p)$ is at least $pk+1$. Since $pk+1\geq lk+1\geq N-k-1$, and $SH^i(M)=0$ for $i\geq N-k-1$, $b_p(sx_1,\ldots,sx_p)$ must be 0. Therefore, $b_p=0$ for all $p\geq l$. 
\end{proof}

In the proof of Theorem \ref{longest}, $g_l$ and $h_{l+1}$ are constructed in a similar way. It would be interesting to construct $A_\infty$-minimal models for other types of  DGAs or $A_\infty$-algebras following this way. For example, we may extend Cavalcanti's result \cite{caval2} that a compact orientable $k$-connected manifold where the $(k+1)$-Betti number $b_{k+1}=1$ is formal if its dimension $N\leq 4k+4$. A conjecture is that the de Rham complex of such a manifold has an $A_\infty$-minimal model with $m_p=0$ for $p\geq j$ if its dimension $N\leq (l+1)k+4$.

\section{Minimal model of an extension of formal DGA and its application on odd-dimensional sphere bundles}

Let $A$ be a formal DGA, and $A[\theta]=A\otimes \wedge(\theta)$ be the extension of $A$ by an odd-degree generator $\theta$. The odd degree of $\theta$ implies $\theta\wedge\theta=0$. Hence, we can also write $A[\theta]=A\oplus \theta A$. For the differential, we set $d\theta=\omega$ for some even-degree $\omega\in A$, with $|\omega|=|\theta|+1$. In Section 4.1 below, we will describe how simple this $A_\infty$-minimal model of $A[\theta]$ can be.

A topological example is that $A=C^*(M)$, the singular cochain complex of a formal manifold $M$. Then $A[\theta]$ is the singular cochain complex of the mapping cone $\omega:C^*(M)\to C^*(M)$ by taking cup product with $\omega$. Geometrically, when $\omega$ is an integral differential form, we consider $A=\Omega^*(M)$. In this case $A[\theta]$ is quasi-isomorphic to $\Omega^*(X)$, where $X$ is a sphere bundle over $M$ whose Euler class is $\omega$.

When $M$ is a symplectic manifold and $A=\Omega^*(M)$, $A[\theta]$ is also quasi-isomorphic to an $A_3$-algebra over $M$ \cite{tt}, constructed by Tsai, Tseng and Yau \cite{tty}. Moreover, when $\omega\in A$ is taken as the symplectic form of $M$ and is integral, $X$ becomes a circle bundle over $M$ and is called the Boothby-Wang fibration. In Section 4.2, we will describe the formality of such $X$ with $M$ formal and satisfying the hard Lefschetz property.

\subsection{Minimal model of an extension of formal DGA}

First we will show that when $M$ is formal, $A[\theta]$ is quasi-isomorphic to $H[\theta]$, which is the extension of $H^*(A)$ by $\theta$. Then we can consider a much simpler DGA.

\begin{thm}\label{extension-quasi}
Suppose $A,B$ are two DGAs and $f:A\to B$ is a quasi-isomorphism. $\omega_A\in A,\omega_B\in B$ are $d$-closed even-degree elements such that $f^*([\omega_A])=[\omega_B]$. Extend $A,B$ to $A[\theta]=\{ \alpha+\theta_A\beta | \ \alpha,\beta\in A \}$ with $d\theta_A=\omega_A$ and $B[\theta]=\{ x+\theta_B y | \  x,y\in B \}$ with $d\theta_B=\omega_B$. Then there exists a quasi-isomorphism $g:A[\theta]\to B[\theta]$.
\end{thm}

\noindent \textit{Proof.} Without loss of generality, we can assume $f(\omega_A)=\omega_B$. Otherwise, by assumption $\omega_B=f(\omega_A)+dr$ for some $r\in A$. Then we can consider $\omega'_B=\omega_B-dr$ and $\theta'_B=\theta_B-r$ instead of $\omega_B$ and $\theta_B$.

Set
$$
g(\alpha+\theta_A\beta)=f(\alpha)+\theta_Bf(\beta).
$$
It is easy to check that $g$ is linear, preserves wedge products and $gd_A=d_Bg$. It remains to show that $g^*$ is bijective.

\noindent\textit{1) $g^*$ is injective.}

Suppose $\alpha+\theta_A\beta$ is closed in $A[\theta]$ and $g^*[\alpha+\theta_A\beta]=0$. There exists $x,y\in B$ such that 
$$d(x+\theta_By)=g(\alpha+\theta_A\beta).$$
Thus, $$dx+\omega_B\wedge y-\theta_Bdy=f(\alpha)+\theta_Bf(\beta).$$
So we have 
$$dx+\omega_B\wedge y=f(\alpha)\quad\text{and}\quad dy=-f(\beta).$$

On the other hand,
$$0=d(\alpha+\theta_A\beta)=d\alpha+\omega_A\wedge\beta-\theta_Ad\beta.$$
Hence, $\beta$ is closed and $\omega_A\wedge\beta=-d\alpha$. Since $f^*[\beta]=-[dy]=0$, $\beta$ must be exact in $A$.

Assume $\eta\in A$ such that $d\eta=\beta$. By
$$d(y+f(\eta))=-f(\beta)+f(d\eta)=0$$
and $f^*$ is surjective, there exists $\xi\in A$ and $z\in B$ such that 
$$d\xi=0 \quad\text{and}\quad f(\xi)=y+f(\eta)+dz.$$
Then
\begin{align*}
f(\alpha+\omega_A\wedge\eta) &= f(\alpha)+\omega_B\wedge f(\eta) \\
&= dx+\omega_B\wedge y+\omega_B\wedge f(\eta) \\
&= dx+\omega_B\wedge(f(\xi)-dz)\\
&= f(\omega_A\wedge\xi)+dx-dz.
\end{align*}
So $f(\alpha+\omega_A\wedge\eta-\omega_A\wedge\xi)=dx-dz$ is exact in $B$. Also
$$d(\alpha+\omega_A\wedge\eta-\omega_A\wedge\xi)=d\alpha+\omega_A\wedge\beta=0.$$
Hence, $\alpha+\omega_A\wedge\eta-\omega_A\wedge\xi$ is exact since $f^*$ is injective. Let $\gamma\in A$ such that
$$d\gamma=\alpha+\omega_A\wedge\eta-\omega_A\wedge\xi.$$

Therefore,
\begin{equation*}
\begin{split}
\alpha+\theta_A\beta &= d\gamma-\omega_A\wedge\eta+\omega_A\wedge\xi+\theta_A d\eta \\
&= d\gamma-d(\theta_A\eta)+d(\theta_A\xi),
\end{split}
\end{equation*}
which is exact in $A$. That shows $g^*$ is injective. \\[0.1in]

\noindent \textit{2) $g^*$ is surjective.}

Given arbitrary closed $x+\theta_B y\in B[\theta]$, we have
$$ dx+\omega_B\wedge y=0 \quad\text{and}\quad dy=0.$$
As $y$ is closed and $f^*$ is surjective, there exists $\beta\in A,z\in B$ such that
$$f(\beta)=y+dz \quad\text{and}\quad d\beta=0.$$
Then
$$f(\omega_A\wedge\beta)=\omega_B\wedge(y+dz)=-dx+d(\omega_B\wedge z).$$
Since $\omega_A\wedge\beta$ is closed and $f^*$ is injective, $\omega_A\wedge\beta$ must be exact. So there exists $\alpha\in A$ such that $d\alpha=\omega_A\wedge\beta$. Thus,
\begin{equation*}
\begin{split}
d(x-\omega_B\wedge z+f(\alpha)) &= -\omega_B\wedge y-\omega_B\wedge dz+f(d\alpha)\\
&= -\omega_B\wedge(y+dz)+f(\omega_A\wedge\beta)\\
&= -\omega_B\wedge f(\beta)+\omega_B\wedge f(\beta)\\
&= 0.
\end{split}
\end{equation*}
So there exists $\xi\in A$ and $w\in B$ such that
$$f(\xi)=x-\omega_B\wedge z+f(\alpha)+dw \quad\text{and}\quad d\xi=0.$$
Therefore,
\begin{equation*}
\begin{split}
f(\xi-\alpha)+\theta_B f(\beta) &= x-\omega_B\wedge z+dw+\theta_B(y+dz)\\
&= x+\theta_B\wedge y-d(\theta_B z)+dw
\end{split}
\end{equation*}
i.e. $$g^*[\xi-\alpha+\theta_A\beta]=[x+\theta_B y].$$
Thus, $g^*$ is surjective.\qed

When $A$ is formal, there exists a zigzag of quasi-isomorphisms between $A$ and $H^*(A)$. We can extend each quasi-isomorphism by the previous theorem, and obtain the following statement.

\begin{cor}\label{extension-cohomology}
Suppose $A$ is a formal DGA. $\omega\in A$ is a closed even-degree element. $A[\theta]=\{ \alpha+\theta_A\beta | \ \alpha,\beta\in A \}$, where $d\theta_A=\omega_A$. $A[\theta]$ is quasi-isomorphic to $H[\theta]=\{ x+\theta_H y | \ x,y\in H^*(A) \}$, where $d\theta_H=[\omega_A]$.
\end{cor}

Now we can consider the extension of a DGA $\mathcal{H}$ whose differential is 0, then construct its $A_\infty$-minimal model.

\begin{thm}\label{extension-formal}
Suppose $\mathcal{H}$ is a DGA and $d_\mathcal{H}=0$. $\omega_\mathcal{H}\in \mathcal{H}$ is an even-degree element. Let $\mathcal{H}[\theta]=\{ \alpha+\theta_\mathcal{H}\beta | \ \alpha,\beta\in \mathcal{H} \}$ with $d\theta_\mathcal{H}=\omega_\mathcal{H}$. Then $\mathcal{H}[\theta]$ has an $A_\infty$-minimal model with $m_p=0$ for all $p$ except for $p=2$ or $3$.
\end{thm}

\noindent \textit{Proof.} Since $d_\mathcal{H}=0$, for arbitrary $\alpha,\beta\in \mathcal{H}$, $\alpha+\theta_\mathcal{H}\beta$ is closed if and only if $\omega_\mathcal{H}\wedge\beta=0$. It is exact if and only if $\alpha\in I(\omega_\mathcal{H})$ and $\beta=0$, where $I(\omega_\mathcal{H})=\{ \omega\wedge\alpha | \ \alpha\in \mathcal{H} \}$ is the ideal generated by $\omega_\mathcal{H}$ in $\mathcal{H}$. Thus,
$$H^*\big(\mathcal{H}[\theta]\big)=\big( \mathcal{H}/I(\omega_\mathcal{H}) \big) \oplus \ker \omega_\mathcal{H},$$
where $\omega_\mathcal{H}$ is an operator on $\mathcal{H}$ by multiplying $\omega_\mathcal{H}$. \\[0.1in]

\noindent \textit{1) Defining $f_1$.}

Decompose $\mathcal{H}=I(\omega_\mathcal{H})\oplus \mathcal{H}^C$ for some subspace $\mathcal{H}^C$ of $\mathcal{H}$. For each cohomology class $[\alpha+\theta_\mathcal{H}\beta]$ in $H^*\big(\mathcal{H}[\theta]\big)$, by the discussion above there exists a unique $\alpha_0\in \mathcal{H}^C,\beta_0\in \ker \omega_\mathcal{H}$ such that $\alpha_0+\theta_\mathcal{H}\beta_0\in[\alpha+\theta_\mathcal{H}\beta].$ So we can set
\begin{equation*}
\begin{split}
f_1: H^*\big(\mathcal{H}[\theta]\big)\to \mathcal{H}[\theta], \quad [\alpha+\theta_\mathcal{H}\beta]\mapsto \alpha_0+\theta_\mathcal{H}\beta_0.
\end{split}
\end{equation*}
It is easy to verify $f_1$ is a quasi-isomorphism. \\[0.1in]

\noindent \textit{2) Defining $f_2$.}

Give another decomposition of $\mathcal{H}$ by $\mathcal{H}=\ker \omega_\mathcal{H} \oplus \mathcal{H}^{\perp}$ for some subspace $\mathcal{H}^{\perp}$ of $\mathcal{H}$. For each $\alpha\in I(\omega_\mathcal{H})$, there exists a unique $\beta\in \mathcal{H}^{\perp}$ such that $\alpha=\omega_\mathcal{H}\wedge\beta$. So we can define a map $Q:I(\omega_\mathcal{H})\to \theta_\mathcal{H} \mathcal{H}$ by $Q(\alpha)=\theta_\mathcal{H}\beta.$ Then set $f_2$ as
$$f_2(x,y)=Q \big(f_1m_2(x,y)-f_1(x)\wedge f_1(y)\big).$$
Such $f_2$ is well-defined. Suppose
$$
f_1(x)=\alpha+\theta_\mathcal{H}\beta, \quad f_1(y)=\xi+\theta_\mathcal{H}\eta,
$$
then
$$
m_2(x,y)=[f_1(x)\wedge f_1(y)]=[\alpha\wedge\xi+\theta_\mathcal{H}\beta\wedge\xi+(-1)^{|\alpha|}\theta_\mathcal{H}\alpha\wedge\eta].
$$
Hence,
$$
f_1m_2(x,y)=f_1([\alpha\wedge\xi])+\theta_\mathcal{H}\beta\wedge\xi+(-1)^{|\alpha|}\theta_\mathcal{H}\alpha\wedge\eta,
$$
and
$$
f_1m_2(x,y)-f_1(x)\wedge f_1(y)=f_1([\alpha\wedge\xi])-\alpha\wedge\xi\in I(\omega_\mathcal{H}).
$$
As $dQ$ is the identity map on $I(\omega_\mathcal{H})$, $f_2$ satisfies the equation 
$$m_1f_2=dQ\big( f_1m_2-m_2(f_1\otimes f_1) \big)=f_1m_2-m_2(f_1\otimes f_1).$$
\\[0.1in]

\noindent \textit{3) Defining $m_3$ and $f_3$.}

$m_3$ and $f_3$ need to satisfy
$$f_1m_3-m_1f_3=F_3=m_2(f_1\otimes f_2-f_2\otimes f_1)-f_2(m_2\otimes \mathbf{1}-\mathbf{1}\otimes m_2),$$
and $m_3$ is the cohomology class of $F_3$. By the definition of $f_2$, its image is in $I(\theta_\mathcal{H})$, which is the ideal generated by $\theta_\mathcal{H}$ in $A[\theta]$. Hence, for any $x,y,z\in H^*\big(\mathcal{H}[\theta]\big)$,
$$m_2(f_1\otimes f_2-f_2\otimes f_1)(x,y,z)-f_2(m_2\otimes\mathbf{1}-\mathbf{1}\otimes m_2)(x,y,z)=\theta_\mathcal{H} \alpha$$
for some $\alpha\in \mathcal{H}$. Thus,
$$m_3(x,y,z)=[\theta_\mathcal{H}\alpha],\quad\text{and}\quad f_1m_3(x,y,z)=f_1([\theta_\mathcal{H}\alpha])=\theta_\mathcal{H}\alpha.$$

Therefore, $m_1f_3(x,y,z)=0$, and we can set $f_3=0$.

\noindent \textit{4) Triviality of $m_4$ and $f_4$.}

As $f_3=0$, $m_4$ and $f_4$ need to satisfy
$$
f_1m_4-m_1f_4=-m_2(f_2\otimes f_2)+f_2(m_3\otimes \mathbf{1}+\mathbf{1}\otimes m_3).
$$
We claim $m_2(f_2\otimes f_2)=0$ since $\im f_2\in I(\theta_\mathcal{H})$ and $\theta_\mathcal{H}\wedge\theta_\mathcal{H}=0$. On the other hand, for any $x,y,z,w\in H^*\big(\mathcal{H}[\theta]\big)$, we can assume
$$
m_3(x,y,z)=\theta\alpha \quad\text{and}\quad f_1(w)=\beta+\theta\gamma
$$
for some $\alpha,\beta,\gamma\in \mathcal{H}$. Then
$$
f_1m_2 \big(m_3(x,y,z),w \big)= f_1 \big( [\theta_\mathcal{H}\alpha\wedge(\beta+\theta\gamma)] \big)
= f_1[\theta_\mathcal{H}\alpha\beta]
= \theta_\mathcal{H}\alpha\beta,
$$
and
$$
m_2\big( f_1m_3(x,y,z),f_1(w) \big)= f_1\big( [\theta_\mathcal{H}\alpha] \big) \wedge(\beta+\theta_\mathcal{H}\gamma) = \theta_\mathcal{H}\alpha\beta.
$$
Hence, $m_1f_2 \big(m_3(x,y,z),w \big)=0$. By previous discussion we have $f_2=Qm_1f_2$, so $f_2(m_3\otimes \mathbf{1})=0$. Similarly, $f_2(\mathbf{1}\otimes m_3)=0$.

Therefore, $m_4=0$ and we can set $f_4=0$.

\noindent \textit{5) Triviality of higher $m_p$ and $f_p$.}

For higher degrees, we will prove $m_p=0$ and $f_p=0$ by induction. Suppose $m_p=0$ on $H^*\big(\mathcal{H}[\theta]\big)$ for $4\leq p\leq n-1$ and $f_p=0$ for $3\leq p\leq n-1$, where $n\geq 5$. $m_n$ and $f_n$ need to satisfy
\begin{equation*}
\begin{split}
& \quad f_1m_n-m_1f_n \\
&= \sum_{\substack{i_1+\ldots+i_r=n \\ r\geq 2}} (-1)^{\delta_1}m_r(f_{i_1}\otimes\ldots\otimes f_{i_r})-\sum_{\substack{r+s+t=n \\ 2\leq s\leq n-1}}(-1)^{\delta_2}f_{r+t+1}(\mathbf{1}^{\otimes r}\otimes m_s \otimes \mathbf{1}^{\otimes t}) \\
&= \sum_{i=1}^{n-1}(-1)^{\delta_1}m_2(f_i\otimes f_{n-i})-\sum_{r=0}^1 (-1)^{\delta_2}f_2(\mathbf{1}^{\otimes r}\otimes m_{n-1}\otimes \mathbf{1}^{\otimes (1-r)})
\end{split}
\end{equation*}
where $\delta_1=\sum_{t=1}^r(n-t)(i_t-1)$ and $\delta_2=r+st$.

Since $n\geq 5$, either $i\geq 3$ or $n-i\geq 3$, so $m_2(f_i\otimes f_{n-i})=0$. Also, $n-1\geq 4$. So $m_{n-1}=0$. That implies $f_1m_n-m_1f_n=0$. Therefore, $m_n=0$ and we can take $f_n=0$. \qed

By the previous theorem, we have the following statement for formal DGA.

\begin{thm}\label{extension-dga}
Suppose $A$ is a formal DGA, and $\omega_A\in A$ is an even-degree element. Extend $A$ to $A[\theta]=\{ \alpha+\theta_A\beta | \ \alpha,\beta\in A \}$ with $d\theta_A=\omega_A$. Then $A[\theta]$ has an $A_\infty$-minimal model with $m_p=0$ for all $p$ except for $p=2$ or $3$.
\end{thm}

When $A=\Omega^*(M)$ for some formal manifold $M$ and $\omega\in A$ is an integral differential form, $A[\theta]$ is quasi-isomorphic to $\Omega^*(X)$. Here $X$ is a sphere bundle over $M$ whose Euler class is $\omega$. So we have 

\begin{thm}\label{extension}
Let $M$ be a formal manifold, $\omega\in \Omega^*(M)$ be an even-dimensional integral differential form, and $X$ be the sphere bundle over $M$ with Euler class $\omega$. Then $X$ is formal if $\omega$ is exact. When $\omega$ is non-exact, $\Omega^*(X)$ has an $A_\infty$-minimal model with $m_p=0$ for all $p$ except for $p=2$ or $3$.
\end{thm}

When $\omega$ is non-exact, $X$ may be or may not be formal. There are examples for both cases. We will talk about this in next subsection.

In the Introduction, we asked whether we can construct formal manifolds from a given formal manifold, and we have only described the special case of odd dimensional sphere bundles here. It is natural to consider the case of even dimensional sphere bundles next, or more generally, other types of fiber bundles. We can also think about other ways of obtaining new manifolds, such as symplectic reduction, blowing up and down.

Another question is whether we can extend Theorem \ref{extension-dga}, which would fit as the $k=3$ case of the following broader statement. Suppose $A$ is a DGA. $\omega\in A$ is an even-degree element. Extend $A$ to $A[\theta]=\{ \alpha+\theta\beta | \ \alpha,\beta\in A \}$ with $d\theta=\omega$. If $A$ has an $A_\infty$-minimal model with $m_p=0$ for $p\geq k$, can we prove that $A[\theta]$ has an $A_\infty$-minimal model with $m_p=0$ for $p\geq k+1$?

\subsection{Formality of circle bundle over formal manifolds with hard Lefschetz property}

In this subsection, we will focus on the special case of circle bundles. The base $(M^{2n},\omega)$ is assumed to be a compact formal symplectic manifold satisfying the hard Lefschetz property. For simplicity, we use $A$ to denote the DGA $H^*(M)$, and use $\omega$ to denote its cohomology class $[\omega]$ in $A$ unless otherwise stated. When $\omega$ is integral, let $X$ denote the circle bundle over $M$ with Euler class $\omega$.

By Corollary \ref{extension-cohomology}, $\Omega^*(M)[\theta]$ is quasi-isomorphic to $A[\theta]$, where $d\theta=\omega$.  Then the circle bundle $X$ is formal if and only if $A[\theta]$ is formal.

Since $M$ satisfies the hard Lefschetz property, $\omega^j:A^{n-j}\to A^{n+j}$ is an isomorphism. We can set the space of primitive classes 
$$
PH(M)^k = PA^k =\ker \omega^{n-k+1}.
$$
Then $A$ has the Lefschetz decomposition
$$
A^k=\bigoplus \omega^j PA^{k-2j}.
$$
So
$$
H^k\big(A[\theta]\big)\simeq
\begin{cases}
PA^k, & 0\leq k\leq n, \\
\theta\omega^{k-1-n}PA^{2n+1-k}, & n+1\leq k\leq 2n+1.
\end{cases}
$$

With this decomposition, we can introduce an operator $L^{-1}$ such that for 
$a\in PA^k$,
$$
L^{-1}a=0,\quad \text{and}\quad L^{-1}(\omega^ja)=\omega^{j-1}a \text{ for } 1\leq j\leq n-k.
$$

Let's first consider a simple case.

\begin{ex}\label{ex-projective}
Suppose $M=\mathbb{C}P^n$, and $\omega$ is taken as a representative of the generator of $A^2=H^2(M)$. Then $M$ is formal since it is K\"ahler. So the circle bundle $X$ is formal if and only if $A[\theta]=\{ \alpha+\theta\beta| \alpha,\beta\in A, d\theta=\omega \}$ is formal. The cohomology ring of $\mathbb{C}P^n$ is
$$
A^i=
\begin{cases}
\langle\,\omega^p\,\rangle, & \text{if } i=2p,0\leq i\leq 2n, \\
0, & \text{otherwise.}
\end{cases}
$$
Thus,
$$
A^i[\theta]=
\begin{cases}
\langle\,[\omega^p]\,\rangle, & \text{if } i=2p,0\leq p\leq n, \\
\langle\,\theta[\omega^p]\,\rangle, & \text{if } i=2p+1,0\leq p\leq n,\\
0, & i>2n+1.
\end{cases}
$$
Since $\omega^p=d(\theta\omega^{p-1})$, $H^i\big( A[\theta] \big)$ must be trivial except for $i=0$ or $2n+1$. The morphism $1\mapsto 1, \theta[\omega^n]\mapsto\theta\omega^n$ is a DGA quasi-isomorphism from $H^i\big( A[\theta] \big)$ to $A[\theta]$. So $X$ and $A[\theta]$ are formal.
\end{ex}

As we will see later, $X$ may not be formal even if $M$ is formal. By the following lemma, when $X$ is formal we can construct an $A_\infty$-quasi-isomorphism $f:H^i\big( A[\theta] \big)\to A[\theta]$ such that the image of $f_1$ is $PA^i\oplus\theta\omega^{n-i}PA^i$. Note that in this following theorem $A$ denote a general formal DGA rather than $H^*(M)$.

\begin{lem}\label{quasi-representative}
Suppose $A$ is an arbitrary formal DGA. Then there exists an $A_\infty$-quasi-isomorphism $\phi:SH^*(A)\to SA$ between suspensions, where $b_p=0$ on $SH^*(A)$ for $p\neq 2$. For any linear map $h:SH^*(A)\to SA$ of degree $-1$, we can find another $A_\infty$-quasi-isomorphism $\psi:SH^*(A)\to SA$ such that $\phi_1-\psi_1=b_1\circ h$.
\begin{proof}
Set
\begin{align*}
\psi_1 &= \phi_1-b_1h, \\
\psi_2 &= \phi_2-b_2(\phi_1\otimes h+h\otimes \psi_1)-hb_2, \\
\psi_p &= \phi_p-b_2(\phi_{p-1}\otimes h+h\otimes \psi_{p-1}), \text{ for } p\geq 3.
\end{align*}
By a straightforward calculation, we can verity that
$$
\sum_{r+s+t=p} \psi_p(\mathbf{1}^{\otimes r}\otimes b_s \otimes \mathbf{1}^{\otimes t})=\sum_{i_1+\ldots+i_r=p}b_r(\psi_{i_1}\otimes\ldots\otimes\psi_{i_r})
$$
for all $p\geq1$. So $\psi$ is the $A_\infty$-quasi-isomorphism we want.
\end{proof}
\end{lem}

Then we have a necessary condition to make $A[\theta]$ formal.

\begin{thm}\label{necessary condition for formal}
Suppose $A[\theta]$ is formal. Take arbitrary $x_1,\ldots,x_k\in PA^r$, $y_1,\ldots y_k\in PA^s$ satisfying $r+s\leq n+1$ and $\sum x_iy_i$ is in the ideal generated by $\omega$, i.e.
$$
\sum_{i=1}^{k}x_iy_i=\omega\alpha
$$
for some $\alpha\in A^{r+s-2}$. Then for any $z\in PA^{n+1-s}$, the following equation must hold
\begin{align}\label{necessary equation for formal}
\sum_{i=1}^{k}x_iL^{-1}(y_iz)=\sum_{i=1}^{k}L^{-1}(x_iy_i)z.
\end{align}
\begin{proof}
Let $f:H^*\big(A[\theta]\big)\to A$ be an $A_\infty$-quasi-isomorphism. By Lemma \ref{quasi-representative}, we can modify $f_1$ sending each cohomology class to any representative. So we can assume that
$$
f_1\big([x]\big)=x, \quad \text{and} \quad f_1\big([\theta\omega^{n-j}x]\big)=\theta\omega^{n-j}x
$$
for any $x\in PA^j$.

By $f_1m_2=m_1f_2+m_2(f_1\otimes f_1)$, we can obtain $m_1f_2\big([x_i],[y_i]\big)$ by calculating the other terms. When $r+s\leq n$, $f_1\big([x_iy_i]\big)$ is the primitive part of $x_iy_i$, i.e. projecting $x_iy_i$ to $PA$. This primitive part can be written as $(1-\omega L^{-1})(x_iy_i)$. Hence,
$$
m_1f_2\big([x_i],[y_i]\big)=(1-\omega L^{-1})(x_iy_i)-x_iy_i=-\omega L^{-1}(x_iy_i).
$$
When $r+s=n+1$, $x_iy_i$ is in the ideal generated by $\omega$ according to the hard Lefschetz property. So its cohomology class is 0, and $x_iy_i=\omega L^{-1}(x_iy_i)$. Then we also have
$$
m_1f_2\big([x_i],[y_i]\big)=0-x_iy_i=-\omega L^{-1}(x_iy_i).
$$
Thus, $f_2\big([x_i],[y_i]\big)$ is $-\theta L^{-1}(x_iy_i)$ plus some closed element in $A[\theta]$. Since its degree is not greater than n, that closed element must be in $A$. Similarly, $[y_iz]=0$ and $f_2\big([y_i],[z]\big)$ is in the coset $-\theta L^{-1}(y_iz)+A$.

By $m_2(f_1\otimes f_2-f_2\otimes f_1)-f_2(m_2\otimes 1-1\otimes m_2)+m_1f_3=0$, we let the left-hand side acting on $(x_i,y_i,z)$ for each $i$ and add them together. Then the $\theta A$ part of the first term is
\begin{align*}
\sum_{i=1}^k \Big( (-1)^r x_i (-\theta L^{-1}(y_iz))-(-\theta L^{-1}(x_iy_i))z \Big) = \theta  \sum_{i=1}^k \Big( L^{-1}(x_iy_i)z-x_i L^{-1}(y_iz) \Big).
\end{align*}
The second term $-f_2\big([\sum x_iy_i],[z]\big)+\sum f_2\big( [x_i],[y_iz] \big)$ vanishes as $[\sum x_iy_i]$ and all $[y_iz]$ are 0. The third term is an exact element in $A[\theta]$, which must be in $A$. Therefore, the coefficient of $\theta$ must be 0, i.e.
$$
\sum_{i=1}^k \Big( L^{-1}(x_iy_i)z-x_i L^{-1}(y_iz) \Big)=0.
$$
\end{proof}
\end{thm}

With this theorem, we can claim that $X$ is non-formal quickly in some special cases.

\begin{defn}
A cohomology class $a\in A=H^*(M)$ is called \textbf{reducible} if it is in $A^+\cdot A^+$, i.e. there exist $x_1,y_1,\ldots,x_k,y_k\in A$ such that
$$
a=\sum_{i=1}^k x_iy_i,
$$
and all $x_i,y_i$ have positive degree.
\end{defn}

\begin{cor}\label{necessary equation for formal kahler}
When $M$ is a compact K\"ahler manifold and $[\omega]$ is reducible, $X$ is non-formal.
\end{cor}

\begin{proof}
Since $\omega\in A$ is reducible, $x_1,y_1,\ldots,x_k,y_k\in A^1$ such that
$$
\omega=\sum_{i=1}^k x_iy_i.
$$
As $M$ is K\"ahler, $\omega\in H^{1,1}(M)$. So we can assume that all $x_i\in H^{1,0}(M)$ and all $y_i\in H^{0,1}(M)$. Since $\omega^n\neq 0$, there exists some $x_{i_1}y_{i_1}\ldots x_{i_n}y_{i_n}\neq 0$. Take $z=y_{i_1}\ldots y_{i_n}\neq 0$, then $z$ is a non-trivial class in $H^{0,n}(M)$. So $y_iz=0$ for any $i$, and we have
$$
\sum_{i=1}^{k}x_iL^{-1}(y_iz)=0.
$$
On the other hand,
$$
\sum_{i=1}^{k}L^{-1}(x_iy_i)z=(L^{-1}\omega)z=z\neq 0.
$$
By Theorem \ref{necessary condition for formal}, $A[\theta]$ is not formal. Neither is $X$.
\end{proof}

\begin{ex}\label{ex-torus}
Let $M=\Omega^*(T^{2n})$. Take $\omega=\sum_{i=1}^n x_iy_i$, where $x_i,y_i\in H^1(M)$ and $H^1(M)=\langle\, x_1,\ldots,x_n,y_1,\ldots,y_n \,\rangle$. Then $M$ is a K\"ahler manifold and $\omega$ is reducible. Hence, $X$ is non-formal.
\end{ex}

When the dimension of $M$ is low, equation $(\ref{necessary equation for formal})$ is also sufficient for the formality of $X$.

\begin{thm}\label{condition equivalent to formal}
When the dimension of $M$ is not greater than 6, $A[\theta]$ is formal if and only if (\ref{necessary equation for formal}) holds for all $x_i,y_i,z$.
\end{thm}

\begin{proof}
The 2-dimensional case follows from Example \ref{ex-projective} and Corollary \ref{necessary equation for formal kahler}. When the genus of $M$ is greater than or equal to 1, $\omega$ is reducible and $M$ is K\"ahler. So (\ref{necessary equation for formal}) does not hold and $A[\theta]$ is non-formal. When the genus is 0, $PA=\langle\, 1 \,\rangle$. So $(\ref{necessary equation for formal})$ trivially holds and $A[\theta]$ is formal.

When $M$ is 6-dimensional, we will construct an $A_\infty$-quasi-isomorphism $f:H^*\big(A[\theta]\big)\to A[\theta]$. The 4-dimensional case is similar. We just give the definition of $f$ here, and will give the proof that it is indeed an $A_\infty$-quasi-isomorphism in the Appendix.

Choose a basis such that $H^1\big(A[\theta]\big)\cdot H^1\big(A[\theta]\big)=\langle\, [a_i^{[2]}b_i^{[2]}] \,\rangle$, where $1\leq i\leq \dim \left\{ H^1\big(A[\theta]\big)\cdot H^1\big(A[\theta]\big) \right\}$, $a_i^{[2]},b_i^{[2]}\in PA^1$. Then expand this basis such that $H^2\big(A[\theta]\big)=\langle\, [a_i^{[2]}b_i^{[2]}] \,\rangle \oplus \langle\,[y_j^{(2)}]\,\rangle$ with $y_j^{(2)}\in PA^2$. Set $x_i^{(2)}=a_i^{[2]}b_i^{[2]}-\omega L^{-1}(a_i^{[2]}b_i^{[2]})$. Then we also have $PA^2=\langle\, x_i^{(2)} \,\rangle \oplus \langle\, y_j^{(2)} \,\rangle$.

By Poincar\'e duality, for each $x_k^{(2)}$ there exists some $z_k\in A^4$ such that $x_i^{(2)} z_k=\delta_{ik}\omega^3$ and $y_j^{(2)} z_k=0$. Let $(x_k^{(2)})^*$ be the projection of $z_k$ to $\omega PA^2$. Since $\omega^2\wedge PA^2=0$, $(x_k^{(2)})^*$ is also orthogonal to other $x_i^{(2)},y_j^{(2)}$ and satisfies $x_k^{(2)}(x_k^{(2)})^*=\omega^3$. Similarly we can define $(y_j^{(2)})^*\in \omega PA^2$. The choices of $(x_i^{(2)})^*,(y_j^{(2)})^*$ are unique because $\dim PA^2=\dim\,\omega PA^2$.

Similarly, we can choose a basis such that $H^2\big(A[\theta]\big)\cdot H^1\big(A[\theta]\big)=\langle\, [a_i^{[3]}b_i^{[3]}] \,\rangle$ with $a_i^{[3]}\in PA^2$ and $b_i^{[3]}\in PA^1$. Then expand this basis such that $H^3\big(A[\theta]\big)=\langle\, [a_i^{[3]}b_i^{[3]}] \,\rangle \oplus \langle\,[y_j^{(3)}]\,\rangle$ with $y_j^{(3)}\in PA^3$. Set $x_i^{(3)}=a_i^{[3]}b_i^{[3]}-\omega L^{-1}(a_i^{[3]}b_i^{[3]})$ and we have $PA^3=\langle\, x_i^{(3)} \,\rangle \oplus \langle\, y_j^{(3)} \,\rangle$. We can define $(x_i^{(3)})^*,(y_j^{(3)})^*$ in a same way such that they are orthogonal to all other $x_i^{(3)},y_j^{(3)}$ except for $x_i^{(3)}(x_i^{(3)})^*=y_j^{(3)}(y_j^{(3)})^*=\omega^3$.

Now we can start define $f:H^*\big(A[\theta]\big)\to A[\theta]$. For $z^{(r)}\in PA^r$, set
\begin{align*}
f_1\big([z^{(r)}]\big)=z^{(r)}, \quad f_1\big([\theta\omega^{n-r}z^{(r)}]\big)=\theta\omega^{n-r}z^{(r)}.
\end{align*}

Next we define $f_2$. For $z^{(r)}\in PA^r$, $w^{(s)}\in PA^s$ with $r+s\leq 4$ we set
$$
f_2\big([z^{(r)}],[w^{(s)}]\big)=-\theta L^{-1}(z^{(r)} w^{(s)}).
$$

When acting on $H^3\big(A[\theta]\big)\otimes H^2\big(A[\theta]\big)$, for $z^{(3)}\in PA^3$, we set
\begin{align}
f_2\big([z^{(3)}],[x_i^{(2)}]\big)=\theta\Big( z^{(3)} L^{-1}(a_i^{[2]}b_i^{[2]})-L^{-1}(z^{(3)} a_i^{[2]})b_i^{[2]} \Big).
\end{align}
For $z^{(2)}\in PA^2$, we set
\begin{align}
f_2\big([x_i^{(3)}],[z^{(2)}]\big)=\theta\Big( L^{-1}(b_i^{[3]}a_i^{[3]})z^{(2)}-b_i^{[3]}L^{-1}(a_i^{[3]}z^{(2)}) \Big).
\end{align}
For irreducible generators $y_j^{(3)}\in PA^3$ and $y_{k}^{(2)}\in PA^2$, set
$$
f_2\big([y_j^{(3)}],[y_k^{(2)}]\big)=-\theta L^{-1}(y_j^{(3)} y_k^{(2)})-\sum_i L^{-3}\Big( -y_j^{(3)} f_2\big([y_k^{(2)}],[x_i^{(3)}]\big) \Big)(x_i^{(3)})^*.
$$
Here $L^{-3}$ means $(L^{-1})^3$ and sends $\theta\omega^3$ to $\theta$.

Similarly when acting on $H^2\big(A[\theta]\big)\otimes H^3\big(A[\theta]\big)$ we define
$$
f_2\big([z^{(2)}],[x_i^{(3)}]\big)=\theta\Big( z^{(2)} L^{-1}(a_i^{[3]}b_i^{[3]})-L^{-1}(z^{(2)} a_i^{[3]})b_i^{[3]} \Big)
$$
for $z^{(2)}\in PA^2$,
$$
f_2\big([x_i^{(2)}],[z^{(3)}]\big)=\theta\Big( L^{-1}(a_i^{[2]}b_i^{[2]})z^{(3)}-a_i^{[2]}L^{-1}(b_i^{[2]}z^{(3)}) \Big)
$$
for $z^{(3)}\in PA^3$, and
$$
f_2\big([y_j^{(2)}],[y_k^{(3)}]\big)=-\theta L^{-1}(y_j^{(2)} y_k^{(3)})-\sum_i L^{-3}\Big( -f_2\big([x_i^{(3)}],[y_j^{(2)}]\big) y_k^{(3)} \Big)(x_i^{(3)})^*.
$$

When acting on $H^3\big(A[\theta]\big)\otimes H^3\big(A[\theta]\big)$, for $z^{(3)}\in PA^3$ we define
\begin{align*}
f_2\big([z^{(3)}],[x_i^{(3)}]\big)=\theta\Big( z^{(3)} L^{-1}(b_i^{[3]}a_i^{[3]})-L^{-1}(z^{(3)} b_i^{[3]})a_i^{[3]} \Big), \\
f_2\big([x_i^{(3)}],[z^{(3)}]\big)=\theta\Big( L^{-1}(a_i^{[3]}b_i^{[3]})z^{(3)}-a_i^{[3]}L^{-1}(b_i^{[3]}z^{(3)}) \Big),
\end{align*}
and
$$
f_2\big([y_j^{(3)}],[y_k^{(3)}]\big)=-\theta L^{-1}(y_j^{(3)} y_k^{(3)})-\sum_i L^{-3}\Big( y_j^{(3)} f_2\big([y_k^{(3)}],[x_i^{(2)}]\big) \Big)(x_i^{(2)})^*.
$$

Finally, we set $f_2=0$ when acting on $H^r\big(A[\theta]\big)\otimes H^s\big(A[\theta]\big)$ and one of $r,s\geq 4$. For $p\geq3$, set $f_p=0$. By straightforward calculation we will see that such $f$ is an $A_\infty$-quasi-isomorphism.
\end{proof}

\begin{rmk}
In the proof of Theorem \ref{condition equivalent to formal}, we split $H^*(A)$ as the subspace of reducible cohomology classes and its complement, then define $f_2$ acting on them separately. This works for $\dim M\leq 6$. But when $\dim M=8$, the reducible cohomology classes of degree 4 contains two subspaces $H^2\cdot H^2$ and $H^3\cdot H^1$, these two subspace may have a non-trivial intersection. This makes defining $f_2$ much more complicated. So we have to find other ways to generalize this theorem.
\end{rmk}

Observe that equation (\ref{necessary equation for formal}) is a special case of the vanishing generalized Massey product (c.f. \cite{bt2}). Also all generalized Massey products vanishing is a necessary condition for formal. So we have the following corollary.

\begin{cor}
When the dimension of $M$ is not greater than 6, $A[\theta]$ is formal if and only if its generalized Massey products all vanish.
\end{cor}

Therefore, we have following theorem.

\begin{thm}\label{condition equivalent to formal geometircally}
Let $(M^{2n},\omega)$ be a formal symplectic manifold satisfying the hard Lefschetz property. Suppose $\omega$ is integral, then there exists a circle bundle $X$ over $M$ whose Euler class is $\omega$. In the following statements, (1) implies (2), and (2) implies (3). Moreover, when $\dim M\leq 6$, (3) also implies (1).
\begin{enumerate}[(1)]
\item X is formal.
\item All generalized Massey products of $\Omega^*(X)$ vanish.
\item For arbitrary $x_1,\ldots,x_k\in PH^r(M)$, $y_1,\ldots y_k\in PH^s(M)$ satisfying $r+s\leq n+1$ and
$$
\sum_{i=1}^{k}x_iy_i=\omega\alpha
$$
for some $\alpha\in H^{r+s-2}(M)$, and for arbitrary $z\in PH^{n+1-s}(M)$, we have
$$
\sum_{i=1}^{k}x_iL^{-1}(y_iz)=\sum_{i=1}^{k}L^{-1}(x_iy_i)z.
$$
\end{enumerate} 
\end{thm}

\begin{ex}
Let $(\Sigma_1,\omega_1),(\Sigma_2,\omega_2)$ be two Riemann surfaces, where $\omega_1,\omega_2$ are their volume form respectively. Let $M=\Sigma_1\times\Sigma_2$ and $\omega=\omega_1+\omega_2$ (Here $\omega_1,\omega_2$ are actually the pullback of $\omega_1,\omega_2$ associated with the projection. For simplicity we omit the pullback sign). Then $(M,\omega)$ is a K\"ahler manifold. So it is formal and satisfies the hard Lefschetz property.

If the genus of both $\Sigma_1$ and $\Sigma_2$ are at least 1, $\omega_1$ and $\omega_2$ are reducible. Then $\omega$ is also reducible. By Corollary \ref{necessary equation for formal kahler}, the circle bundle $X$ over $M$ with Euler class $\omega$ is non-formal.

If one of the genus is 0, we will show that $X$ is formal. Without loss of generality, suppose the genus of $\Sigma_2$ is 0. Choose a basis of $H^1(\Sigma_1)$ such that
\begin{align*}
& H^1(\Sigma_1)=\langle\, a_1,b_1,\ldots,a_k,b_k \,\rangle,\\
& a_ib_j=\delta_{ij}\omega_1, \quad a_ia_j=b_ib_j=0
\end{align*}
for all $1\leq i,j\leq k$, where $k$ is the genus of $\Sigma_1$. Let $\bar{\omega}=\omega_1-\omega_2$. Then the ring structure of $A=H^*(M)$ can be described as follows.
\begin{center}
\begin{tabular}{l l}
$A^0=\langle\, 1 \,\rangle,$ & $A^1=\langle\, a_1,b_1,\ldots,a_k,b_k \,\rangle, \quad A^2=\langle\, \omega,\bar{\omega} \,\rangle,$ \\
$A^3=\langle\, \omega a_1,\omega b_1,\ldots,\omega a_k,\omega b_k \,\rangle, $ & $A^4=\langle\, \omega^2 \,\rangle.$ \\
$\omega a_i=\omega_2 a_i=-\bar{\omega} a_i,$ & $\omega b_i=\omega_2 b_i=-\bar{\omega} b_i,$ \\
$\bar{\omega}^2=-2\omega_1\omega_2=-\omega^2,$ & $\omega\bar{\omega}=0.$
\end{tabular}
\end{center}
The last equation implies that $\bar{\omega}$ is the generator of $PA^2$.

To prove that $X$ is formal, we will verify that (\ref{necessary equation for formal}) holds for all $x_1,\ldots,x_l\in PA^r$, $y_1,\ldots y_l\in PA^s,z\in PA^t$ satisfying the following conditions: $r+s\leq 3,s+t=3$, and $\sum x_iy_i$ is in the ideal generated by $\omega$.

When $s=1$, we have $t=2$. So we can assume $z=\bar{\omega}$. In this case $r$ can be 1 or 2. We first discuss the case $r=1$. Since the product of any two elements in $A^1$ is proportional to $\omega_1=\frac{1}{2}(\omega+\bar{\omega})$, $\sum x_iy_i$ must be 0 if it is in the ideal generated by $\omega$. On the other hand, $y_iz=y_i\bar{\omega}=-y_i\omega$. Thus
$$
\sum_{i=1}^l x_i L^{-1}(y_i z) = \sum_{i=1}^l x_i (-y_i) = 0 = \sum_{i=1}^l L^{-1}(x_i y_i) z.
$$

For the case $r=2$, since the product of two elements in $PA^1$ and $PA^2$ is always in the ideal generated to $\omega$, we only need to verify that (\ref{necessary equation for formal}) holds for $l=1$, i.e. 
$$
x_1 L^{-1}(y_1 z) = L^{-1}(x_1 y_1) z.
$$
As $x_1,z\in PA^2$, we can assume that they are both $\bar{\omega}$. Then 
$$
x_1 L^{-1}(y_1 z) = \bar{\omega} L^{-1}(y_1 \bar{\omega}) = L^{-1}(\bar{\omega} y_1) \bar{\omega} = L^{-1}(x_1 y_1) z.
$$

When $s=2$, $r$ and $t$ can only be 1. Same as the case above, we only need to verify that (\ref{necessary equation for formal}) holds for $l=1$. Suppose $y_1=\bar{\omega}$. Then
$$
x_1 L^{-1}(y_1 z) = x_1 L^{-1}(\bar{\omega} z) = x_1 (-z) = (-x_1)z = L^{-1}(x_1 \bar{\omega}) z = L^{-1}(x_1 y_1) z.
$$
Therefore, $X$ is formal.
\end{ex}

\begin{ex}
Let $M=S^2\times S^2\times S^2$, and $\omega_1,\omega_2,\omega_3$ be the corresponding volume form of each $S^2$. Set $\omega=\omega_1+\omega_2+\omega_3$. Then $(M,\omega)$ is a K\"ahler manifold.

Let $a=\omega_1-\omega_3$ and $b=\omega_2-\omega_3$. As $\omega_1^2=\omega_2^2=\omega_3^2=0$, $\omega^2 a=2\omega_2\omega_3\omega_1-2\omega_1\omega_2\omega_3=0$. Thus, $a\in PA^2$. Similarly $b\in PA^2$. We will show that $aL^{-1}(ab) \neq L^{-1}(a^2)b$.

For elements in $A^4$, we have
\begin{center}
\begin{tabular}{l l}
$\omega^2=2\omega_1\omega_2+2\omega_1\omega_3+2\omega_2\omega_3,$
& $\omega a=\omega_1\omega_2-\omega_2\omega_3,$ \\
 $\omega b=\omega_1\omega_2-\omega_1\omega_3,$ 
& $a^2=-2\omega_1\omega_3,$ \\
 $ab=\omega_1\omega_2-\omega_1\omega_3-\omega_2\omega_3,$
& $b^2=-2\omega_2\omega_3.$ 
\end{tabular}
\end{center}

So
\begin{align*}
& ab=\omega_1\omega_2-\omega_1\omega_3-\omega_2\omega_3=-\frac{1}{6}\omega^2+\frac{2}{3}\omega a+\frac{2}{3}\omega b, \\
& aL^{-1}(ab)=a(-\frac{1}{6}\omega+\frac{2}{3}a+\frac{2}{3}b)=\frac{1}{2}\omega_1\omega_2-2\omega_1\omega_3-\frac{1}{2}\omega_2\omega_3.
\end{align*}
On the other hand,
\begin{align*}
& a^2=-2\omega_1\omega_3=-\frac{1}{3}\omega^2-\frac{2}{3}\omega a+\frac{4}{3}\omega b, \\
& L^{-1}(a^2)b=(-\frac{1}{3}\omega-\frac{2}{3}a+\frac{4}{3}b)b=-\omega_1\omega_2+\omega_1\omega_3-2\omega_2\omega_3.
\end{align*}
Therefore, The circle bundle $X$ over $M$ with Euler class $\omega$ is non-formal.

This example shows that $\omega$ being irreducible, even if $M$ is simply connected, is not enough to guarantee that $X$ is formal.
\end{ex}

\appendix
\section{Proof of Theorem \ref{condition equivalent to formal}}
Here, we will show that $f:H^*\big( A[\theta] \big)\to A[\theta]$ defined in the proof of Theorem \ref{condition equivalent to formal} is an $A_\infty$-quasi-isomorphism. We assume $\dim M=6$.

\begin{proof}
Recall that we choose bases of $H^2\big(A[\theta]\big)$ and $H^3\big(A[\theta]\big)$ satisfying
$$
H^2\big(A[\theta]\big)=\langle\, [x_i^{(2)}] \,\rangle \oplus \langle\,[y_j^{(2)}]\,\rangle, \quad H^3\big(A[\theta]\big)=\langle\, [x_i^{(3)}] \,\rangle \oplus \langle\,[y_j^{(3)}]\,\rangle.
$$
Here $x_i^{(r)},y_i^{(r)}\in PA^r$ for $r=2,3$. All $[x_i^{(r)}]=[a_i^{[r]}b_i^{[r]}]$ are generators of the subspace of reducible cohomologies, and $[y_i^{(r)}]$ is irreducible. The degree of $[a_i^{[3]}]$ is 2, and the degree of $[b_i^{[3]}],[a_i^{[2]}],[b_i^{[2]}]$ are 1. We use $(x_i^{(r)})^*,(y_j^{(r)})^*\in\omega^{3-r}PA^r$ denote the dual of $x_i^{(r)},y_i^{(r)}$ respectively corresponding to these bases.

$f_1$ is defined as follows. For $z^{(r)}\in PA^r$, set
\begin{align*}
f_1\big([z^{(r)}]\big)=z^{(r)}, \quad f_1\big([\theta\omega^{n-r}z^{(r)}]\big)=\theta\omega^{n-r}z^{(r)}.
\end{align*}

$f_2$ is defined as follows. For $z^{(r)}\in PA^r$, $w^{(s)}\in PA^s$ with $r+s\leq 4$ we set
$$
f_2\big([z^{(r)}],[w^{(s)}]\big)=-\theta L^{-1}(z^{(r)} w^{(s)}).
$$

When acting on $H^3\big(A[\theta]\big)\otimes H^2\big(A[\theta]\big)$, we define
\begin{align}\label{f2(H3,H1.H1)}
f_2\big([z^{(3)}],[x_i^{(2)}]\big)=\theta\Big( z^{(3)} L^{-1}(a_i^{[2]}b_i^{[2]})-L^{-1}(z^{(3)} a_i^{[2]})b_i^{[2]} \Big)
\end{align}
for $z^{(3)}\in PA^3$,
\begin{align}\label{f2(H2.H1,H2)}
f_2\big([x_i^{(3)}],[z^{(2)}]\big)=\theta\Big( L^{-1}(b_i^{[3]}a_i^{[3]})z^{(2)}-b_i^{[3]}L^{-1}(a_i^{[3]}z^{(2)}) \Big)
\end{align}
For $z^{(2)}\in PA^2$, and
$$
f_2\big([y_j^{(3)}],[y_k^{(2)}]\big)=-\theta L^{-1}(y_j^{(3)} y_k^{(2)})-\sum_i L^{-3}\Big( -y_j^{(3)} f_2\big([y_k^{(2)}],[x_i^{(3)}]\big) \Big)(x_i^{(3)})^*.
$$
Here $L^{-3}$ sends $\theta\omega^3$ to $\theta$.

When acting on $H^2\big(A[\theta]\big)\otimes H^3\big(A[\theta]\big)$ we define
$$
f_2\big([z^{(2)}],[x_i^{(3)}]\big)=\theta\Big( z^{(2)} L^{-1}(a_i^{[3]}b_i^{[3]})-L^{-1}(z^{(2)} a_i^{[3]})b_i^{[3]} \Big)
$$
for $z^{(2)}\in PA^2$,
$$
f_2\big([x_i^{(2)}],[z^{(3)}]\big)=\theta\Big( L^{-1}(a_i^{[2]}b_i^{[2]})z^{(3)}-a_i^{[2]}L^{-1}(b_i^{[2]}z^{(3)}) \Big)
$$
for $z^{(3)}\in PA^3$, and
$$
f_2\big([y_j^{(2)}],[y_k^{(3)}]\big)=-\theta L^{-1}(y_j^{(2)} y_k^{(3)})-\sum_i L^{-3}\Big( -f_2\big([x_i^{(3)}],[y_j^{(2)}]\big) y_k^{(3)} \Big)(x_i^{(3)})^*.
$$

When acting on $H^3\big(A[\theta]\big)\otimes H^3\big(A[\theta]\big)$, for $z^{(3)}\in PA^3$ we define
\begin{align*}
f_2\big([z^{(3)}],[x_i^{(3)}]\big)=\theta\Big( z^{(3)} L^{-1}(b_i^{[3]}a_i^{[3]})-L^{-1}(z^{(3)} b_i^{[3]})a_i^{[3]} \Big), \\
f_2\big([x_i^{(3)}],[z^{(3)}]\big)=\theta\Big( L^{-1}(a_i^{[3]}b_i^{[3]})z^{(3)}-a_i^{[3]}L^{-1}(b_i^{[3]}z^{(3)}) \Big),
\end{align*}
and
$$
f_2\big([y_j^{(3)}],[y_k^{(3)}]\big)=-\theta L^{-1}(y_j^{(3)} y_k^{(3)})-\sum_i L^{-3}\Big( y_j^{(3)} f_2\big([y_k^{(3)}],[x_i^{(2)}]\big) \Big)(x_i^{(2)})^*.
$$

When acting on $H^r\big(A[\theta]\big)\otimes H^s\big(A[\theta]\big)$ and one of $r,s\geq 4$, we set $f_2=0$. For $p\geq3$, set $f_p=0$.\\[0.1in]

\textbf{f is well-defined.}

First we need to verify that $f$ is well-defined. Namely, when there are two different ways of defining $f_2([z],[w])$, they should be compatible.

For example, $f_2([x_i^{(3)}],[x_k^{(2)}])$ are defined by both (\ref{f2(H3,H1.H1)}) and (\ref{f2(H2.H1,H2)}). Under these definitions, we have
\begin{align*}
f_2([x_i^{(3)}],[x_k^{(2)}]) &= \theta\Big( x_i^{(3)} L^{-1}(a_k^{[2]}b_k^{[2]})-L^{-1}(x_i^{(3)} a_k^{[2]})b_k^{[2]} \Big) \\
&= \theta\Big( a_i^{[3]}b_i^{[3]} L^{-1}(a_k^{[2]}b_k^{[2]})-\omega L^{-1}(a_i^{[3]}b_i^{[3])}) L^{-1}(a_k^{[2]}b_k^{[2]})\\
&\quad\quad -L^{-1}(a_i^{[3]}b_i^{[3]} a_k^{[2]})b_k^{[2]} + L^{-1}\big( \omega L^{-1}(a_i^{[3]}b_i^{[3]}) a_k^{[2]} \big)b_k^{[2]}\Big),
\end{align*}
and
\begin{align*}
f_2\big([x_{(i)}^3],[x_k^{(2)}]\big) &= \theta\Big( L^{-1}(b_i^{[3]}a_i^{[3]})x_k^{(2)}-b_i^{[3]}L^{-1}(a_i^{[3]}x_k^{(2)}) \Big) \\
&= \theta\Big( L^{-1}(b_i^{[3]}a_i^{[3]})a_k^{[2]}b_k^{[2]}-L^{-1}(b_i^{[3]}a_i^{[3]})\omega L^{-1}(a_k^{[2]}b_k^{[2]}) \\
&\quad\quad -b_i^{[3]}L^{-1}(a_i^{[3]}a_k^{[2]}b_k^{[2]})+b_i^{[3]}L^{-1}\big(a_i^{[3]} \omega L^{-1}(a_k^{[2]}b_k^{[2]}) \big) \Big).
\end{align*}
Compare corresponding terms. The degree of $a_i^{[3]}a_k^{[2]}$ is 3, and the degree of $b_i^{[3]}$ and $b_k^{[2]}$ are 1. By assumption equation (\ref{necessary equation for formal}) holds. So we have
$$
L^{-1}(a_i^{[3]}b_i^{[3]} a_k^{[2]})b_k^{[2]} = L^{-1}(b_i^{[3]}a_i^{[3]} a_k^{[2]})b_k^{[2]}=b_i^{[3]}L^{-1}(a_i^{[3]}a_k^{[2]}b_k^{[2]}).
$$
Also, the degree of $L^{-1}(a_i^{[3]}b_i^{[3]}) a_k^{[2]}$ is 2. Since $L^{-1}$ and $\omega$ are isomorphisms between $A^2$ and $A^4$, we have
$$
L^{-1}\big( \omega L^{-1}(a_i^{[3]}b_i^{[3]}) a_k^{[2]} \big)b_k^{[2]} = \big( L^{-1}(a_i^{[3]}b_i^{[3]}) a_k^{[2]} \big)b_k^{[2]} = L^{-1}(b_i^{[3]}a_i^{[3]}) a_k^{[2]}b_k^{[2]}.
$$
Similarly
$$
a_i^{[3]}b_i^{[3]} L^{-1}(a_k^{[2]}b_k^{[2]}) = b_i^{[3]}L^{-1}\big(a_i^{[3]} \omega L^{-1}(a_k^{[2]}b_k^{[2]})\big).
$$
Thus, the two definitions agree.

In the same way, we can verify that the definitions of $f_2\big([x_i^{(2)}],[x_k^{(3)}]\big)$ and $f_2\big([x_i^{(3)}],[x_k^{(3)}]\big)$ are also compatible. So $f_2$ is well-defined. \\[0.1in]

\textbf{$f$ is a quasi-isomorphism.}

Next we prove that $f$ is a quasi-isomorphism. By the definition of $f_1$, it is clear that $m_1f_1=f_1m_1=0$ and $f_1^*$ is an isomorphism. So it remains to show
$$
\sum_{r+s+t=p}(-1)^{r+st}f_{r+t+1}(\textbf{1}^{\otimes r}\otimes m_s\otimes\textbf{1}^{\otimes r})=\sum_{i_1+\ldots+i_r=p}(-1)^sm_r(f_{i_1}\otimes f_{i_2}\otimes\ldots\otimes f_{i_r})
$$
for $p\geq 2$.

\textbf{Case $p=2$}.

When we $p=2$, the equation becomes $f_1m_2=m_1f_2+m_2(f_1\otimes f_1)$. For $z^{(r)}\in PA^r$, $w^{(s)}\in PA^s$ and $r+s\leq 4$,
\begin{align*}
m_1f_2\big([z^{(r)}],[w^{(s)}]\big) &= -\omega L^{-1}(z^{(r)} w^{(s)}) \\
&= (1-\omega L^{-1})(z^{(r)} w^{(s)})-z^{(r)} w^{(s)} \\
&= \big( f_1m_2-m_2(f_1\otimes f_1)\big) \big([z^{(r)}],[w^{(s)}]\big).
\end{align*}

For $z^{(3)}\in PA^3$,
$$
m_1f_2\big([z^{(3)}],[x_i^{(2)}]\big)=\omega\Big( z^{(3)}L^{-1}(a_i^{[2]}b_i^{[2]})-L^{-1}(z^{(3)} a_i^{[2]})b_i^{[2]} \Big).
$$
As the degree of $z^{(3)} a_i^{[2]}$ is 4, $\omega L^{-1}(z^{(3)} a_i^{[2]})=z^{(3)} a_i^{[2]}$. On the other hand,
$$
\omega z^{(3)} L^{-1}(a_i^{[2]}b_i^{[2]}) = z^{(3)}a_i^{[2]}b_i^{[2]}-z^{(3)}(1-\omega L^{-1})(a_i^{[2]}b_i^{[2]}) = z^{(3)}a_i^{[2]}b_i^{[2]}-f_1\big([z^{(3)}]\big) f_1\big([x_i^{(2)}]\big).
$$
Then
$$
m_1f_2\big([z^{(3)}],[x_i^{(2)}]\big) = -f_1\big([z^{(3)}]\big) f_1\big([x_i^{(2)}]\big) =\big( f_1m_2-m_2(f_1\otimes f_1)\big)\big([z^{(3)}],[x_i^{(2)}]\big)
$$
because $[z^{(3)}x_i^{(2)}]=0$.

Follow the same way we have $f_2\big([x_i^{(3)}],[z^{(2)}]\big)=\big( f_1m_2-m_2(f_1\otimes f_1)\big)\big([x_i^{(3)}],[z^{(2)}]\big)$ for $z^{(2)}\in PA^2$.

Next we consider 
$$
f_2\big([y_j^{(3)}],[y_k^{(2)}]\big)=-\theta L^{-1}(y_j^{(3)} y_k^{(2)})-\sum_i L^{-3}\Big( -y_j^{(3)} f_2\big([y_k^{(2)}],[x_i^{(3)}]\big) \Big)(x_i^{(3)})^*.
$$
Since $L^{-3}\Big( -y_j^{(3)} f_2\big([y_k^{(2)}],[x_i^{(3)}]\big) \Big)$ is some constant times $\theta$, $m_1$ acting on the second term is proportional to $\omega (x_i^{(3)})^*$. But $(x_i^{(3)})^*\in PA^3$. So $\omega (x_i^{(3)})^*=0$ and
$$
m_1f_2\big([y_j^{(3)}],[y_k^{(2)}]\big)=-\omega L^{-1}(y_j^{(3)} y_k^{(2)}).
$$
As $y_j^{(3)} y_k^{(2)}\in A^5$ and $L^{-1}:A^5\to A^3$ is injective, $-\omega L^{-1}(y_j^{(3)} y_k^{(2)})=-y_j^{(3)} y_k^{(2)}$. Also $[y_j^{(3)} y_k^{(2)}]=0$. Therefore, we have
$$
m_1f_2\big([y_j^{(3)}],[y_k^{(2)}]\big) = \big( f_1m_2-m_2(f_1\otimes f_1)\big) \big([y_j^{(3)}],[y_k^{(2)}]\big).
$$

Similarly we can show that $f_1m_2=m_1f_2+m_2(f_1\otimes f_1)$ holds when acting on $H^2\big(A[\theta]\big)\otimes H^3\big(A[\theta]\big)$ and $H^3\big(A[\theta]\big)\otimes H^3\big(A[\theta]\big)$.

For any $z^{(r)}\in PA^r$ and $w^{(s)}\in PA^s$, $f_2\big([\theta\omega^{n-r} z^{(r)}],[w^{(s)}]\big)=0$. When $r<s$, the total degree of $\theta\omega^{n-r} z^{(r)} w^{(s)}$ is greater than $2n+1$. So $f_1\big([\theta\omega^{n-r} z^{(r)} w^{(s)}]\big)=f_1\big([\theta\omega^{n-r}z^{(r)}]\big) f_1\big([w^{(s)}]\big)=0$.

When $r\geq s$, we claim that $\omega^{n-r} z^{(r)} w^{(s)}\in \omega^{n-r+s}PA^{r-s}$. Observe that $\omega^{n-r+s}PA^{r-s}$ is the kernel of $\omega: A^{2n-r+s}\to A^{2n-r+s+2}$. Also $\omega(\omega^{n-r} z^{(r)} w^{(s)})=(\omega^{n-r+1} z^{(r)})w^{(s)}=0$. So $f_1\big([\theta\omega^{n-r} z^{(r)} w^{(s)}]\big) = \theta\omega^{n-r} z^{(r)} w^{(s)} = f_1\big([\theta\omega^{n-r}z^{(r)}]\big) f_1\big([w^{(s)}]\big)$.

This proves $f_1m_2=m_1f_2+m_2(f_1\otimes f_1)$ holds when acting on $H^r\big(A[\theta]\big)\otimes H^s\big(A[\theta]\big)$ with $r\geq 4, s\leq 3$. A same discussion shows the equation also holds for the case $r\leq 3, s\geq 4$. Finally, when $r,s\geq 4$, the total degree of both sides are at least 8. So they must be 0 and the equation trivially holds. \\[0.1in]

\textbf{Properties of $f_2$.}

Before talk about the case $p=3$, we go through some properties of $f_2$. The first one is that $f_2$ is graded commutative. Since $m_2$ is graded commutative, it follows from the definition directly that
\begin{align*}
f_2\big([z^{(r)}],[w^{(s)}]\big) &= (-1)^{rs}f_2\big([w^{(s)}],[z^{(r)}]\big), \text{ when }  r+s\leq 4, \\
f_2\big([x_i^{(3)}],[z^{(2)}]\big) &= f_2\big([z^{(2)}],[x_i^{(3)}]\big) \\
f_2\big([x_i^{(3)}],[z^{(3)}]\big) &= f_2\big([z^{(3)}],[x_i^{(3)}]\big).
\end{align*}

By equation (\ref{necessary equation for formal}),
$$
L^{-1}(z^{(3)} a_i^{[2]})b_i^{[2]} = -L^{-1}(a_i^{[2]} z^{(3)})b_i^{[2]} = -a_i^{[2]}L^{-1}(z^{(3)} b_i^{[2]}) = a_i^{[2]}L^{-1}(b_i^{[2]} z^{(3)}).
$$
Then
$$
f_2\big([z^{(3)}],[x_i^{(2)}]\big) = f_2\big([x_i^{(2)}],[z^{(3)}]\big).
$$

A similar discussion shows that
$$
f_2\big([z^{(3)}],[x_i^{(3)}]\big) = -f_2\big([x_i^{(3)}],[z^{(3)}]\big).
$$

We have proved that $f_2\big([y_k^{(2)}],[x_i^{(3)}]\big) = f_2\big([x_i^{(3)}],[y_k^{(2)}]\big)$. It follows that
$$
L^{-3}\Big( -y_j^{(3)} f_2\big([y_k^{(2)}],[x_i^{(3)}]\big) \Big)(x_i^{(3)})^* = L^{-3}\Big( -f_2\big([x_i^{(3)}],[y_k^{(2)}]\big) y_j^{(3)} \Big)(x_i^{(3)})^*.
$$
Thus,
$$
f_2\big([y_j^{(3)}],[y_k^{(2)}]\big) = f_2\big([y_k^{(2)}],[y_j^{(3)}]\big).
$$

Similarly, we prove $f_2\big([y_j^{(3)}],[y_k^{(3)}]\big) = -f_2\big([y_k^{(3)}],[y_j^{(3)}]\big)$ by showing that
$$
y_j^{(3)} f_2\big([y_k^{(3)}],[x_i^{(2)}]\big) = -y_k^{(3)} f_2\big([y_j^{(3)}],[x_i^{(2)}]\big).
$$
This follows from the definition of $f_2$ and (\ref{necessary equation for formal}).
\begin{align*}
y_j^{(3)} f_2\big([y_k^{(3)}],[x_i^{(2)}]\big) &= y_j^{(3)} \theta\Big( y_k^{(3)} L^{-1}(a_i^{[2]}b_i^{[2]})-L^{-1}(y_k^{(3)} a_i^{[2]})b_i^{[2]} \Big) \\
&= -\theta y_j^{(3)} y_k^{(3)} L^{-1}(a_i^{[2]}b_i^{[2]})+\theta L^{-1}(y_k^{(3)} a_i^{[2]}) y_j^{(3)} b_i^{[2]} \\
&= -y_k^{(3)}\theta y_j^{(3)} L^{-1}(a_i^{[2]}b_i^{[2]})+\theta y_k^{(3)} L^{-1}(a_i^{[2]}y_j^{(3)}) b_i^{[2]} \\
&= -y_k^{(3)} f_2\big([y_j^{(3)}],[x_i^{(2)}]\big).
\end{align*}

In the remaining cases $f_2=0$. So it is graded commutative. \\[0.1in]

Besides, we can generalize equation (\ref{f2(H3,H1.H1)}) as follows. For arbitrary $z^{(3)}\in PA^3$ and $w^{(1)},v^{(1)}\in PA^1$, we have
\begin{align}\label{f2(H3,H1.H1) generalized}
f_2\big([z^{(3)}],[w^{(1)} v^{(1)}]\big)=\theta\Big( z^{(3)}L^{-1}(w^{(1)} v^{(1)})-L^{-1}(z^{(3)} w^{(1)})v^{(1)} \Big).
\end{align}

Since $[w^{(1)} v^{(1)}]\in H^1\big(A[\theta]\big)\cdot H^1\big(A[\theta]\big)$, there exist constant numbers $\lambda_i$ such that
$$
[w^{(1)} v^{(1)}]=\sum_i \lambda_i[x_i^{(2)}].
$$
By the definition of $f_2$, we have
$$
f_2\big([z^{(3)}],[w^{(1)} v^{(1)}]\big) = \sum_i \lambda_i f_2\big([z^{(3)}],[x_i^{(2)}] \big) = \sum_i \lambda_i\theta\Big( z^{(3)}L^{-1}(a_i^{[2]}b_i^{[2]})-L^{-1}(z^{(3)} a_i^{[2]})b_i^{[2]} \Big).
$$
On the other hand, as $\sum\lambda_i(a_i^{[2]}b_i^{[2]})-w^{(1)}v^{(1)}$ is exact, it is in the ideal generated by $\omega$. So are all $z^{(3)} a_i^{[2]}$. So we can apply equation (\ref{necessary equation for formal}) and get
$$
\Big( \sum_i \lambda_i z^{(3)} L^{-1}(a_i^{[2]}b_i^{[2]}) \Big)-z^{(3)} L^{-1}(w^{(1)} v^{(1)}) = \Big( \sum_i \lambda_i L^{-1}(z^{(3)} a_i^{[2]})b_i^{[2]} \Big)-L^{-1}(z^{(3)} w^{(1)})v^{(1)}.
$$
This implies (\ref{f2(H3,H1.H1) generalized}).

Similarly, we have
\begin{align}\label{f2(H2,H2.H1) generalized}
f_2\big([z^{(2)}],[w^{(2)} v^{(1)}]\big)=\theta\Big( z^{(2)} L^{-1}(w^{(2)} v^{(1)})-L^{-1}(z^{(2)} w^{(2)})v^{(1)} \Big)
\end{align}
for arbitrary $z^{(2)},w^{(2)}\in PA^2, v^{(1)}\in PA^1$, and
\begin{align}\label{f2(H3,H2.H1) generalized}
f_2\big([z^{(3)}],[w^{(2)} v^{(1)}]\big)=\theta\Big( z^{(3)} L^{-1}(v^{(1)} w^{(2)})-L^{-1}(z^{(3)} v^{(1)})w^{(2)} \Big)
\end{align}
for arbitrary $z^{(3)}\in PA^3, w^{(2)}\in PA^2, v^{(1)}\in PA^1$.

Because $f_2$ is graded commutative, we have similar statements that 
\begin{align}\label{f2(H1.H1,H3) generalized}
f_2\big([z^{(1)} w^{(1)}],[v^{(3)}]\big) = \theta\Big( L^{-1}(z^{(1)} w^{(1)})v^{(3)}-z^{(1)} L^{-1}(w^{(1)} v^{(3)}) \Big),
\end{align}

\begin{align}\label{f2(H2.H1,H2) generalized}
f_2\big([z^{(2)} w^{(1)}],[v^{(2)}]\big) = \theta\Big( L^{-1}(w^{(1)} z^{(2)})v^{(2)}-w^{(1)} L^{-1}(z^{(2)} v^{(2)}) \Big),
\end{align}
and
\begin{align}\label{f2(H2.H1,H3) generalized}
f_2\big([z^{(2)} w^{(1)}],[v^{(3)}]\big) = \theta\Big( L^{-1}(z^{(2)} w^{(1)})v^{(3)}-z^{(2)} L^{-1}(w^{(1)} v^{(3)}) \Big).
\end{align}
\\[0.1in]

\textbf{Case $p=3$.}

When $p=3$, the equation becomes
$$
m_2(f_1\otimes f_2-f_2\otimes f_1)-f_2(m_2\otimes 1-1\otimes m_2)=0.
$$
We will go through all the cases that the left-hand side of the above equation acts on $H^r\big(A[\theta]\big)\otimes H^s\big(A[\theta]\big) \otimes H^t\big(A[\theta]\big)$. Since $f_2$ is graded commutative, we only need to check the case $r\geq t$.

We first consider the case that $r,s,t\leq 3$. That is, for any $z^{(r)}\in PA^r, w^{(s)}\in PA^s, v^{(t)}\in PA^t$, we have
\begin{align}\label{equation for p=3}
\Big( m_2(f_1\otimes f_2-f_2\otimes f_1)-f_2(m_2\otimes 1-1\otimes m_2) \Big) ([z^{(r)}],[w^{(s)}],[v^{(t)}]) = 0.
\end{align}
\\[0.1in]

\noindent\textbf{1)}  When $r+s+t\leq 4$,
\begin{align}\label{m2(f1.f2-f2.f1)}
m_2(f_1\otimes f_2-f_2\otimes f_1)([z^{(r)}],[w^{(s)}],[v^{(t)}]) = \theta \Big( -z^{(r)} L^{-1}(w^{(s)} v^{(t)})+L^{-1}(z^{(r)} w^{(s)})v^{(t)} \Big).
\end{align}
On the other hand,
\begin{align*}
-f_2([z^{(r)} w^{(s)}],[v^{(t)}]) = \theta L^{-1} \Big( (1-\omega L^{-1})(z^{(r)} w^{(s)})v^{(t)} \Big) = \theta L^{-1}(z^{(r)} w^{(s)} v^{(t)})-\theta L^{-1}(z^{(r)} w^{(s)}) v^{(t)}
\end{align*}
because the degree of $L^{-1}(z^{(r)} w^{(s)})v^{(t)}$ is smaller than or equal to 2 and $L^{-1}\circ\omega:A^{\leq2}\to A^{\leq2}$ is the identity map. For the same reason
\begin{align*}
f_2([z^{(r)}],[w^{(s)} v^{(t)}]) &= -\theta L^{-1} \Big( z^{(r)}(1-\omega L^{-1})(w^{(s)} v^{(t)}) \Big) \\
&= -\theta L^{-1}(z^{(r)} w^{(s)} v^{(t)})+\theta z^{(r)} L^{-1}(w^{(s)} v^{(t)}).
\end{align*}
Add all terms together we get (\ref{equation for p=3}). \\[0.1in]

\noindent\textbf{2)}  When $r+s=s+t=4$, i.e. when $(r,s,t)=(1,3,1),(2,2,2),$ or $(3,1,3)$, equation (\ref{m2(f1.f2-f2.f1)}) also holds. But in this case, $[z^{(r)} w^{(s)}]=[w^{(s)} v^{(t)}]=0$. Apply equation (\ref{necessary equation for formal}) we have
$$
z^{(r)} L^{-1}(w^{(s)} v^{(t)})=L^{-1}(z^{(r)} w^{(s)})v^{(t)}.
$$
Thus, we get (\ref{equation for p=3}). \\[0.1in]

\noindent\textbf{3)}  When $r+s+t=5$, we have went through the case $(r,s,t)=(1,3,1)$ and the remaining cases are $(2,2,1),(3,1,1)$ and $(2,1,2)$.

\noindent\textbf{3.1)}  For the first two cases that $t=1$, equation (\ref{m2(f1.f2-f2.f1)}) still holds. As $r+s=4$, $[z^{(r)} w^{(s)}]=0$. By (\ref{f2(H3,H1.H1) generalized}) or (\ref{f2(H2,H2.H1) generalized}),
\begin{align*}
f_2\big([z^{(r)}],[w^{(s)} v^{(t)}]\big) &= \theta\Big( z^{(r)} L^{-1}(w^{(s)} v^{(t)})-L^{-1}(z^{(r)} w^{(s)})v^{(t)} \Big) \\
&= -m_2(f_1\otimes f_2-f_2\otimes f_1)([z^{(r)}],[w^{(s)}],[v^{(t)}]).
\end{align*}
Then we get (\ref{equation for p=3}). \\[0.1in]

\noindent\textbf{3.2)} For the case $(r,s,t)=(2,1,2)$, we also have equation (\ref{m2(f1.f2-f2.f1)}). By (\ref{f2(H2.H1,H2) generalized}) we have
$$
-f_2\big([z^{(2)} w^{(1)}],[v^{(2)}]\big) = -\theta\Big( L^{-1}(w^{(1)} z^{(2)})v^{(2)}-w^{(1)} L^{-1}(z^{(2)} v^{(2)}) \Big),
$$
and by (\ref{f2(H2,H2.H1) generalized})
$$
f_2\big([z^{(2)}],[w^{(1)} v^{(2)}]\big) = f_2\big([z^{(2)}],[v^{(2)} w^{(1)}]\big) = \theta\Big( z^{(2)}L^{-1}(v^{(2)} w^{(1)})-L^{-1}(z^{(2)} v^{(2)})w^{(1)} \Big).
$$
Add all terms together again we get (\ref{equation for p=3}). \\[0.1in]

\noindent\textbf{4)} When $r+s+t=6$, we have went through the case $(r,s,t)=(2,2,2)$. It remains to check $(3,1,2),(3,2,1)$, and $(2,3,1)$.

\noindent\textbf{4.1)} For $(r,s,t)=(3,1,2)$, the proof is similar as before. Equation (\ref{m2(f1.f2-f2.f1)}) holds, $[z^{(3)} w^{(1)}]=0$, and
$$
f_2\big([z^{(3)}],[w^{(1)} v^{(2)}]\big) = f_2\big([z^{(3)}],[v^{(2)} w^{(1)}]\big) = \theta\Big( z^{(3)} L^{-1}(w^{(1)} v^{(2)})-L^{-1}(z^{(3)} w^{(1)})v^{(2)} \Big)
$$
by (\ref{f2(H3,H2.H1) generalized}). Adding them together we get (\ref{equation for p=3}). \\[0.1in]

\noindent\textbf{4.2)} For $(r,s,t)=(3,2,1)$, the definition of $f_2\big([z^{(3)}],[w^{(2)}]\big)$ is depending on whether $[z^{(3)}]$ and $[w^{(2)}]$ are reducible. We will talk about these cases separately. In each case the term $f_2\big([z^{(3)} w^{(2)}],[v^{(1)}]\big)=0$ as $[z^{(3)} w^{(2)}]=0$.

\noindent\textbf{4.2.1)} If $[z^{(3)}]=[x_i^{(3)}]$, we have
\begin{align*}
m_2(f_1\otimes f_2)\big([z^{(3)}],[w^{(2)}],[v^{(1)}]\big) &= -\theta z^{(3)} L^{-1}(w^{(2)} v^{(1)}) \\
&= -\theta a_i^{[3]}b_i^{[3]}L^{-1}(w^{(2)} v^{(1)})+\theta\omega L^{-1}(a_i^{[3]}b_i^{[3]})L^{-1}(w^{(2)} v^{(1)}),
\end{align*}
$$
-f_2\big([z^{(3)}],[w^{(2)}]\big)f_1\big([v^{(1)}]\big) = -\theta\Big( L^{-1}(b_i^{[3]}a_i^{[3]})w^{(2)}v^{(1)}-b_i^{[3]}L^{-1}(a_i^{[3]}w^{(2)})v^{(1)} \Big)
$$
by (\ref{f2(H2.H1,H2)}), and
\begin{align*}
&\quad f_2\big([z^{(3)}],[w^{(2)} v^{(1)}]\big) \\
&= \theta\Big( z^{(3)} L^{-1}(v^{(1)} w^{(2)})-L^{-1}(z^{(3)} v^{(1)})w^{(2)} \Big) \\
&= \theta\Big( (1-\omega L^{-1})(a_i^{[3]}b_i^{[3]})L^{-1}(v^{(1)} w^{(2)}) -L^{-1}\big( (1-\omega L^{-1})(a_i^{[3]}b_i^{[3]}) v^{(1)} \big)w^{(2)} \Big) \\
&= \theta\Big( a_i^{[3]}b_i^{[3]}L^{-1}(v^{(1)} w^{(2)})-\omega L^{-1}(a_i^{[3]}b_i^{[3]})L^{-1}(v^{(1)} w^{(2)}) \\
&\quad\quad -L^{-1}(a_i^{[3]}b_i^{[3]}v^{(1)})w^{(2)}+L^{-1}(a_i^{[3]}b_i^{[3]})v^{(1)} w^{(2)} \Big).
\end{align*}
by (\ref{f2(H3,H2.H1) generalized}). Adding all these terms together we get
$$
\theta b_i^{[3]}L^{-1}(a_i^{[3]}w^{(2)})v^{(1)}-\theta L^{-1}(a_i^{[3]}b_i^{[3]}v^{(1)})w^{(2)}.
$$
According to (\ref{necessary equation for formal}) we have
\begin{align*}
(1-\omega L^{-1})(b_i^{[3]}v^{(1)})L^{-1}(a_i^{[3]}w^{(2)}) &= L^{-1}\Big( (1-\omega L^{-1})(b_i^{[3]}v^{(1)})a_i^{[3]} \Big)w^{(2)} \\
b_i^{[3]}v^{(1)} L^{-1}(a_i^{[3]}w^{(2)})-L^{-1}(b_i^{[3]}v^{(1)})\omega L^{-1}(a_i^{[3]}w^{(2)}) &= L^{-1}(b_i^{[3]}v^{(1)} a_i^{[3]})w^{(2)} -L^{-1}(b_i^{[3]}v^{(1)})a_i^{[3]}w^{(2)} \\
b_i^{[3]}v^{(1)} L^{-1}(a_i^{[3]}w^{(2)}) &= L^{-1}(b_i^{[3]}v^{(1)} a_i^{[3]})w^{(2)} \\
b_i^{[3]}L^{-1}(a_i^{[3]}w^{(2)})v^{(1)} &= L^{-1}(a_i^{[3]}b_i^{[3]}v^{(1)})w^{(2)}.
\end{align*}
So (\ref{equation for p=3}) holds. \\[0.1in]

\noindent\textbf{4.2.2)}  If $[w^{(2)}]=[x_i^{(2)}]$, we have
\begin{align*}
m_2(f_1\otimes f_2)\big([z^{(3)}],[w^{(2)}],[v^{(1)}]\big) &= -\theta z^{(3)} L^{-1}(w^{(2)} v^{(1)}) \\
&= -\theta z^{(3)} L^{-1}\Big( (1-\omega L^{-1})(a_i^{[2]}b_i^{[2]}) v^{(1)} \Big) \\
&=  -\theta z^{(3)} L^{-1}(a_i^{[2]}b_i^{[2]} v^{(1)})+\theta z^{(3)} L^{-1}(a_i^{[2]}b_i^{[2]}) v^{(1)},
\end{align*}

$$
-f_2\big([z^{(3)}],[w^{(2)}]\big)f_1\big([v^{(1)}]\big) = -\theta\Big( z^{(3)}L^{-1}(a_i^{[2]}b_i^{[2]})v^{(1)}-L^{-1}(z^{(3)} a_i^{[2]})b_i^{[2]}v^{(1)} \Big)
$$
by (\ref{f2(H3,H1.H1)}), and
\begin{align*}
&\quad f_2\big([z^{(3)}],[w^{(2)} v^{(1)}]\big) \\
&= \theta\Big( z^{(3)} L^{-1}(v^{(1)} w^{(2)})-L^{-1}(z^{(3)} v^{(1)})w^{(2)} \Big) \\
&= \theta\Big( z^{(3)} L^{-1} \big(v^{(1)} (1-\omega L^{-1})(a_i^{[2]}b_i^{[2]}) \big)-L^{-1}(z^{(3)} v^{(1)})(1-\omega L^{-1})(a_i^{[2]}b_i^{[2]}) \Big) \\
&= \theta\Big( z^{(3)} L^{-1}(v^{(1)} a_i^{[2]}b_i^{[2]})-z^{(3)} v^{(1)} L^{-1}(a_i^{[2]}b_i^{[2]}) -L^{-1}(z^{(3)} v^{(1)})a_i^{[2]}b_i^{[2]}+z^{(3)} v^{(1)} L^{-1}(a_i^{[2]}b_i^{[2]}) \Big) \\
&= \theta\Big( z^{(3)} L^{-1}(v^{(1)} a_i^{[2]}b_i^{[2]}) -L^{-1}(z^{(3)} v^{(1)})a_i^{[2]}b_i^{[2]} \Big)
\end{align*}
by (\ref{f2(H3,H2.H1) generalized}). Adding all these terms together we get
$$
\theta L^{-1}(z^{(3)} a_i^{[2]})b_i^{[2]}v^{(1)}-\theta L^{-1}(z^{(3)} v^{(1)})a_i^{[2]}b_i^{[2]}.
$$
According to (\ref{necessary equation for formal}) we have
$$
L^{-1}(z^{(3)} a_i^{[2]})b_i^{[2]}v^{(1)} = -v^{(1)} L^{-1}(z^{(3)} a_i^{[2]})b_i^{[2]} = -L^{-1}(v^{(1)} z^{(3)})a_i^{[2]}b_i^{[2]} = L^{-1}(z^{(3)} v^{(1)})a_i^{[2]}b_i^{[2]}.
$$
So (\ref{equation for p=3}) holds. \\[0.1in]

\noindent\textbf{4.2.3)} If $[z^{(3)}]=[y_j^{(3)}]$ and $[w^{(2)}]=[y_k^{(2)}]$, the terms become
$$
m_2(f_1\otimes f_2)\big([z^{(3)}],[w^{(2)}],[v^{(1)}]\big) = -\theta y_j^{(3)} L^{-1}(y_k^{(2)} v^{(1)}),
$$
\begin{align*}
-f_2\big([z^{(3)}],[w^{(2)}]\big)f_1\big([v^{(1)}]\big) &= \theta L^{-1}(y_j^{(3)} y_k^{(2)})v^{(1)}+\sum_i L^{-3}\Big( -y_j^{(3)} f_2\big([y_k^{(2)}],[x_i^{(3)}]\big) \Big)(x_i^{(3)})^*v^{(1)},
\end{align*}
and
\begin{align*}
f_2\big([z^{(3)}],[w^{(2)} v^{(1)}]\big) &= \theta\Big( z^{(3)} L^{-1}(v^{(1)} w^{(2)})-L^{-1}(z^{(3)} v^{(1)})w^{(2)} \Big) \\
&= \theta\Big( y_j^{(3)} L^{-1}(v^{(1)} y_k^{(2)})-L^{-1}(y_j^{(3)} v^{(1)})y_k^{(2)} \Big)
\end{align*}
by (\ref{f2(H3,H2.H1) generalized}).  Adding all these terms together we get \begin{align}\label{m3(3,2,1)}
\theta L^{-1}(y_j^{(3)} y_k^{(2)})v^{(1)} -\theta L^{-1}(y_j^{(3)} v^{(1)})y_k^{(2)} +\sum_i L^{-3}\Big( -y_j^{(3)} f_2\big([y_k^{(2)}],[x_i^{(3)}]\big) \Big)(x_i^{(3)})^*v^{(1)}.
\end{align}
Observe that (\ref{m3(3,2,1)}) is an element in $\theta A^4$. So according to Poincar\'e duality, we can show that (\ref{m3(3,2,1)}) vanishes by verifying that it multiplies any element in $A^2$ is 0. When it multiplies $u^{(2)}\in PA^2$, we assume $[v^{(1)}u^{(2)}]=\sum_i \lambda_i [x_i^{(3)}]$ in $H^3\big(A[\theta]\big)$, i.e. $v^{(1)} u^{(2)}=\sum_i \lambda_i x_i^{(3)}+\omega L^{-1}(v^{(1)} u^{(2)})$. Then the last term of (\ref{m3(3,2,1)}) multiplying $u^{(2)}$ is
\begin{align*}
&\quad \sum_{i,l} \lambda_l L^{-3}\Big( -y_j^{(3)} f_2\big([y_k^{(2)}],[x_i^{(3)}]\big) \Big)(x_i^{(3)})^*x_l^{(3)} \\
&\quad +\sum_i L^{-3}\Big( -y_j^{(3)} f_2\big([y_k^{(2)}],[x_i^{(3)}]\big) \Big)(x_i^{(3)})^*\omega L^{-1}(v^{(1)}u^{(2)}) \\
&= \sum_i \lambda_i L^{-3}\Big( -y_j^{(3)} f_2\big([y_k^{(2)}],[x_i^{(3)}]\big) \Big) (-\omega^3) \\
&= y_j^{(3)} f_2\big([y_k^{(2)}],[v^{(1)} u^{(2)}]\big)
\end{align*}
as $(x_i^{(3)})^*\in PA^3$ and $\omega PA^3=0$. Thus, we get
\begin{align}\label{x3*.x3}
\sum_i L^{-3}\Big( -y_j^{(3)} f_2\big([y_k^{(2)}],[x_i^{(3)}]\big) \Big)(x_i^{(3)})^*v^{(1)} u^{(2)} = y_j^{(3)} f_2\big([y_k^{(2)}],[v^{(1)} u^{(2)}]\big)
\end{align}

Then by (\ref{f2(H2,H2.H1) generalized})
$$
y_j^{(3)} f_2\big([y_k^{(2)}],[v^{(1)}u^{(2)}]\big) = y_j^{(3)} \theta\Big( y_k^{(2)}L^{-1}(u^{(2)} v^{(1)})-L^{-1}(y_k^{(2)} u^{(2)})v^{(1)} \Big).
$$
Observe that
$$
y_j^{(3)} \theta L^{-1}(y_k^{(2)} u^{(2)})v^{(1)} = -\theta\omega L^{-1}(y_j^{(3)} v^{(1)}) L^{-1}(y_k^{(2)} u^{(2)}) = -\theta L^{-1}(y_j^{(3)} v^{(1)})y_k^{(2)} u^{(2)}.
$$
which is the second term of (\ref{m3(3,2,1)}) multiplying $u^{(2)}$. On the other hand, the first term of (\ref{m3(3,2,1)}) multiplying $u^{(2)}$ is
$$
\theta L^{-1}(y_j^{(3)} y_k^{(2)})v^{(1)}u^{(2)} = \theta L^{-1}(y_j^{(3)} y_k^{(2)})(1-\omega L^{-1})(v^{(1)}u^{(2)})+\theta L^{-1}(y_j^{(3)} y_k^{(2)})\omega L^{-1}(v^{(1)}u^{(2)}).
$$
Since $y_j^{(3)} y_k^{(2)}\in A^5=\omega^2 PA^1$, $L^{-1}(y_j^{(3)} y_k^{(2)})\in\omega PA^1$. Also $(1-\omega L^{-1})(v^{(1)}u^{(2)})\in PA^3$. So their product vanishes. Thus,
\begin{align*}
\theta L^{-1}(y_j^{(3)} y_k^{(2)})v^{(1)}u^{(2)} &= \theta L^{-1}(y_j^{(3)} y_k^{(2)})\omega L^{-1}(v^{(1)}u^{(2)}) \\
&= \theta y_j^{(3)} y_k^{(2)} L^{-1}(v^{(1)}u^{(2)}) \\
&= -y_j^{(3)} \theta y_k^{(2)}L^{-1}(u^{(2)} v^{(1)}).
\end{align*}
This implies that (\ref{m3(3,2,1)}) multiplying $u^{(2)}$ is 0. When it multiplies $\omega$, the last term vanishes as $(x_i^{(3)})^* \omega=0$. So the remaining terms are
$$
\theta L^{-1}(y_j^{(3)} y_k^{(2)})v^{(1)}\omega -\theta L^{-1}(y_j^{(3)} v^{(1)})y_k^{(2)}\omega = \theta y_j^{(3)} y_k^{(2)}v^{(1)}-\theta y_j^{(3)} v^{(1)}y_k^{(2)} = 0.
$$
Therefore, (\ref{m3(3,2,1)}) vanishes and (\ref{equation for p=3}) holds. \\[0.1in]

\noindent\textbf{4.3)} For $(r,s,t)=(2,3,1)$, we also have to go through the cases that whether $[z^{(2)}],[w^{(3)}]$ are reducible. This time both $[z^{(2)} w^{(3)}]$ and $[w^{(3)} v^{(1)}]$ are 0.

\noindent\textbf{4.3.1)} If $[z^{(2)}]=[x_i^{(2)}]$, we have
\begin{align*}
m_2(f_1\otimes f_2)\big([z^{(2)}],[w^{(3)}],[v^{(1)}]\big) &= -\theta z^{(2)} L^{-1}(w^{(3)} v^{(1)}) \\
&= -\theta a_i^{[2]}b_i^{[2]}L^{-1}(w^{(3)} v^{(1)})+\theta\omega L^{-1}(a_i^{[2]}b_i^{[2]})L^{-1}(w^{(3)} v^{(1)}),
\end{align*}
and
$$
-f_2\big([z^{(2)}],[w^{(3)}]\big)f_1\big([v^{(1)}]\big) = -\theta\Big( L^{-1}(a_i^{[2]}b_i^{[2]})w^{(3)}v^{(1)}-a_i^{[2]}L^{-1}(b_i^{[2]}w^{(3)})v^{(1)} \Big).
$$
By (\ref{necessary equation for formal}) we have 
$$
a_i^{[2]}L^{-1}(b_i^{[2]}w^{(3)})v^{(1)} = a_i^{[2]}b_i^{[2]} L^{-1}(w^{(3)} v^{(1)}).
$$
On the other hand,
$$
\omega L^{-1}(a_i^{[2]}b_i^{[2]})L^{-1}(w^{(3)} v^{(1)}) = L^{-1}(a_i^{[2]}b_i^{[2]})w^{(3)}v^{(1)}.
$$
So we get (\ref{equation for p=3}). \\[0.1in]

\noindent\textbf{4.3.2)}  If $[w^{(3)}]=[x_i^{(3)}]$, we have
\begin{align*}
m_2(f_1\otimes f_2)\big([z^{(2)}],[w^{(3)}],[v^{(1)}]\big) &= -\theta z^{(2)} L^{-1}(w^{(3)} v^{(1)}) \\
&= -\theta z^{(2)} L^{-1}\Big( (1-\omega L^{-1})(a_i^{[3]}b_i^{[3]}) v^{(1)} \Big) \\
&=  -\theta z^{(2)} L^{-1}(a_i^{[3]}b_i^{[3]} v^{(1)})+\theta z^{(2)} L^{-1}(a_i^{[3]}b_i^{[3]}) v^{(1)},
\end{align*}
and
$$
-f_2\big([z^{(2)}],[w^{(3)}]\big)f_1\big([v^{(1)}]\big) = -\theta\Big( z^{(2)}L^{-1}(a_i^{[3]}b_i^{[3]})v^{(1)}-L^{-1}(z^{(2)} a_i^{[3]})b_i^{[3]}v^{(1)} \Big).
$$
According to (\ref{necessary equation for formal}),
\begin{align*}
z^{(2)} L^{-1}(a_i^{[3]}b_i^{[3]} v^{(1)}) &= z^{(2)} L^{-1}\Big( a_i^{[3]}(1-\omega L^{-1})(b_i^{[3]} v^{(1)}) \Big) +z^{(2)} L^{-1}\Big( a_i^{[3]}\omega L^{-1}(b_i^{[3]} v^{(1)}) \Big) \\
&= L^{-1}(z^{(2)} a_i^{[3]})(1-\omega L^{-1})(b_i^{[3]} v^{(1)}) +z^{(2)} a_i^{[3]} L^{-1}(b_i^{[3]} v^{(1)}) \\
&= L^{-1}(z^{(2)} a_i^{[3]})b_i^{[3]} v^{(1)} -L^{-1}(z^{(2)} a_i^{[3]})\omega L^{-1}(b_i^{[3]} v^{(1)}) +z^{(2)} a_i^{[3]} L^{-1}(b_i^{[3]} v^{(1)}) \\
&= L^{-1}(z^{(2)} a_i^{[3]})b_i^{[3]} v^{(1)}.
\end{align*}
So we get (\ref{equation for p=3}). \\[0.1in]

\noindent\textbf{4.3.3)} If $[z^{(2)}]=[y_j^{(2)}]$ and $[w^{(3)}]=[y_k^{(3)}]$, the terms become
$$
m_2(f_1\otimes f_2)\big([z^{(2)}],[w^{(3)}],[v^{(1)}]\big) = -\theta y_j^{(2)} L^{-1}(y_k^{(3)} v^{(1)}),
$$
and
\begin{align*}
-f_2\big([z^{(2)}],[w^{(3)}]\big)f_1\big([v^{(1)}]\big) = \theta L^{-1}(y_j^{(2)} y_k^{(3)})v^{(1)}+\sum_i L^{-3}\Big( -f_2\big([x_i^{(3)}],[y_j^{(2)}]\big)y_k^{(3)} \Big)(x_i^{(3)})^*v^{(1)}.
\end{align*}
We will also show that their sum $m_2(f_1\otimes f_2-f_2\otimes f_1)\big([z^{(2)}],[w^{(3)}],[v^{(1)}]\big)$ is zero by multiplying it to any elements in $A^2$. When multiplying $u^{(2)}\in PA^2$, same as (\ref{x3*.x3}), the last term becomes
\begin{align*}
\sum_i L^{-3}\Big( -f_2\big([x_i^{(3)}],[y_j^{(2)}]\big)y_k^{(3)} \Big)(x_i^{(3)})^*v^{(1)}u^{(2)} &= f_2\big([v^{(1)}u^{(2)}],[y_j^{(2)}]\big)y_k^{(3)} \\
&= \theta\Big( L^{-1}(v^{(1)} u^{(2)})y_j^{(2)}-v^{(1)} L^{-1}(u^{(2)} y_j^{(2)}) \Big)y_k^{(3)}.
\end{align*}
For the other two terms, we have
\begin{align*}
y_j^{(2)} L^{-1}(y_k^{(3)} v^{(1)})u^{(2)} &= \omega L^{-1}(u^{(2)} y_j^{(2)}) L^{-1}(y_k^{(3)} v^{(1)}) \\
&= L^{-1}(u^{(2)} y_j^{(2)}) y_k^{(3)} v^{(1)} \\
&= -v^{(1)} L^{-1}(u^{(2)} y_j^{(2)}) y_k^{(3)},
\end{align*}
and
\begin{align*}
L^{-1}(y_j^{(2)} y_k^{(3)})v^{(1)}u^{(2)} &= L^{-1}(y_j^{(2)} y_k^{(3)})(1-\omega L^{-1})(v^{(1)}u^{(2)}) +L^{-1}(y_j^{(2)} y_k^{(3)})\omega L^{-1}(v^{(1)}u^{(2)}) \\
&= y_j^{(2)} y_k^{(3)} L^{-1}(v^{(1)}u^{(2)}) \\
&= -L^{-1}(v^{(1)}u^{(2)})y_j^{(2)} y_k^{(3)}
\end{align*}
as $ L^{-1}(y_j^{(2)} y_k^{(3)})\in\omega PA^1$ and $(1-\omega L^{-1})(v^{(1)}u^{(2)})\in PA^3$. This shows the sum $m_2(f_1\otimes f_2-f_2\otimes f_1)\big([z^{(2)}],[w^{(3)}],[v^{(1)}]\big)$ multiplying $u^{(2)}$ is 0. It remains to check this sum multiplying $\omega$ is 0.

When multiplying $\omega$, we have $(x_i^{(3)})^*\omega=0$ as $(x_i^{(3)})^*\in PA^3$ and $\omega PA^3=0$. The remaining terms are
\begin{align*}
-\theta y_j^{(2)} L^{-1}(y_k^{(3)} v^{(1)})\omega +\theta L^{-1}(y_j^{(2)} y_k^{(3)})v^{(1)}\omega = -\theta y_j^{(2)}y_k^{(3)} v^{(1)} +\theta y_j^{(2)} y_k^{(3)} v^{(1)} = 0.
\end{align*}
Thus, we get (\ref{equation for p=3}).\\[0.1in]

\noindent\textbf{5)} When $r+s+t=7$, we have went through the case $(r,s,t)=(3,1,3)$. It remains to check $(3,2,2),(3,3,1)$, and $(2,3,2)$.

\noindent\textbf{5.1)} For $(r,s,t)=(3,2,2)$, both $[z^{(3)} w^{(2)}]$ and $[w^{(2)} v^{(2)}]$ are 0. Similar as before we have to go through the cases that whether $[z^{(3)}],[w^{(2)}]$ are reducible.

\noindent\textbf{5.1.1)} If $[z^{(3)}]=[x_i^{(3)}]$, we have
\begin{align*}
m_2(f_1\otimes f_2)\big([z^{(3)}],[w^{(2)}],[v^{(2)}]\big) &= -\theta z^{(3)} L^{-1}(w^{(2)} v^{(2)}) \\
&= -\theta a_i^{[3]}b_i^{[3]}L^{-1}(w^{(2)} v^{(2)})+\theta\omega L^{-1}(a_i^{[3]}b_i^{[3]})L^{-1}(w^{(2)} v^{(2)}),
\end{align*}
and
$$
-f_2\big([z^{(3)}],[w^{(2)}]\big)f_1\big([v^{(2)}]\big) = -\theta\Big( L^{-1}(b_i^{[3]}a_i^{[3]})w^{(2)}v^{(2)}-b_i^{[3]}L^{-1}(a_i^{[3]}w^{(2)})v^{(2)} \Big).
$$
Since
$$
a_i^{[3]}b_i^{[3]}L^{-1}(w^{(2)} v^{(2)}) = b_i^{[3]}a_i^{[3]}L^{-1}(w^{(2)} v^{(2)}) = b_i^{[3]}L^{-1}(a_i^{[3]} w^{(2)}) v^{(2)}
$$
by (\ref{necessary equation for formal}), and
$$
\omega L^{-1}(a_i^{[3]}b_i^{[3]})L^{-1}(w^{(2)} v^{(2)}) =  L^{-1}(b_i^{[3]}a_i^{[3]})w^{(2)}v^{(2)},
$$
we get (\ref{equation for p=3}). \\[0.1in]

\noindent\textbf{5.1.2)}  If $[w^{(2)}]=[x_i^{(2)}]$, we have
\begin{align*}
m_2(f_1\otimes f_2)\big([z^{(3)}],[w^{(2)}],[v^{(2)}]\big) &= -\theta z^{(3)} L^{-1}(w^{(2)} v^{(2)}) \\
&= -\theta z^{(3)} L^{-1}\Big( (1-\omega L^{-1})(a_i^{[2]}b_i^{[2]}) v^{(2)} \Big) \\
&= -\theta z^{(3)} L^{-1}(a_i^{[2]}b_i^{[2]} v^{(2)})+\theta z^{(3)} L^{-1}(a_i^{[2]}b_i^{[2]}) v^{(2)},
\end{align*}
and
$$
-f_2\big([z^{(3)}],[w^{(2)}]\big)f_1\big([v^{(2)}]\big) = -\theta\Big( z^{(3)}L^{-1}(a_i^{[2]}b_i^{[2]})v^{(2)}-L^{-1}(z^{(3)} a_i^{[2]})b_i^{[2]}v^{(2)} \Big).
$$
According to (\ref{necessary equation for formal}) we have
\begin{align*}
z^{(3)} L^{-1}\Big( a_i^{[2]}(1-\omega L^{-1})(b_i^{[2]} v^{(2)}) \Big) &= L^{-1}(z^{(3)} a_i^{[2]})(1-\omega L^{-1})(b_i^{[2]} v^{(2)}) \\
z^{(3)} L^{-1}(a_i^{[2]}b_i^{[2]} v^{(2)})-z^{(3)} a_i^{[2]} L^{-1}(b_i^{[2]} v^{(2)}) &= L^{-1}(z^{(3)} a_i^{[2]}) b_i^{[2]} v^{(2)}-L^{-1}(z^{(3)} a_i^{[2]})\omega L^{-1}(b_i^{[2]} v^{(2)}) \\
z^{(3)} L^{-1}(a_i^{[2]}b_i^{[2]} v^{(2)}) &= L^{-1}(z^{(3)} a_i^{[2]}) b_i^{[2]} v^{(2)}.
\end{align*}
So (\ref{equation for p=3}) holds.

\noindent\textbf{5.1.3)} If $[z^{(3)}]=[y_j^{(3)}]$ and $[w^{(2)}]=[y_k^{(2)}]$, the terms become
$$
m_2(f_1\otimes f_2)\big([z^{(3)}],[w^{(2)}],[v^{(2)}]\big) = -\theta y_j^{(3)} L^{-1}(y_k^{(2)} v^{(2)}),
$$
and
\begin{align*}
-f_2\big([z^{(3)}],[w^{(2)}]\big)f_1\big([v^{(2)}]\big) &= \theta L^{-1}(y_j^{(3)} y_k^{(2)})v^{(2)}+\sum_i L^{-3}\Big( -y_j^{(3)} f_2\big([y_k^{(2)}],[x_i^{(3)}]\big) \Big)(x_i^{(3)})^*v^{(2)}.
\end{align*}
As the sum of these terms is in $\theta A^5$, we will show it is 0 by verifying it vanishes when multiplying any $u^{(1)}\in A^1$. Same as (\ref{x3*.x3}), we have
\begin{align*}
\sum_i L^{-3}\Big( -y_j^{(3)} f_2\big([y_k^{(2)}],[x_i^{(3)}]\big) \Big)(x_i^{(3)})^*v^{(2)}u^{(1)} &= y_j^{(3)} f_2\big([y_k^{(2)}],[v^{(2)} u^{(1)}]\big) \\
&= y_j^{(3)} \theta\Big( y_k^{(2)}L^{-1}(v^{(2)} u^{(1)})-L^{-1}(y_k^{(2)} v^{(2)})u^{(1)} \Big).
\end{align*}
Then
$$
-\theta y_j^{(3)} L^{-1}(y_k^{(2)} v^{(2)})u^{(1)} = y_j^{(3)}\theta L^{-1}(y_k^{(2)} v^{(2)})u^{(1)},
$$
and
\begin{align*}
\theta L^{-1}(y_j^{(3)} y_k^{(2)})v^{(2)}u^{(1)} &= \theta L^{-1}(y_j^{(3)} y_k^{(2)})(1-\omega L^{-1})(v^{(2)}u^{(1)})+\theta L^{-1}(y_j^{(3)} y_k^{(2)})\omega L^{-1}(v^{(2)}u^{(1)}) \\
&= \theta y_j^{(3)} y_k^{(2)} L^{-1}(v^{(2)}u^{(1)}) \\
&= -y_j^{(3)}\theta y_k^{(2)} L^{-1}(v^{(2)}u^{(1)})
\end{align*}
as $L^{-1}(y_j^{(3)} y_k^{(2)})\in \omega PA^1$ and $(1-\omega L^{-1})(v^{(2)}u^{(1)})\in PA^3$. So we get (\ref{equation for p=3}). \\[0.1in]

\noindent\textbf{5.2)} For $(r,s,t)=(3,3,1)$, both $[z^{(3)} w^{(3)}]$ and $[w^{(3)} v^{(1)}]$ are 0. We will also go through the cases that whether $[z^{(3)}],[w^{(2)}]$ are reducible.

\noindent\textbf{5.2.1)} If $[z^{(3)}]=[x_i^{(3)}]$, we have
\begin{align*}
m_2(f_1\otimes f_2)\big([z^{(3)}],[w^{(3)}],[v^{(1)}]\big) &= -\theta z^{(3)} L^{-1}(w^{(3)} v^{(1)}) \\
&= -\theta a_i^{[3]}b_i^{[3]}L^{-1}(w^{(3)} v^{(1)})+\theta\omega L^{-1}(a_i^{[3]}b_i^{[3]})L^{-1}(w^{(3)} v^{(1)}),
\end{align*}
and
$$
-f_2\big([z^{(3)}],[w^{(3)}]\big)f_1\big([v^{(1)}]\big) = -\theta\Big( L^{-1}(a_i^{[3]}b_i^{[3]})w^{(3)}v^{(1)}-a_i^{[3]}L^{-1}(b_i^{[3]}w^{(3)})v^{(1)} \Big).
$$
By (\ref{necessary equation for formal}) we have
$$
a_i^{[3]}b_i^{[3]}L^{-1}(w^{(3)} v^{(1)}) = a_i^{[3]}L^{-1}(b_i^{[3]}w^{(3)}) v^{(1)}.
$$
Also
$$
\omega L^{-1}(a_i^{[3]}b_i^{[3]})L^{-1}(w^{(3)} v^{(1)}) = L^{-1}(a_i^{[3]}b_i^{[3]})w^{(3)}v^{(1)}.
$$
So (\ref{equation for p=3}) holds. \\[0.1in]

\noindent\textbf{5.2.2)}  If $[w^{(3)}]=[x_i^{(3)}]$, we have
\begin{align*}
m_2(f_1\otimes f_2)\big([z^{(3)}],[w^{(3)}],[v^{(1)}]\big) &= -\theta z^{(3)} L^{-1}(w^{(3)} v^{(1)}) \\
&= -\theta z^{(3)} L^{-1}\Big( (1-\omega L^{-1})(a_i^{[3]}b_i^{[3]}) v^{(1)} \Big) \\
&=  -\theta z^{(3)} L^{-1}(a_i^{[3]}b_i^{[3]} v^{(1)})+\theta z^{(3)} L^{-1}(a_i^{[3]}b_i^{[3]}) v^{(1)},
\end{align*}
and
$$
-f_2\big([z^{(3)}],[w^{(3)}]\big)f_1\big([v^{(1)}]\big) = -\theta\Big( z^{(3)}L^{-1}(b_i^{[3]}a_i^{[3]})v^{(1)}-L^{-1}(z^{(3)} b_i^{[3]})a_i^{[3]}v^{(1)} \Big).
$$
By (\ref{necessary equation for formal}) we have
\begin{align*}
z^{(3)} L^{-1} \big( b_i^{[3]} (1-\omega L^{-1}) (a_i^{[3]} v^{(1)}) \big) &= L^{-1}(z^{(3)} b_i^{[3]}) (1-\omega L^{-1}) (a_i^{[3]} v^{(1)}) \\
z^{(3)} L^{-1}(b_i^{[3]} a_i^{[3]} v^{(1)}) -z^{(3)} b_i^{[3]} L^{-1} (a_i^{[3]} v^{(1)}) &= L^{-1}(z^{(3)} b_i^{[3]}) a_i^{[3]} v^{(1)} -L^{-1}(z^{(3)} b_i^{[3]})\omega L^{-1} (a_i^{[3]} v^{(1)}) \\
z^{(3)} L^{-1}(a_i^{[3]} b_i^{[3]} v^{(1)}) &= L^{-1}(z^{(3)} b_i^{[3]}) a_i^{[3]} v^{(1)}.
\end{align*}
Thus, (\ref{equation for p=3}) holds. \\[0.1in]

\noindent\textbf{5.2.3)} If $[z^{(3)}]=[y_j^{(3)}]$ and $[w^{(3)}]=[y_k^{(3)}]$, the terms become
$$
m_2(f_1\otimes f_2)\big([z^{(3)}],[w^{(3)}],[v^{(1)}]\big) = -\theta y_j^{(3)} L^{-1}(y_k^{(3)} v^{(1)}),
$$
and
\begin{align*}
-f_2\big([z^{(3)}],[w^{(3)}]\big)f_1\big([v^{(1)}]\big) &= \theta L^{-1}(y_j^{(3)} y_k^{(3)})v^{(1)}+\sum_i L^{-3}\Big( y_j^{(3)} f_2\big([y_k^{(3)}],[x_i^{(2)}]\big) \Big)(x_i^{(2)})^* v^{(1)}.
\end{align*}
We will also prove their sum is 0 by multiplying $u^{(1)}\in A^1$. Similar as (\ref{x3*.x3}), we assume $[v^{(1)} u^{(1)}] = \sum_i \lambda_i[x_i^{(2)}]$, i.e. $v^{(1)}u^{(1)}=\sum_i \lambda_i x_i^{(2)}+\omega L^{-1}(v^{(1)}u^{(1)})$. Then the last term multiplying $u^{(1)}$ is
$$
\sum_{i,l} \lambda_l L^{-3}\Big( y_j^{(3)} f_2\big([y_k^{(3)}],[x_i^{(2)}]\big) \Big)(x_i^{(2)})^* x_l^{(2)} +\sum_i L^{-3}\Big( y_j^{(3)} f_2\big([y_k^{(3)}],[x_i^{(2)}]\big) \Big)(x_i^{(2)})^* \omega L^{-1}(v^{(1)}u^{(1)}).
$$
Its second term vanishes because $(x_i^{(2)})^*$ is in $\omega PA^2$ and becomes 0 when multiplies $\omega$. The remaining first term is
\begin{align*}
\sum_i \lambda_i L^{-3}\Big( y_j^{(3)} f_2\big([y_k^{(3)}],[x_i^{(2)}]\big) \Big)\omega^3 = \sum_i \lambda_i y_j^{(3)} f_2\big([y_k^{(3)}],[x_i^{(2)}]\big) = y_j^{(3)} f_2\big([y_k^{(3)}],[v^{(1)}u^{(1)}]\big).
\end{align*}
Thus, we get
\begin{align}\label{x2*.x2}
\sum_i L^{-3}\Big( y_j^{(3)} f_2\big([y_k^{(3)}],[x_i^{(2)}]\big) \Big)(x_i^{(2)})^* v^{(1)} u^{(1)} = y_j^{(3)} f_2\big([y_k^{(3)}],[v^{(1)}u^{(1)}]\big).
\end{align}
By (\ref{f2(H3,H1.H1) generalized}),
$$
y_j^{(3)} f_2\big([y_k^{(3)}],[v^{(1)}u^{(1)}]\big) = y_j^{(3)}\theta \Big( y_k^{(3)} L^{-1}(v^{(1)} u^{(1)})-L^{-1}(y_k^{(3)} v^{(1)})u^{(1)} \Big).
$$
Observes that
$$
\theta y_j^{(3)} L^{-1}(y_k^{(3)} v^{(1)})u^{(1)} = -y_j^{(3)}\theta L^{-1}(y_k^{(3)} v^{(1)})u^{(1)},
$$
and
\begin{align*}
\theta L^{-1}(y_j^{(3)} y_k^{(3)})v^{(1)}u^{(1)} &= \theta L^{-1}(y_j^{(3)} y_k^{(3)})(1-\omega L^{-1})(v^{(1)}u^{(1)})+\theta L^{-1}(y_j^{(3)} y_k^{(3)})\omega L^{-1}(v^{(1)}u^{(1)}) \\
&= \theta L^{-1}(y_j^{(3)} y_k^{(3)})\omega L^{-1}(v^{(1)}u^{(1)}) \\
&= -y_j^{(3)} \theta y_k^{(3)} L^{-1}(v^{(1)}u^{(1)})
\end{align*}
as $L^{-1}(y_j^{(3)} y_k^{(3)})\in \omega^2 A^0$ and $(1-\omega L^{-1})(v^{(1)}u^{(1)})\in PA^2$. So we get (\ref{equation for p=3}). \\[0.1in]

\noindent\textbf{5.3)} For $(r,s,t)=(2,3,2)$, both $[z^{(2)} w^{(3)}]$ and $[w^{(3)} v^{(2)}]$ are 0. This time we need to go through all the cases that whether $[z^{(2)}],[w^{(3)}],[v^{(2)}]$ are reducible.

\noindent\textbf{5.3.1)} If $[w^{(3)}]=[x_i^{(3)}]$, we have
$$
m_2(f_1\otimes f_2)\big([z^{(2)}],[w^{(3)}],[v^{(2)}]\big) = z^{(2)} \theta\Big( L^{-1}(b_i^{[3]}a_i^{[3]})v^{(2)}-b_i^{[3]}L^{-1}(a_i^{[3]}v^{(2)}) \Big),
$$
and
$$
-f_2\big([z^{(2)}],[w^{(3)}]\big)f_1\big([v^{(2)}]\big) = -\theta\Big( z^{(2)}L^{-1}(a_i^{[3]}b_i^{[3]})v^{(2)}-L^{-1}(z^{(2)} a_i^{[3]})b_i^{[3]}v^{(2)} \Big).
$$
As
$$
z^{(2)} \theta L^{-1}(b_i^{[3]}a_i^{[3]})v^{(2)} = \theta z^{(2)}L^{-1}(a_i^{[3]}b_i^{[3]})v^{(2)},
$$
and by (\ref{necessary equation for formal})
$$
z^{(2)} \theta b_i^{[3]}L^{-1}(a_i^{[3]}v^{(2)}) = \theta b_i^{[3]} z^{(2)} L^{-1}(a_i^{[3]}v^{(2)}) = \theta b_i^{[3]} L^{-1}(z^{(2)} a_i^{[3]}) v^{(2)} = \theta L^{-1}(z^{(2)} a_i^{[3]}) b_i^{[3]} v^{(2)},
$$
we get (\ref{equation for p=3}). \\[0.1in]

\noindent\textbf{5.3.2)} If $[z^{(2)}]=[x_i^{(2)}],[w^{(3)}]=[y_j^{(3)}],[v^{(2)}]=[x_k^{(2)}]$, we have
\begin{align*}
&\quad m_2(f_1\otimes f_2)\big([z^{(2)}],[w^{(3)}],[v^{(2)}]\big) \\
&= (1-\omega L^{-1})(a_i^{[2]}b_i^{[2]}) \theta\Big( w^{(3)}L^{-1}(a_k^{[2]}b_k^{[2]})-L^{-1}(w^{(3)} a_k^{[2]})b_k^{[2]} \Big) \\
&= \theta\Big( a_i^{[2]}b_i^{[2]} w^{(3)}L^{-1}(a_k^{[2]}b_k^{[2]}) -a_i^{[2]}b_i^{[2]} L^{-1}(w^{(3)} a_k^{[2]})b_k^{[2]} \\
&\quad\quad -\omega L^{-1}(a_i^{[2]}b_i^{[2]}) w^{(3)}L^{-1}(a_k^{[2]}b_k^{[2]}) +\omega L^{-1}(a_i^{[2]}b_i^{[2]}) L^{-1}(w^{(3)} a_k^{[2]})b_k^{[2]} \Big),
\end{align*}
and
\begin{align*}
&\quad -f_2\big([z^{(2)}],[w^{(3)}]\big)f_1\big([v^{(2)}]\big) \\
&= -\theta\Big( L^{-1}(a_i^{[2]}b_i^{[2]})w^{(3)}-a_i^{[2]}L^{-1}(b_i^{[2]}w^{(3)}) \Big) (1-\omega L^{-1})(a_k^{[2]}b_k^{[2]}) \\
&= -\theta\Big( L^{-1}(a_i^{[2]}b_i^{[2]})w^{(3)} a_k^{[2]}b_k^{[2]} -L^{-1}(a_i^{[2]}b_i^{[2]})w^{(3)} \omega L^{-1}(a_k^{[2]}b_k^{[2]}) \\
&\quad\quad\quad -a_i^{[2]}L^{-1}(b_i^{[2]}w^{(3)}) a_k^{[2]}b_k^{[2]} +a_i^{[2]}L^{-1}(b_i^{[2]}w^{(3)}) \omega L^{-1}(a_k^{[2]}b_k^{[2]}) \Big)
\end{align*}
Observe that
$$
a_i^{[2]}b_i^{[2]} w^{(3)}L^{-1}(a_k^{[2]}b_k^{[2]}) = a_i^{[2]}L^{-1}(b_i^{[2]}w^{(3)}) \omega L^{-1}(a_k^{[2]}b_k^{[2]})
$$
as $b_i^{[2]}w^{(3)}\in A^4$ and $\omega\circ L^{-1}:A^4\to A^4$ is an isomorphism. For the same reason
$$
\omega L^{-1}(a_i^{[2]}b_i^{[2]}) L^{-1}(w^{(3)} a_k^{[2]})b_k^{[2]} = L^{-1}(a_i^{[2]}b_i^{[2]})w^{(3)} a_k^{[2]}b_k^{[2]}.
$$
Also by (\ref{necessary equation for formal}) we have
$$
a_i^{[2]}b_i^{[2]} L^{-1}(w^{(3)} a_k^{[2]})b_k^{[2]} = a_i^{[2]}L^{-1}(b_i^{[2]}w^{(3)}) a_k^{[2]}b_k^{[2]}.
$$
So the sum of all these terms are 0. Then (\ref{equation for p=3}) holds. \\[0.1in]

\noindent\textbf{5.3.3)} If $[z^{(2)}]=[x_i^{(2)}],[w^{(3)}]=[y_j^{(3)}],[v^{(2)}]=[y_k^{(2)}]$, we have
\begin{align*}
&\quad m_2(f_1\otimes f_2)\big([z^{(2)}],[w^{(3)}],[v^{(2)}]\big) \\
&= (1-\omega L^{-1})(a_i^{[2]}b_i^{[2]}) \Bigg( -\theta L^{-1}(y_j^{(3)} y_k^{(2)})-\sum_l L^{-3}\Big( -y_j^{(3)} f_2\big([y_k^{(2)}],[x_l^{(3)}]\big) \Big)(x_l^{(3)})^* \Bigg) \\
&= -\theta a_i^{[2]}b_i^{[2]} L^{-1}(y_j^{(3)} y_k^{(2)}) -a_i^{[2]}b_i^{[2]} \sum_l L^{-3}\Big( -y_j^{(3)} f_2\big([y_k^{(2)}],[x_l^{(3)}]\big) \Big)(x_l^{(3)})^* \\
&\quad +\theta\omega L^{-1}(a_i^{[2]}b_i^{[2]}) L^{-1}(y_j^{(3)} y_k^{(2)}) +\omega L^{-1}(a_i^{[2]}b_i^{[2]}) \sum_l L^{-3}\Big( -y_j^{(3)} f_2\big([y_k^{(2)}],[x_l^{(3)}]\big) \Big)(x_l^{(3)})^* \\
&= -\theta a_i^{[2]}b_i^{[2]} L^{-1}(y_j^{(3)} y_k^{(2)}) -a_i^{[2]}b_i^{[2]} \sum_l L^{-3}\Big( -y_j^{(3)} f_2\big([y_k^{(2)}],[x_l^{(3)}]\big) \Big)(x_l^{(3)})^* +\theta L^{-1}(a_i^{[2]}b_i^{[2]}) y_j^{(3)} y_k^{(2)}
\end{align*}
as $\omega\circ L^{-1}:A^5\to A^5$ is an isomorphism and $\omega (x_l^{(3)})^*=0$. On the other hand,
\begin{align*}
-f_2\big([z^{(2)}],[w^{(3)}]\big)f_1\big([v^{(2)}]\big) = -\theta\Big( L^{-1}(a_i^{[2]}b_i^{[2]})y_j^{(3)}-a_i^{[2]}L^{-1}(b_i^{[2]}y_j^{(3)}) \Big) y_k^{(2)}.
\end{align*}
Adding them together we get
$$
-\theta a_i^{[2]}b_i^{[2]} L^{-1}(y_j^{(3)} y_k^{(2)}) +\theta a_i^{[2]}L^{-1}(b_i^{[2]}y_j^{(3)})y_k^{(2)} -a_i^{[2]}b_i^{[2]} \sum_l L^{-3}\Big( -y_j^{(3)} f_2\big([y_k^{(2)}],[x_l^{(3)}]\big) \Big)(x_l^{(3)})^*.
$$
Similar as before, we will prove it is 0 by multiplying $u^{(1)}\in A^1$. Same as (\ref{x3*.x3}), we have
\begin{align*}
&\quad a_i^{[2]}b_i^{[2]} \sum_l L^{-3}\Big( -y_j^{(3)} f_2\big([y_k^{(2)}],[x_l^{(3)}]\big) \Big)(x_l^{(3)})^* u^{(1)} \\
&= y_j^{(3)} f_2\big([y_k^{(2)}],[a_i^{[2]}b_i^{[2]} u^{(1)}]\big) \\
&= y_j^{(3)} \theta\Big( y_k^{(2)} L^{-1} \big( (1-\omega L^{-1})(a_i^{[2]}b_i^{[2]}) u^{(1)} \big) -L^{-1} \big( y_k^{(2)} (1-\omega L^{-1})(a_i^{[2]}b_i^{[2]}) \big) u^{(1)} \Big) \\
&= \theta \Big( -y_j^{(3)} y_k^{(2)} L^{-1} (a_i^{[2]}b_i^{[2]} u^{(1)}) +y_j^{(3)} y_k^{(2)} L^{-1} (a_i^{[2]}b_i^{[2]}) u^{(1)} +y_j^{(3)} L^{-1} (y_k^{(2)} a_i^{[2]}b_i^{[2]})u^{(1)} \\
&\quad\quad -y_j^{(3)} y_k^{(2)} L^{-1}(a_i^{[2]}b_i^{[2]}) u^{(1)} \Big) \\
&= \theta \Big( -y_j^{(3)} y_k^{(2)} L^{-1} (a_i^{[2]}b_i^{[2]} u^{(1)}) +y_j^{(3)} L^{-1} (y_k^{(2)} a_i^{[2]}b_i^{[2]})u^{(1)} \Big)
\end{align*}
Observe that
\begin{align*}
&\quad a_i^{[2]}b_i^{[2]} L^{-1}(y_j^{(3)} y_k^{(2)})u^{(1)} \\
&= L^{-1}(y_j^{(3)} y_k^{(2)})(1-\omega L^{-1})(a_i^{[2]}b_i^{[2]} u^{(1)}) +L^{-1}(y_j^{(3)} y_k^{(2)})\omega L^{-1}(a_i^{[2]}b_i^{[2]} u^{(1)}) \\
&= L^{-1}(y_j^{(3)} y_k^{(2)})\omega L^{-1}(a_i^{[2]}b_i^{[2]} u^{(1)}) \\
&= y_j^{(3)} y_k^{(2)} L^{-1}(a_i^{[2]}b_i^{[2]} u^{(1)})
\end{align*}
as $L^{-1}(y_j^{(3)} y_k^{(2)})\in \omega A^1$ and $(1-\omega L^{-1})(a_i^{[2]}b_i^{[2]} u^{(1)}) \in PA^3$. On the other hand, by (\ref{necessary equation for formal}),
$$
L^{-1}(y_j^{(3)} b_i^{[2]})(1-\omega L^{-1})(y_k^{(2)} a_i^{[2]}) u^{(1)} = y_j^{(3)} L^{-1}\big( b_i^{[2]}(1-\omega L^{-1})(y_k^{(2)} a_i^{[2]}) \big) u^{(1)}.
$$
Its left-hand side is
\begin{align*}
& \quad L^{-1}(y_j^{(3)} b_i^{[2]}) y_k^{(2)} a_i^{[2]} u^{(1)} -L^{-1}(y_j^{(3)} b_i^{[2]})\omega L^{-1} (y_k^{(2)} a_i^{[2]}) u^{(1)} \\
&= -a_i^{[2]} L^{-1}(b_i^{[2]} y_j^{(3)}) y_k^{(2)} u^{(1)} -y_j^{(3)} b_i^{[2]} L^{-1} (y_k^{(2)} a_i^{[2]}) u^{(1)},
\end{align*}
and right-hand side is
\begin{align*}
& \quad y_j^{(3)} L^{-1}( b_i^{[2]} y_k^{(2)} a_i^{[2]} )u^{(1)} -y_j^{(3)} L^{-1}\big( b_i^{[2]} \omega L^{-1} (y_k^{(2)} a_i^{[2]}) \big) u^{(1)} \\
&= -y_j^{(3)} L^{-1}( y_k^{(2)} a_i^{[2]} b_i^{[2]} )u^{(1)} -y_j^{(3)} b_i^{[2]} L^{-1} (y_k^{(2)} a_i^{[2]}) u^{(1)}.
\end{align*}
So we get
$$
a_i^{[2]} L^{-1}(b_i^{[2]} y_j^{(3)}) y_k^{(2)} u^{(1)} = y_j^{(3)} L^{-1}( y_k^{(2)} a_i^{[2]} b_i^{[2]} )u^{(1)},
$$
and then (\ref{equation for p=3}) holds. \\[0.1in]

\noindent\textbf{5.3.4)} If $[z^{(2)}]=[y_j^{(2)}],[w^{(3)}]=[y_k^{(3)}],[v^{(2)}]=[x_i^{(2)}]$, we can get (\ref{equation for p=3}) similarly as the above case, because $m_2$ and $f_2$ are graded commutative. \\[0.1in]

\noindent\textbf{5.3.5)} If $[z^{(2)}]=[y_j^{(2)}],[w^{(3)}]=[y_k^{(3)}],[v^{(2)}]=[y_l^{(2)}]$, we have
\begin{align*}
&\quad m_2(f_1\otimes f_2)\big([z^{(2)}],[w^{(3)}],[v^{(2)}]\big) \\
&= y_j^{(2)} \Bigg( -\theta L^{-1}(y_k^{(3)} y_l^{(2)})-\sum_i L^{-3}\Big( -y_k^{(3)} f_2\big([y_l^{(2)}],[x_i^{(3)}]\big) \Big)(x_i^{(3)})^* \Bigg) \\
&= -\theta y_j^{(2)} L^{-1}(y_k^{(3)} y_l^{(2)}) -y_j^{(2)} \sum_i L^{-3}\Big( -y_k^{(3)} f_2\big([y_l^{(2)}],[x_i^{(3)}]\big) \Big)(x_i^{(3)})^*. \\
\end{align*}
On the other hand,
\begin{align*}
-f_2\big([z^{(2)}],[w^{(3)}]\big)f_1\big([v^{(2)}]\big) = \theta L^{-1}(y_j^{(2)} y_k^{(3)})y_l^{(2)}+\sum_i L^{-3}\Big( -f_2\big([x_i^{(3)}],[y_j^{(2)}]\big)y_k^{(3)} \Big)(x_i^{(3)})^*y_l^{(2)}.
\end{align*}
By (\ref{necessary equation for formal}) $y_j^{(2)} L^{-1}(y_k^{(3)} y_l^{(2)})=L^{-1}(y_j^{(2)} y_k^{(3)})y_l^{(2)}$. So the sum is
$$
-y_j^{(2)} \sum_i L^{-3}\Big( -y_k^{(3)} f_2\big([y_l^{(2)}],[x_i^{(3)}]\big) \Big)(x_i^{(3)})^*+\sum_i L^{-3}\Big( -f_2\big([x_i^{(3)}],[y_j^{(2)}]\big)y_k^{(3)} \Big)(x_i^{(3)})^*y_l^{(2)}.
$$
We will show it is 0 by multiplying $u^{(1)}\in A^1$. Same as (\ref{x3*.x3}), we have
\begin{align*}
y_j^{(2)} \sum_i L^{-3}\Big( -y_k^{(3)} f_2\big([y_l^{(2)}],[x_i^{(3)}]\big) \Big)(x_i^{(3)})^* u^{(1)} &= y_k^{(3)} f_2\big([y_l^{(2)}],[y_j^{(2)}u^{(1)}]\big) \\
&= y_k^{(3)} \theta\Big( y_l^{(2)} L^{-1}(y_j^{(2)} u^{(1)})-L^{-1}(y_l^{(2)} y_j^{(2)})u^{(1)} \Big),
\end{align*}
and
\begin{align*}
\sum_i L^{-3}\Big( -f_2\big([x_i^{(3)}],[y_j^{(2)}]\big)y_k^{(3)} \Big)(x_i^{(3)})^*y_l^{(2)} u^{(1)} &= f_2\big([y_l^{(2)} u^{(1)}],[y_j^{(2)}]\big)y_k^{(3)} \\
&= \theta\Big( L^{-1}(u^{(1)} y_l^{(2)})y_j^{(2)}-u^{(1)} L^{-1}(y_l^{(2)} y_j^{(2)}) \Big)y_k^{(3)}.
\end{align*}
Observe that
\begin{align*}
\theta y_j^{(2)} L^{-1}(y_k^{(3)} y_l^{(2)})u^{(1)} &= \theta L^{-1}(y_k^{(3)} y_l^{(2)})y_j^{(2)} u^{(1)} \\
&= \theta L^{-1}(y_k^{(3)} y_l^{(2)})(1-\omega L^{-1})(y_j^{(2)} u^{(1)})+\theta L^{-1}(y_k^{(3)} y_l^{(2)})\omega L^{-1}(y_j^{(2)} u^{(1)}) \\
&= \theta L^{-1}(y_k^{(3)} y_l^{(2)})\omega L^{-1}(y_j^{(2)} u^{(1)}) \\
&= -y_k^{(3)} \theta y_l^{(2)} L^{-1}(y_j^{(2)} u^{(1)}).
\end{align*}
Similarly we have
$$
\theta L^{-1}(y_j^{(2)} y_k^{(3)})y_l^{(2)} u^{(1)} = \theta y_j^{(2)} y_k^{(3)} L^{-1}(y_l^{(2)} u^{(1)}) = -\theta L^{-1}(u^{(1)} y_l^{(2)})y_j^{(2)} y_k^{(3)}.
$$
By (\ref{necessary equation for formal}) $y_j^{(2)} L^{-1}(y_k^{(3)} y_l^{(2)})u^{(1)}=L^{-1}(y_j^{(2)} y_k^{(3)})y_l^{(2)} u^{(1)}$. Then we have
$$
y_k^{(3)} \theta y_l^{(2)} L^{-1}(y_j^{(2)} u^{(1)}) = \theta L^{-1}(u^{(1)} y_l^{(2)})y_j^{(2)} y_k^{(3)}.
$$
Therefore, the sum is 0 and (\ref{equation for p=3}) holds. \\[0.1in]

\noindent\textbf{6)} When $r+s+t=8$, we need to check the cases $(r,s,t)=(3,2,3)$ and $(3,3,2)$.

\noindent\textbf{6.1)} For $(r,s,t)=(3,2,3)$, $[z^{(3)} w^{(2)}]=[w^{(2)} v^{(3)}]=0$. We will also go through the cases that whether $[z^{(3)}],[w^{(2)}],[v^{(3)}]$ are reducible.

\noindent\textbf{6.1.1)} If $[w^{(2)}]=[x_i^{(2)}]$, we have
$$
m_2(f_1\otimes f_2)\big([z^{(3)}],[w^{(2)}],[v^{(3)}]\big) = -z^{(3)} \theta\Big( L^{-1}(a_i^{[2]}b_i^{[2]})v^{(3)}-a_i^{[2]}L^{-1}(b_i^{[2]}v^{(3)}) \Big),
$$
and
$$
-f_2\big([z^{(3)}],[w^{(2)}]\big)f_1\big([v^{(3)}]\big) = -\theta\Big( z^{(3)} L^{-1}(a_i^{[2]}b_i^{[2]})v^{(3)}-L^{-1}(z^{(3)} a_i^{[2]})b_i^{[2]}v^{(3)} \Big).
$$
As $\omega\circ L^{-1}:A^4\to A^4$ is an isomorphism,
$$
z^{(3)} \theta a_i^{[2]}L^{-1}(b_i^{[2]}v^{(3)}) = -\theta\omega L^{-1}(z^{(3)} a_i^{[2]})L^{-1}(b_i^{[2]}v^{(3)}) = -\theta L^{-1}(z^{(3)} a_i^{[2]})b_i^{[2]}v^{(3)}.
$$
So (\ref{equation for p=3}) holds. \\[0.1in]

\noindent\textbf{6.1.2)} If $[z^{(3)}]=[x_i^{(3)}],[w^{(2)}]=[y_j^{(2)}],[v^{(3)}]=[x_k^{(3)}]$, we have
\begin{align*}
&\quad m_2(f_1\otimes f_2)\big([z^{(3)}],[w^{(2)}],[v^{(3)}]\big) \\
&= -(1-\omega L^{-1})(a_i^{[3]}b_i^{[3]}) \theta\Big( w^{(2)} L^{-1}(a_k^{[3]}b_k^{[3]})-L^{-1}(w^{(2)} a_k^{[3]})b_k^{[3]} \Big) \\
&= \theta\Big( a_i^{[3]}b_i^{[3]} w^{(2)} L^{-1}(a_k^{[3]}b_k^{[3]}) -a_i^{[3]}b_i^{[3]} L^{-1}(w^{(2)} a_k^{[3]})b_k^{[3]} \\
&\quad\quad -\omega L^{-1}(a_i^{[3]}b_i^{[3]}) w^{(2)} L^{-1}(a_k^{[3]}b_k^{[3]}) +\omega L^{-1}(a_i^{[3]}b_i^{[3]}) L^{-1}(w^{(2)} a_k^{[3]})b_k^{[3]} \Big),
\end{align*}
and
\begin{align*}
&\quad -f_2\big([z^{(3)}],[w^{(2)}]\big)f_1\big([v^{(3)}]\big) \\
&= -\theta\Big( L^{-1}(b_i^{[3]}a_i^{[3]})w^{(2)}-b_i^{[3]}L^{-1}(a_i^{[3]}w^{(2)}) \Big) (1-\omega L^{-1})(a_k^{[3]}b_k^{[3]}) \\
&= -\theta\Big( L^{-1}(b_i^{[3]}a_i^{[3]})w^{(2)} a_k^{[3]}b_k^{[3]} -L^{-1}(b_i^{[3]}a_i^{[3]})w^{(2)} \omega L^{-1}(a_k^{[3]}b_k^{[3]}) \\
&\quad\quad\quad -b_i^{[3]}L^{-1}(a_i^{[3]}w^{(2)}) a_k^{[3]}b_k^{[3]} +b_i^{[3]}L^{-1}(a_i^{[3]}w^{(2)}) \omega L^{-1}(a_k^{[3]}b_k^{[3]}) \Big).
\end{align*}

Since $\omega\circ L^{-1}:A^4\to A^4$ is an isomorphism, we have
$$
a_i^{[3]}b_i^{[3]} w^{(2)} L^{-1}(a_k^{[3]}b_k^{[3]}) = b_i^{[3]}L^{-1}(a_i^{[3]}w^{(2)}) \omega L^{-1}(a_k^{[3]}b_k^{[3]}),
$$
and
$$
\omega L^{-1}(a_i^{[3]}b_i^{[3]}) L^{-1}(w^{(2)} a_k^{[3]})b_k^{[3]} = L^{-1}(b_i^{[3]}a_i^{[3]})w^{(2)} a_k^{[3]}b_k^{[3]}.
$$
Also by (\ref{necessary equation for formal}),
$$
a_i^{[3]}b_i^{[3]} L^{-1}(w^{(2)} a_k^{[3]})b_k^{[3]} = b_i^{[3]}a_i^{[3]} L^{-1}(w^{(2)} a_k^{[3]})b_k^{[3]} = b_i^{[3]} L^{-1}(a_i^{[3]} w^{(2)}) a_k^{[3]}b_k^{[3]}. 
$$
So the sum of all these terms are 0, and we have (\ref{equation for p=3}). \\[0.1in]

\noindent\textbf{6.1.3)} If $[z^{(3)}]=[x_i^{(3)}],[w^{(2)}]=[y_j^{(2)}],[v^{(3)}]=[y_k^{(3)}]$, we have
\begin{align*}
&\quad m_2(f_1\otimes f_2)\big([z^{(3)}],[w^{(2)}],[v^{(3)}]\big) \\
&= -x_i^{(3)} \Bigg( -\theta L^{-1}(y_j^{(2)} y_k^{(3)})-\sum_l L^{-3}\Big( -f_2\big([x_l^{(3)}],[y_j^{(2)}]\big) y_k^{(3)} \Big)(x_l^{(3)})^* \Bigg) \\
&= -\theta x_i^{(3)} L^{-1}(y_j^{(2)} y_k^{(3)}) +x_i^{(3)}\sum_l L^{-3}\Big( -f_2\big([x_l^{(3)}],[y_j^{(2)}]\big) y_k^{(3)} \Big)(x_l^{(3)})^* \\
&= -\sum_l L^{-3}\Big( -f_2\big([x_l^{(3)}],[y_j^{(2)}]\big) y_k^{(3)} \Big)x_i^{(3)}(x_l^{(3)})^* \\
&= -L^{-3}\Big( -f_2\big([x_i^{(3)}],[y_j^{(2)}]\big) y_k^{(3)} \Big) \omega^3 \\
&= f_2\big([x_i^{(3)}],[y_j^{(2)}]\big) y_k^{(3)} 
\end{align*}
as $L^{-1}(y_j^{(2)} y_k^{(3)})\in \omega A^1$ and $x_i^{(3)}\in PA^3$. On the other hand, 
$$
-f_2\big([z^{(3)}],[w^{(2)}]\big)f_1\big([v^{(3)}]\big)=-f_2\big([x_i^{(3)}],[y_j^{(2)}]\big) y_k^{(3)}.
$$
So their sum is 0 and (\ref{equation for p=3}) holds. \\[0.1in]

\noindent\textbf{6.1.4)} If $[z^{(3)}]=[y_j^{(3)}],[w^{(2)}]=[y_k^{(2)}],[v^{(3)}]=[x_i^{(3)}]$, we can get (\ref{equation for p=3}) by the above case since $m_2,f_2$ are graded commutative. \\[0.1in]

\noindent\textbf{6.1.5)} If $[z^{(3)}]=[y_j^{(3)}],[w^{(2)}]=[y_k^{(2)}],[v^{(3)}]=[y_l^{(3)}]$, we have
\begin{equation}\label{m3(all irreducible)}
\begin{split}
& \quad m_2(f_1\otimes f_2)\big([z^{(3)}],[w^{(2)}],[v^{(3)}]\big) \\
&= -y_j^{(3)} \Bigg( -\theta L^{-1}(y_k^{(2)} y_l^{(3)})-\sum_i L^{-3}\Big( -f_2\big([x_i^{(3)}],[y_k^{(2)}]\big) y_l^{(3)} \Big)(x_i^{(3)})^* \Bigg).
\end{split}
\end{equation}
Both terms are 0. As $y_j^{(3)}\in PA^3$ and $L^{-1}(y_k^{(2)} y_l^{(3)})\in \omega A^1$, their product are 0. Also $y_j^{(3)}(x_i^{(3)})^*=0$ for any $i$. For the same reason,
$$
-f_2\big([z^{(3)}],[w^{(2)}]\big)f_1\big([v^{(3)}]\big)=0.
$$
Therefore, (\ref{equation for p=3}) holds. \\[0.1in]

\noindent\textbf{6.2)} The last non-trivial case is $(r,s,t)=(3,3,2)$. In this case both $[z^{(3)} w^{(3)}]$ and $[w^{(3)} v^{(2)}]$ are 0. We also have to go through the cases that whether $[z^{(3)}],[w^{(3)}],[v^{(2)}]$ are reducible. \\[0.1in]

\noindent\textbf{6.2.1)} If $[w^{(3)}]=[x_i^{(3)}]$, we have
$$
m_2(f_1\otimes f_2)\big([z^{(3)}],[w^{(3)}],[v^{(2)}]\big) = -z^{(3)} \theta\Big( L^{-1}(b_i^{[3]}a_i^{[3]})v^{(2)}-b_i^{[3]}L^{-1}(a_i^{[3]}v^{(2)}) \Big),
$$
and
$$
-f_2\big([z^{(3)}],[w^{(2)}]\big)f_1\big([v^{(3)}]\big) = -\theta\Big( z^{(3)} L^{-1}(b_i^{[3]}a_i^{[3]})v^{(2)}-L^{-1}(z^{(3)} b_i^{[3]})a_i^{[3]}v^{(2)} \Big).
$$
Observe that
$$
z^{(3)}\theta b_i^{[3]}L^{-1}(a_i^{[3]}v^{(2)}) = -\theta\omega L^{-1}(z^{(3)} b_i^{[3]}) L^{-1}(a_i^{[3]}v^{(2)}) = -\theta L^{-1}(z^{(3)} b_i^{[3]})a_i^{[3]}v^{(2)}.
$$
So the sum all all terms is 0 and (\ref{equation for p=3}) holds. \\[0.1in]

\noindent\textbf{6.2.2)} If $[z^{(3)}]=[x_i^{(3)}],[w^{(3)}]=[y_j^{(3)}],[v^{(2)}]=[x_k^{(2)}]$, we have
\begin{align*}
&\quad m_2(f_1\otimes f_2)\big([z^{(3)}],[w^{(3)}],[v^{(2)}]\big) \\
&= -(1-\omega L^{-1})(a_i^{[3]}b_i^{[3]}) \theta\Big( w^{(3)} L^{-1}(a_k^{[2]}b_k^{[2]})-L^{-1}(w^{(3)} a_k^{[2]})b_k^{[2]} \Big) \\
&= \theta\Big( a_i^{[3]}b_i^{[3]} w^{(3)} L^{-1}(a_k^{[2]}b_k^{[2]}) -a_i^{[3]}b_i^{[3]} L^{-1}(w^{(3)} a_k^{[2]})b_k^{[2]} \\
&\quad\quad -\omega L^{-1}(a_i^{[3]}b_i^{[3]}) w^{(3)} L^{-1}(a_k^{[2]}b_k^{[2]}) +\omega L^{-1}(a_i^{[3]}b_i^{[3]}) L^{-1}(w^{(3)} a_k^{[2]})b_k^{[2]} \Big),
\end{align*}
and
\begin{align*}
&\quad -f_2\big([z^{(3)}],[w^{(3)}]\big)f_1\big([v^{(2)}]\big) \\
&= -\theta\Big( L^{-1}(a_i^{[3]}b_i^{[3]})w^{(3)}-a_i^{[3]}L^{-1}(b_i^{[3]}w^{(3)}) \Big) (1-\omega L^{-1})(a_k^{[2]}b_k^{[2]}) \\
&= -\theta\Big( L^{-1}(a_i^{[3]}b_i^{[3]})w^{(3)} a_k^{[2]}b_k^{[2]} -L^{-1}(a_i^{[3]}b_i^{[3]})w^{(3)} \omega L^{-1}(a_k^{[2]}b_k^{[2]}) \\
&\quad\quad\quad -a_i^{[3]}L^{-1}(b_i^{[3]}w^{(3)}) a_k^{[2]}b_k^{[2]} +a_i^{[3]}L^{-1}(b_i^{[3]}w^{(3)}) \omega L^{-1}(a_k^{[2]}b_k^{[2]}) \Big).
\end{align*}
Since $\omega\circ L^{-1}:A^4\to A^4$ is an isomorphism, we have
$$
a_i^{[3]}b_i^{[3]} w^{(3)} L^{-1}(a_k^{[2]}b_k^{[2]}) = a_i^{[3]}L^{-1}(b_i^{[3]}w^{(3)}) \omega L^{-1}(a_k^{[2]}b_k^{[2]}),
$$
and
$$
\omega L^{-1}(a_i^{[3]}b_i^{[3]}) L^{-1}(w^{(3)} a_k^{[2]})b_k^{[2]} = L^{-1}(a_i^{[3]}b_i^{[3]})w^{(3)} a_k^{[2]}b_k^{[2]}.
$$
Also by (\ref{necessary equation for formal}),
$$
a_i^{[3]}b_i^{[3]} L^{-1}(w^{(3)} a_k^{[2]})b_k^{[2]} = a_i^{[3]}L^{-1}(b_i^{[3]}w^{(3)}) a_k^{[2]}b_k^{[2]}.
$$
So the sum of all terms is 0, and we have (\ref{equation for p=3}). \\[0.1in]

\noindent\textbf{6.2.3)} If $[z^{(3)}]=[x_i^{(3)}],[w^{(3)}]=[y_j^{(3)}],[v^{(2)}]=[y_k^{(2)}]$, we have
\begin{align*}
&\quad m_2(f_1\otimes f_2)\big([z^{(3)}],[w^{(3)}],[v^{(2)}]\big) \\
&= -x_i^{(3)} \Bigg( -\theta L^{-1}(y_j^{(3)} y_k^{(2)})-\sum_l L^{-3}\Big( -y_j^{(3)} f_2\big([y_k^{(2)}],[x_l^{(3)}]\big) \Big)(x_l^{(3)})^* \Bigg) \\
&= -\theta x_i^{(3)} L^{-1}(y_j^{(3)} y_k^{(2)}) +x_i^{(3)}\sum_l L^{-3}\Big( -y_j^{(3)} f_2\big([y_k^{(2)}],[x_l^{(3)}]\big) \Big)(x_l^{(3)})^* \\
&= -\sum_l L^{-3}\Big( -y_j^{(3)} f_2\big([y_k^{(2)}],[x_l^{(3)}]\big) \Big)x_i^{(3)}(x_l^{(3)})^* \\
&= -L^{-3}\Big( -y_j^{(3)} f_2\big([y_k^{(2)}],[x_i^{(3)}]\big) \Big)\omega^3 \\
&= y_j^{(3)} f_2\big([y_k^{(2)}],[x_i^{(3)}]\big) \\
&= f_1\big( [y_j^{(3)}] \big) f_2\big([x_i^{(3)}],[y_k^{(2)}]\big)
\end{align*}
as $L^{-1}(y_j^{(3)} y_k^{(2)})\in \omega A^1$ and $x_i^{(3)}\in PA^3$. On the other hand,
$$
-f_2\big([z^{(3)}],[w^{(3)}]\big)f_1\big([v^{(2)}]\big) = -f_2 \big([x_i^{(3)}],[y_j^{(3)}]\big)f_1\big([y_k^{(2)}]\big) = f_2 \big([y_j^{(3)}],[x_i^{(3)}]\big)f_1\big([y_k^{(2)}]\big).
$$
In case 6.2.1) we have proved that
$$
\Big( m_2(f_1\otimes f_2-f_2\otimes f_1)-f_2(m_2\otimes 1-1\otimes m_2) \Big) ([y_j^{(3)}],[x_i^{(3)}],[y_k^{(2)}]) = 0.
$$
So their sum is 0 and (\ref{equation for p=3}) holds. \\[0.1in]

\noindent\textbf{6.2.4)} If $[z^{(3)}]=[y_j^{(3)}],[w^{(2)}]=[y_k^{(3)}],[v^{(3)}]=[x_i^{(2)}]$, we have
\begin{align*}
m_2(f_1\otimes f_2)\big([z^{(3)}],[w^{(3)}],[v^{(2)}]\big) = -y_j^{(3)} f_2\big([y_k^{(3)}],[x_i^{(2)}]\big).
\end{align*}
On the other hand,
\begin{align*}
-f_2\big([z^{(3)}],[w^{(3)}]\big)f_1\big([v^{(2)}]\big) &= \theta L^{-1}(y_j^{(3)} y_k^{(3)})x_i^{(2)}+\sum_l L^{-3}\Big( y_j^{(3)} f_2\big([y_k^{(3)}],[x_l^{(2)}]\big) \Big)(x_l^{(2)})^*x_i^{(2)} \\
&= L^{-3}\Big( y_j^{(3)} f_2\big([y_k^{(3)}],[x_i^{(2)}]\big) \Big)(x_i^{(2)})^*x_i^{(2)} \\
&= y_j^{(3)} f_2\big([y_k^{(3)}],[x_i^{(2)}]\big).
\end{align*}
So their sum is 0 and (\ref{equation for p=3}) holds. \\[0.1in]

\noindent\textbf{6.2.5)} If $[z^{(3)}]=[y_j^{(3)}],[w^{(3)}]=[y_k^{(3)}],[v^{(3)}]=[y_l^{(2)}]$, we have
\begin{align*}
m_2(f_1\otimes f_2)\big([z^{(3)}],[w^{(3)}],[v^{(2)}]\big) &= -f_1\big( [y_j^{(3)}] \big) f_2\big([y_k^{(3)}],[y_l^{(2)}]\big) \\
&= -f_1\big( [y_j^{(3)}] \big) f_2\big([y_l^{(2)}],[y_k^{(3)}]\big) \\
&= m_2(f_1\otimes f_2)\big( [y_j^{(3)}],[y_l^{(2)}],[y_k^{(3)}]\big).
\end{align*}
It vanishes because we have proved that both sides of equation (\ref{m3(all irreducible)}) are 0. On the other hand
\begin{align*}
-f_2\big([z^{(3)}],[w^{(3)}]\big)f_1\big([v^{(2)}]\big) = \theta L^{-1}(y_j^{(3)} y_k^{(3)})y_l^{(2)}+\sum_i L^{-3}\Big( y_j^{(3)} f_2\big([y_k^{(3)}],[x_i^{(2)}]\big) \Big)(x_i^{(2)})^*y_l^{(2)}.
\end{align*}
Similar as the discussion of (\ref{m3(all irreducible)}), since $L^{-1}(y_j^{(3)} y_k^{(3)})\in \omega^2 A^0$ and $y_l^{(2)}\in PA^2$, their product is 0. Also $(x_i^{(2)})^*y_l^{(2)}=0$. So we get (\ref{equation for p=3}). \\[0.1in]

\noindent\textbf{7)} When $r+s+t=9$, the left-hand side of (\ref{equation for p=3}) has degree 8. So it must be 0. \\[0.1in]

Now we have went through all cases that $r,s,t\leq 3$. Then we consider the other cases. Let $\alpha\in H^{(r)}\big(A[\theta]\big),\beta\in H^{(s)}\big(A[\theta]\big)$ and $\gamma\in H^{(t)}\big(A[\theta]\big)$. We want to verify that
\begin{align}\label{equation for p=3 general}
\Big( m_2(f_1\otimes f_2-f_2\otimes f_1)-f_2(m_2\otimes 1-1\otimes m_2) \Big)(\alpha,\beta,\gamma) = 0.
\end{align}

\noindent\textbf{8)}
When $s\geq 4$, by the definition of $f_2$, $f_2(\alpha,\beta),f_2(\beta,\gamma),f_2(\alpha\beta,\gamma)$ and $f_2(\alpha,\beta\gamma)$ are all 0. So (\ref{equation for p=3 general}) holds. \\[0.1in]

\noindent\textbf{9)}
When $r\geq 4$, $f_2(\alpha,\beta),f_2(\alpha\beta,\gamma)$ and $f_2(\alpha,\beta\gamma)$ are 0. The only non-trivial term is $m_2(f_1\otimes f_2)(\alpha,\beta,\gamma)$. But by definition both $f_1(\alpha)$ and $f_2(\beta,\gamma)$ are in $\theta A$. So their product is 0. Then we get (\ref{equation for p=3 general}).

Since $m_2$ and $f_2$ are graded commutative, (\ref{equation for p=3 general}) also holds when $t\geq 4$. Therefore, we have verified all the cases of $p=3$. \\[0.1in]

\noindent\textbf{Case $p=4$ and higher}.

When $p=4$, the only non-trivial term of 
\begin{align}\label{equation for p=4}
\sum_{i_1+\ldots+i_r=p}(-1)^sm_r(f_{i_1}\otimes f_{i_2}\otimes\ldots\otimes f_{i_r})-\sum_{r+s+t=p}(-1)^{r+st}f_{r+t+1}(\textbf{1}^{\otimes r}\otimes m_s\otimes\textbf{1}^{\otimes r})
\end{align}
is $m_2(f_2\otimes f_2)$. By the definition of $f_2$, it is always in $\theta A$. As the product of two elements in $\theta A$ is 0, (\ref{equation for p=4}) vanishes.

When $p\geq 5$, we claim that every term of (\ref{equation for p=4}) is 0. For the term $m_r(f_{i_1}\otimes f_{i_2}\otimes\ldots\otimes f_{i_r})$, if it is non-trivial then $r\leq 2$. In this case at least one of $i_1$ and $i_2$ is greater than or equal to 3. Hence, $f_{i_1}=0$ or $f_{i_2}=0$.

For the term $f_{r+t+1}(\textbf{1}^{\otimes r}\otimes m_s\otimes\textbf{1}^{\otimes r})$, if it is non-trivial then $r+t+1\leq 2$. In this case $s\geq 4$. So $m_s=0$.

Therefore, we have verified that $f:H^*\big(A[\theta]\big)\to A[\theta]$ is indeed an $A_\infty$-quasi-isomorphism.

This completes the proof of Theorem \ref{condition equivalent to formal}.
\end{proof}

\vskip 1cm
\noindent
{Yau Mathematical Sciences Center, Tsinghua University\\
Hai Dian District, Beijing, China}\\
{\it Email address:}~{\tt jiaweiz1990@mail.tsinghua.edu.cn}
\end{document}